\pdfoutput=1
\documentclass[11pt,a4paper]{article}
\usepackage[T1]{fontenc}
\usepackage[utf8]{inputenc}
\usepackage{lmodern}
\usepackage{amsmath,amssymb,amsthm}
\usepackage{graphicx}
\usepackage{booktabs}
\usepackage{longtable}
\usepackage{microtype}
\usepackage[margin=2.6cm]{geometry}
\usepackage{xcolor}
\usepackage{tikz}
\usetikzlibrary{calc,intersections}
\usepackage{listings}
\usepackage[colorlinks=true,linkcolor=black,citecolor=black,urlcolor=blue!50!black]{hyperref}
\usepackage{caption}
\newtheorem{theorem}{Theorem}
\newtheorem{proposition}[theorem]{Proposition}

\newtheorem{lemma}[theorem]{Lemma}
\theoremstyle{definition}
\newtheorem{definition}[theorem]{Definition}
\newtheorem{result}[theorem]{Computational result}
\newtheorem{example}[theorem]{Example}
\theoremstyle{remark}

\newcommand{\R}{\mathbb{R}}
\newcommand{\Z}{\mathbb{Z}}
\newcommand{\Q}{\mathbb{Q}}
\newcommand{\Zphi}{\mathbb{Z}[\varphi]}
\newcommand{\Qphi}{\mathbb{Q}(\sqrt5)}
\newcommand{\Ico}{\mathbb{I}}
\newcommand{\OK}{\mathcal{O}_K}
\newcommand{\M}{M}
\newcommand{\MV}{M_V}
\newcommand{\Nm}{\operatorname{N}}
\newcommand{\Tr}{\operatorname{Tr}}
\newcommand{\vol}{\operatorname{vol}}
\newcommand{\conv}{\operatorname{conv}}
\newcommand{\sgn}{\operatorname{sgn}}
\newcommand{\Hfour}{H_4}

\newcommand{\NSTAR}{2\,455\,201}
\newcommand{\NSTARCLASSES}{338}
\newcommand{\NSTARSHAPES}{336}
\newcommand{\NSIMPLEXCLASSES}{66}
\newcommand{\NSIMPLEXCELLS}{732\,000}
\newcommand{\NHERED}{8\,161}
\newcommand{\NHEREDCLASSES}{8}
\newcommand{\ORBITHIST}{$1{:}1$, $120{:}4$, $600{:}6$, $720{:}8$, $1200{:}9$, $1440{:}9$, $2400{:}22$, $3600{:}56$, $7200{:}145$, $14400{:}78$}

\IfFileExists{numbers_ball.tex}{\newcommand{\NBALLFOUR}{2\,492\,641}
\newcommand{\NBALLFOURCLASSES}{343}

\newcommand{\NBALLFOUREXTRA}{37\,440}
\newcommand{\NBALLFOUREXTRACLASSES}{5}
}{\newcommand{\NBALLFOUR}{[pending]}\newcommand{\NBALLFOURCLASSES}{[pending]}\newcommand{\NBALLFOUREXTRA}{[pending]}\newcommand{\NBALLFOUREXTRACLASSES}{[pending]}}
\newcommand{\fig}[2][\linewidth]{\includegraphics[width=#1]{#2}}

\title{The 120-Cell Dissected by the Hyperplanes of Its Final Stellation\\[4pt]\large the golden ring $\Zphi$ as foundation: from Kramer's seven cells to $2\,455\,201$ cells in $338$ orbits}
\author{Miguel E.~P\'erez Ibarra\thanks{Formerly Escuela Superior de F\'isica y
    Matem\'aticas, Instituto Polit\'ecnico Nacional, Mexico City, Mexico.
    E-mail: \texttt{supersluky@gmail.com}.}
  \and Jaziel Flores Rodr\'iguez\thanks{Formerly Escuela Superior de F\'isica y
    Matem\'aticas, Instituto Polit\'ecnico Nacional, Mexico City, Mexico.
    E-mail: \texttt{jf@saptiva.com}.}}
\date{October 2026}

\begin{document}
\maketitle

\begin{abstract}
The five diagonals of a regular pentagon dissect it into eleven pieces and draw a pentagram.  The
three-dimensional version, the regular dodecahedron dissected by the twelve face planes of its
inner core, that is, by the planes of its own stellation, dates back to Kramer's seven
``elementary cells'' of 1982: the great stellated dodecahedron is a union of its pieces.  We define
the \emph{stellation dissection} of a golden regular polytope in any dimension as the dissection of
the convex hull of its final regular stellation by the facet hyperplanes of the core, show that the
core has $\varphi^{-d}$ of the inradius of the hull for $d=2,3,4$, and carry the dissection out for
the first time on the 120-cell.  The foundation of the whole construction is the arithmetic of the
\emph{golden ring} $\Zphi=\{m+n\varphi\}$, the ring of integers of $\Qphi$: the vertices and
hyperplanes of the dodecahedron and the 120-cell live in $\Zphi^d$, every orientation test is the
exact sign of an element of $\Zphi$, decided with integers, and the pieces of a dissection are
named by sign vectors.  We develop this foundation from scratch (conjugation, norm, units, the
primes $\sqrt5$, $2$, $3$, the exact sign test, the lattice $x\mapsto(x,\bar x)$, the vertex
modules $\M\subset\Zphi^3$ and $\MV\subset\Zphi^4$, the icosian ring) and then let it carry the
geometry.  The definition bifurcates in dimension four: the \emph{star cut} at depth $\varphi^{-4}$,
a unit, has $\NSTAR$ convex cells in $\NSTARCLASSES$ orbits of the Coxeter group $\Hfour$, while the
\emph{hereditary cut} at depth $1-\varphi^{-4}=\sqrt5/\varphi^2$, a unit times the ramified prime,
induces the three-dimensional dissection on every dodecahedral facet and has $\NHERED$ cells in
$\NHEREDCLASSES$ orbits.  Three facts are new to this version: all $112$ vertices of the $165$
pieces of the dodecahedron and all vertices of the $\NHERED$ pieces of the hereditary cut lie in
the vertex module of the polytope being cut, whereas the star cut needs the denominators $2$, $3$
and $5$; the seven kinds of pieces of the dodecahedron are recovered exactly as the $I_h$-orbits
of sign vectors, with the stellation layer equal to the number of plus signs; and the depths of the
two four-dimensional cuts differ, arithmetically, by a unit and one factor of the prime $\sqrt5$.  We give
the construction in enough detail to be repeated, catalogue the four-dimensional pieces, and
re-derive the dodecahedron case completely, including why plain floating point fails on it.  As a
coda we apply the same cuts to the circumscribed disc, ball and $4$-ball.  Every number quoted is
recomputed exactly by a short script whose output is reproduced in Appendix~\ref{app:checks}.
\end{abstract}

\normalsize

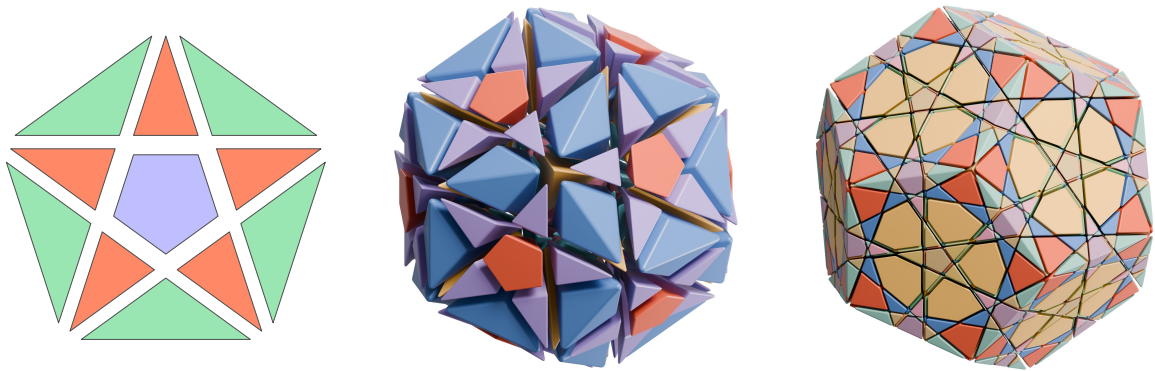
\begin{figure}[t]
  \centering
  \begin{minipage}[c]{0.3\linewidth}\centering
  \begin{tikzpicture}[scale=1.9, line join=round]
    \def\R{1}\def\r{0.381966}
    \foreach \kk in {0,...,4} { \coordinate (v\kk) at ({90+72*\kk}:\R); \coordinate (u\kk) at ({54+72*\kk}:\r); }
    \def\ex{0.25}
    \fill[blue!25!white,draw=black!70,very thin] (u0)--(u1)--(u2)--(u3)--(u4)--cycle;
    \foreach \kk/\kn in {0/1, 1/2, 2/3, 3/4, 4/0} {
      \coordinate (c) at (barycentric cs:v\kk=1,u\kk=1,u\kn=1);
      \fill[red!55!orange!60,draw=black!70,very thin] ($(v\kk)+\ex*(c)$)--($(u\kk)+\ex*(c)$)--($(u\kn)+\ex*(c)$)--cycle;
    }
    \foreach \kk/\kn in {0/1, 1/2, 2/3, 3/4, 4/0} {
      \coordinate (c) at (barycentric cs:v\kk=1,v\kn=1,u\kn=1);
      \fill[green!50!teal!40,draw=black!70,very thin] ($(v\kk)+\ex*(c)$)--($(v\kn)+\ex*(c)$)--($(u\kn)+\ex*(c)$)--cycle;
    }
  \end{tikzpicture}
  \end{minipage}\hfill
  \begin{minipage}[c]{0.33\linewidth}\fig{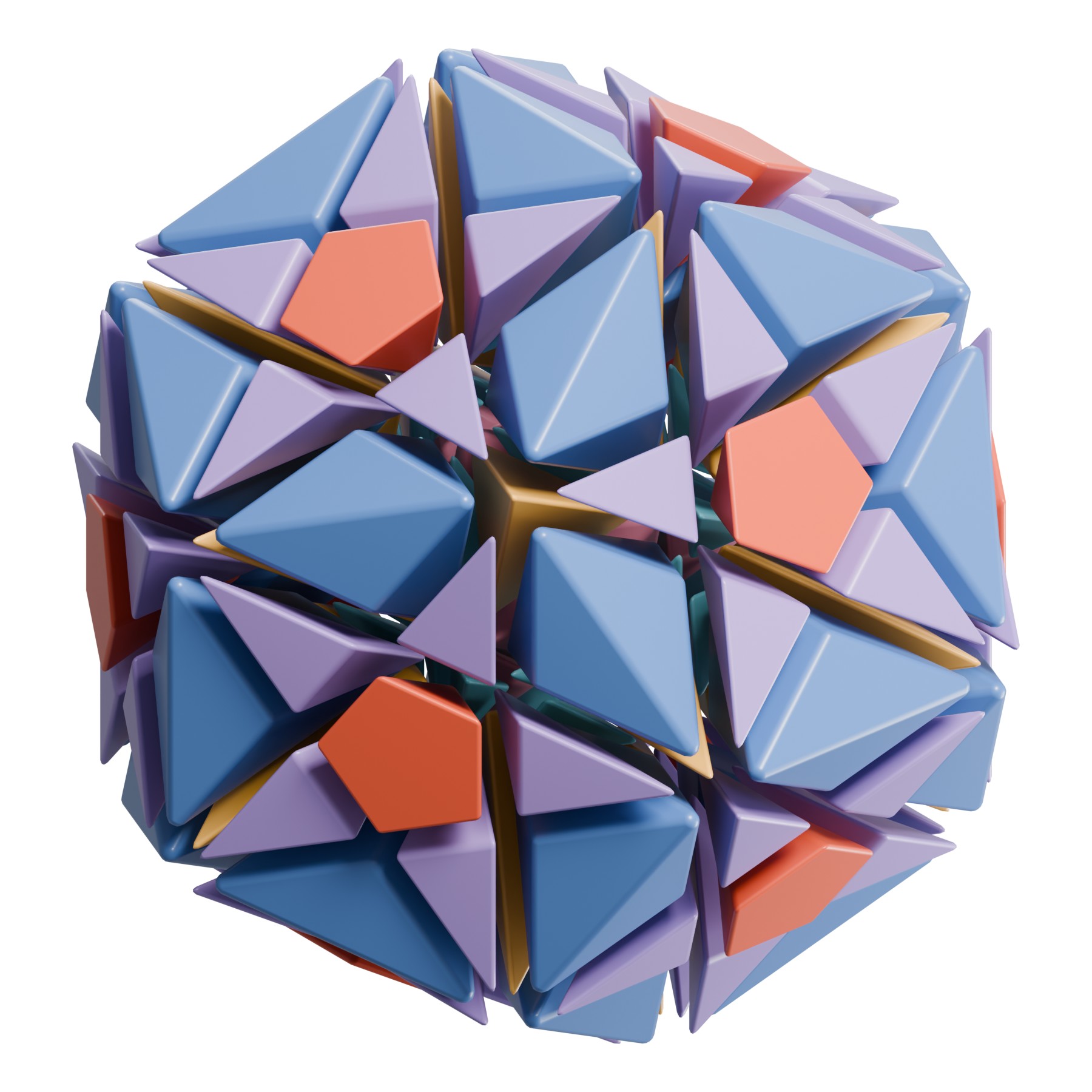}\end{minipage}\hfill
  \begin{minipage}[c]{0.33\linewidth}\fig{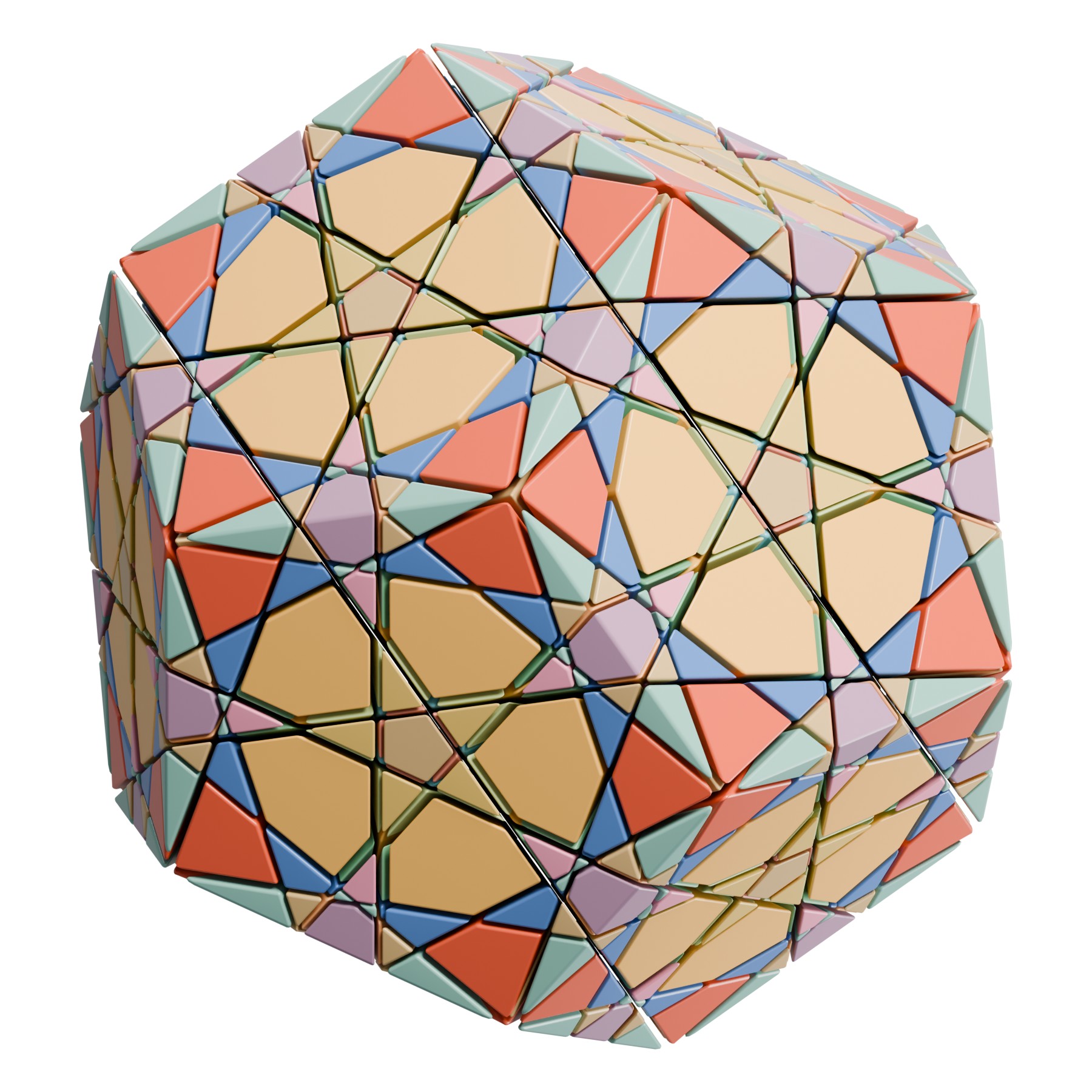}\end{minipage}
  \caption{The stellation dissection in dimensions two, three and four.  Left: the pentagon cut by its
  diagonals ($11$ pieces, exploded).  Middle: the dodecahedron cut by the twelve face planes of its
  inner core, exploded ($165$ pieces, Kramer's seven cells).  Right: the pattern that the $120$
  facet hyperplanes of the inner core of the 120-cell induce on one of its dodecahedral facets,
  exploded ($3\,655$ pieces on the facet, $\NSTAR$ cells in all).}
  \label{fig:teaser}
\end{figure}

\section{Introduction}\label{sec:intro}

Draw the five diagonals of a regular pentagon.  They cut it into eleven pieces: a small
inverted pentagon in the middle, five isosceles triangles that are the points of a pentagram,
and five obtuse triangles between the points and the sides.  The small pentagon is similar to
the large one with ratio $\varphi^{-2}$, where $\varphi=(1+\sqrt5)/2$, and the diagonals are the
lines of its sides.  Said the
other way round: the pentagram is the stellation of the small pentagon, the large pentagon is
its convex hull, and the dissection is by the side lines of the \emph{core}: it is the
stellation diagram of the small pentagon, completed to the hull of the star.  We call this the
\emph{stellation dissection} of the large pentagon.  This is the oldest picture in the subject; the Pythagoreans used the pentagram
as a sign of recognition, and the golden section enters mathematics through exactly this
figure \cite{Huntley1970,Coxeter1969}.

The three-dimensional analogue replaces the pentagon by the regular dodecahedron $B$ and the
diagonals by the twelve face planes of the dodecahedron $C=\varphi^{-3}B$, which is the
stellation core of the great stellated dodecahedron with the vertices of $B$
\cite{Coxeter1973,Wenninger1971}.  The planes cut $B$ into $165$ convex pieces of seven
kinds.  To our knowledge this cut was first made by Peter Kramer in 1982, in the first attempt at
a three-dimensional analogue of the Penrose tilings: he found seven convex ``elementary cells''
that pack a regular dodecahedron and that can be re-packed, scaled by $\varphi^3$, by copies of
themselves \cite{Kramer1982,Kramer2011}; Mosseri and Sadoc reduced his set to four tiles the
same year \cite{MosseriSadoc1982}.  The same pieces reappear in the mid-1990s as the blocks of
Goodman-Strauss's \emph{Dodecafoam}, a substitution tiling in which every block is divided into
smaller blocks and only the dodecahedral ones are drawn
\cite{GoodmanStrauss1999,GoodmanStrauss1999b,GoodmanStraussKrech1995,GoodmanStrauss2008}.  His
description of the construction is the one we start from: ``cutting it apart with planes that
crisscross through its vertices, chopping it into variously shaped cells.  The innermost cell is
itself a miniature dodecahedron, three golden powers smaller than the first.  Nested between the
two, in layers, we see the three stellated dodecahedra'' \cite{GoodmanStraussPoster}.  The 1995
image pages add that the symmetry behind the construction ``generalizes to dimension 4 (but no
higher)'' \cite{GoodmanStraussKrech1995}.

This article carries the construction into dimension four, where the dodecahedron becomes the
120-cell $\{5,3,3\}$, with $600$ vertices and $120$ dodecahedral facets, and the great stellated
dodecahedron becomes the great grand stellated 120-cell $\{5/2,3,3\}$, the only regular star
polychoron whose convex hull is the 120-cell \cite{Coxeter1973,Hess1883}.  We searched the arXiv,
the journal indexes and the polytope community's web resources for a dissection of the 120-cell by
its facet hyperplanes and found none.  The nearest neighbours are the stellation tools of the
software Stella4D \cite{Webb}, which can display cells of the stellation diagram of a polychoron,
and the data collected on the regular star polychora by Bowers and others \cite{Bowers,PolytopeWiki},
which concern the surfaces of the stars (densities, external pieces) rather than the interior
cells.  Neither contains the counts, the classification or the method below.

\paragraph{The ring first.}  The reason the computation can be done at all, and done exactly, is
arithmetic.  Every number that occurs, every coordinate of a vertex, every coefficient of a
hyperplane, every depth at which the arrangement changes, is an element of the quadratic field
$\Qphi=\{a+b\sqrt5:a,b\in\Q\}$, and after one rescaling almost every one of them is an element of
its ring of integers, the \emph{golden ring}
\[
  \Zphi=\{m+n\varphi:\ m,n\in\Z\},\qquad \varphi^2=\varphi+1 .
\]
This ring is strictly larger than $\Z[\sqrt5]$, it has the units $\pm\varphi^k$, which are the
scales of all the self-similar pictures of the subject, and it has a prime $\sqrt5$ that is not a
unit, which is where the self-similarity stops.  We therefore put the ring first
(Section~\ref{sec:ring}) and let it carry everything that follows: the exact sign of $m+n\varphi$
decided by integers, which makes every orientation test exact (Section~\ref{sec:algorithm}); the
$\Zphi$-modules spanned by the vertices of the dodecahedron and the 120-cell, in which the
vertices of the pieces turn out to lie; and the unit/non-unit distinction between the two
four-dimensional cuts.  The pentagon itself does \emph{not} have golden coordinates in the plane
(the norm $5/16$ of $\sin^272^\circ$ is not a square), which is why the story starts in dimension
three.

\paragraph{Contributions.}
\begin{enumerate}
\item A self-contained account of the golden ring as the coordinate ring of the golden polytopes:
the ring of integers $\Zphi$ and why it is not $\Z[\sqrt5]$, the exact sign test, the primes
$\sqrt5$, $2$, $3$ and the ``golden scales'' $\pm\varphi^a\sqrt5^{\,b}$ with an exact test for
them, the lattice picture $x\mapsto(x,\bar x)$, and the vertex modules $\M\subset\Zphi^3$ (index
$4$), $\MV\subset\Zphi^4$ (index $16$) and the icosian ring $\Ico$ (Section~\ref{sec:ring}).
\item A definition of the stellation dissection valid in every dimension in which golden regular polytopes
exist, and the computation of its depth, $\varphi^{-d}$ of the inradius for $d=2,3,4$
(Section~\ref{sec:definition}).
\item The dodecahedron case done completely and exactly: the seven pieces with their metric data
and their identification with Kramer's cells, the great stellated dodecahedron as a union of $63$
of them, the behaviour of the cut at every other depth with the three critical depths
$\varphi^{-3}$, $1/\sqrt5$, $\varphi^{-1}$ derived as heights in $\Qphi$, and two exact statements
that are new: all $112$ vertices of the pieces lie in the vertex module $\M$, and the seven kinds of
pieces are exactly the $I_h$-orbits of sign vectors (Section~\ref{sec:dodeca}).
\item An exact algorithm for cutting a golden polytope by golden hyperplanes, with all predicates
in $\Zphi$, which handles the four-dimensional case in four minutes and $290$\,MB
(Section~\ref{sec:algorithm}).  A reader who wants to repeat the computation needs only the
coordinates of Section~\ref{sec:120cell} and the procedure of Section~\ref{sec:algorithm}.
\item The stellation dissection of the 120-cell, which splits into a star cut and a hereditary cut.  For
each we give the number of cells, their orbits under $\Hfour$ with a complete catalogue, the
pattern they induce on a dodecahedral facet, three-dimensional sections, and the arithmetic of
their vertices: those of the hereditary cut lie in $\MV$, those of the star cut need the
denominators $2$, $3$, $5$ (Section~\ref{sec:results}, Appendix~\ref{app:catalogue}).
\item The same cuts applied to the circumscribed disc, ball and $4$-ball: the disc and the ball
have exactly as many pieces as the pentagon and the dodecahedron, the $4$-ball has
$\NBALLFOUREXTRA$ more than the 120-cell (Section~\ref{sec:balls}).
\end{enumerate}

\section{The golden ring}\label{sec:ring}

This section contains the arithmetic the rest of the article rests on.  Nothing in it is new
(a textbook reference is \cite{Marcus1977}; for the quadratic case \cite[Ch.~XIV]{HardyWright});
what is perhaps unusual is to see how much of it a geometric computation actually uses.

\subsection{The field \texorpdfstring{$\Qphi$}{Q(sqrt5)} and its conjugation}
$\Qphi=\{a+b\sqrt5:a,b\in\Q\}$ is the smallest field containing $\Q$ and $\sqrt5$.  Addition is
componentwise, multiplication uses $(\sqrt5)^2=5$, and division is done with the \emph{conjugate}
$\sqrt5\mapsto-\sqrt5$:
\[
  \frac1{a+b\sqrt5}=\frac{a-b\sqrt5}{a^2-5b^2},
\]
where the denominator never vanishes for $(a,b)\neq(0,0)$ because $\sqrt5$ is irrational.
Adjoining $\varphi=(1+\sqrt5)/2$ instead of $\sqrt5$ gives the same field, since $\sqrt5=2\varphi-1$.
The conjugation $\sigma(a+b\sqrt5)=a-b\sqrt5$, equivalently
$\sigma(\varphi)=\bar\varphi=1-\varphi=-\varphi^{-1}$, respects sums and products: it is the
non-trivial element of the Galois group of $\Qphi/\Q$.  It produces two rational invariants, the
trace and the norm,
\begin{equation}\label{eq:trnorm}
  \Tr(x)=x+\bar x,\qquad \Nm(x)=x\bar x;\qquad
  \text{for } x=m+n\varphi:\quad \Tr(x)=2m+n,\quad \Nm(x)=m^2+mn-n^2 .
\end{equation}
The trace is additive, the norm is multiplicative, and $x$ is a root of $X^2-\Tr(x)X+\Nm(x)$.
Norms do two jobs below: they clear denominators ($1/W=\overline W/\Nm(W)$ turns a hyperplane
with coefficients in $\Qphi$ into one with coefficients in $\Zphi$), and they measure how much
a multiplication stretches the ring (Section~\ref{sec:embedding}).

\subsection{The ring of integers is \texorpdfstring{$\Zphi$}{Z[phi]}, not \texorpdfstring{$\Z[\sqrt5]$}{Z[sqrt5]}}\label{sec:integers}
An \emph{algebraic integer} is a root of a monic polynomial with integer coefficients, and the
algebraic integers in a number field $K$ form a ring $\OK$.  Both conditions matter: $x^2-\tfrac12$
is monic but not integral, $2x^2-10$ is integral but not monic, $x^2-x-1$ is both and has the root
$\varphi$, $x^2-5$ is both and has the root $\sqrt5$.  For $K=\Qphi$ the test is immediate from
\eqref{eq:trnorm}: every polynomial over $\Q$ with the root $x\notin\Q$ is a multiple of
$X^2-\Tr(x)X+\Nm(x)$, so
\[
  x\in\OK\iff\Tr(x)\in\Z\text{ and }\Nm(x)\in\Z .
\]
Write $x=(a+b\sqrt5)/2$ with $a,b\in\Q$.  Then $\Tr(x)=a$ and $\Nm(x)=(a^2-5b^2)/4$, so integrality
forces $a,b\in\Z$ with $a^2\equiv5b^2\equiv b^2\pmod4$, that is, $a\equiv b\pmod2$.  Hence
\begin{equation}\label{eq:OK}
  \OK=\Big\{\tfrac{a+b\sqrt5}2:\ a\equiv b\pmod 2\Big\}=\{m+n\varphi:\ m,n\in\Z\}=\Zphi ,
\end{equation}
with $a=2m+n$, $b=n$.  The ``odd halves'' qualify because $5\equiv1\pmod4$; for $d\equiv2,3\pmod4$
they never do and the ring of integers of $\Q(\sqrt d)$ is $\Z[\sqrt d]$.  The subring
$\Z[\sqrt5]=\{a\equiv b\equiv0\}$ has index $2$ in $\Zphi$ and misses $\varphi$, $\varphi^{-2}=2-\varphi$,
and every coordinate $m+n\varphi$ with $n$ odd, such as the $\varphi^2=1+\varphi$ of the vertex
$(0,\varphi^2,1)$ of the dodecahedron: the coordinate ring of the golden
polytopes is $\Zphi$, and when this article says ``golden ring'' it means $\Zphi$.

\begin{example}
$\varphi$: $\Tr=1$, $\Nm=-1$, integral.  $\sqrt5/2$: $\Tr=0$, $\Nm=-5/4$, \emph{not} integral,
although it has the same denominator as $\varphi$; the pair $(a,b)=(0,1)$ fails the parity test.
$2+3\varphi$: $\Tr=7$, $\Nm=4+6-9=1$, integral and a unit (it is $\varphi^4$).
$(2+\varphi)/5$: $\Tr=1$, $\Nm=1/5$, not integral; this is the number that relates the face centres
of the dodecahedron to its face normals (Section~\ref{sec:modules}), and the reason the face
centres are not in $\Zphi^3$ although the vertices and normals are.
\end{example}

\subsection{Units and Fibonacci numbers}\label{sec:units}
$\Nm(\varphi)=-1$, so $\varphi$ is a unit; by Dirichlet's unit theorem a real quadratic field has
unit rank one, and $\varphi$ is the fundamental unit: $\Zphi^\times=\{\pm\varphi^k:k\in\Z\}$.
From $\varphi^2=\varphi+1$,
\begin{equation}\label{eq:powers}
\begin{gathered}
  \varphi^{n}=F_n\varphi+F_{n-1}\quad(F_0=0,\ F_1=1),\qquad
  \sqrt5=2\varphi-1,\qquad \varphi^3=1+2\varphi,\qquad \varphi^4=2+3\varphi,\\
  \varphi^{-1}=\varphi-1,\qquad \varphi^{-2}=2-\varphi,\qquad \varphi^{-3}=2\varphi-3,\qquad
  \varphi^{-4}=5-3\varphi .
\end{gathered}
\end{equation}  A unit scale is what
makes a construction self-similar \emph{inside the ring}: multiplication by $\varphi^k$ is a
bijection of $\Zphi$, and of $\Zphi^d$, onto itself.  The depths $\varphi^{-2}$, $\varphi^{-3}$,
$\varphi^{-4}$ of the stellation dissections are units, so each core is a copy of the hull inside
the same lattice.  The hereditary depth $1-\varphi^{-4}=3\varphi-4$ is not: its norm is
$16-12-9=-5$.

\subsection{Primes of \texorpdfstring{$\Zphi$}{Z[phi]} and golden scales}\label{sec:primes}
$\Zphi$ is a principal ideal domain (its class number is one), so its primes are elements.  A
rational prime $p$ behaves in one of three ways, decided by whether $x^2-x-1$ has a root modulo $p$:
\begin{itemize}
\item $p=5$ \emph{ramifies}: $5=\sqrt5^{\,2}$, and $\sqrt5=\varphi^{-1}(\varphi+2)$ is a prime of
norm $-5$.  It is the prime of the discriminant, and the one prime that occurs geometrically in
this article.
\item $p\equiv\pm1\pmod5$ \emph{splits} into two conjugate primes of norm $p$: $11=(3+\varphi)(4-\varphi)$,
$19=\Nm(4+\varphi)$, $29=\Nm(5+\varphi)$.
\item $p\equiv\pm2\pmod5$ is \emph{inert}: $2$, $3$, $7$, $13$ stay prime in $\Zphi$, with norm $p^2$.
\end{itemize}
The primes $2$ and $3$ appear below as denominators of vertices of the star cut, and $5$ as the
factor that separates the two four-dimensional cuts.  To say this precisely we name the scales
that occur in golden geometry.

\begin{definition}[golden scale]\label{def:scale}
$\lambda\in\Qphi^\times$ is a \emph{golden scale} if $\lambda=\pm\varphi^a\sqrt5^{\,b}$ with
$a,b\in\Z$; these form the subgroup of $\Qphi^\times$ generated by $-1$, the fundamental unit and
the ramified prime.
\end{definition}

\begin{proposition}[exact test]\label{prop:scale}
$\lambda$ is a golden scale if and only if $|\Nm(\lambda)|=5^{\,b}$ for some $b\in\Z$ and
$\lambda\sqrt5^{-b}$ is an algebraic integer; it is then a unit, hence $\pm\varphi^a$.
\end{proposition}
\begin{proof}
$\Nm(\pm\varphi^a\sqrt5^{\,b})=(-1)^a(-5)^b$ gives one direction.  Conversely $u=\lambda\sqrt5^{-b}$
has norm $\pm1$; if it is an algebraic integer it is a unit of $\Zphi$, so $u=\pm\varphi^a$.  The
integrality condition cannot be dropped: $(-21-4\sqrt5)/19$ has norm $1$ and is not in $\Zphi$.
\end{proof}

\begin{table}[t]
  \centering\footnotesize
  \begin{tabular}{llll}
    \toprule
    ratio & where it occurs & $\Nm$ & golden scale $\pm\varphi^a\sqrt5^{\,b}$\\
    \midrule
    $\varphi^{-2}$, $\varphi^{-3}$, $\varphi^{-4}$ & depth of the dissection, $d=2,3,4$ & $1$, $-1$, $1$ & $b=0$: units\\
    $3\varphi-4=1-\varphi^{-4}$ & hereditary depth & $-5$ & $\sqrt5\,\varphi^{-2}$, $(a,b)=(-2,1)$\\
    $1/\sqrt5$ & critical depth (face centres) & $-1/5$ & $(0,-1)$\\
    $\varphi^{-1}$ & critical depth (edge midpoints) & $-1$ & $(-1,0)$\\
    $5-2\sqrt5=7-4\varphi$ & facet-parallel cut of a facet, \S\ref{sec:120cell} & $5$ & $\sqrt5\,\varphi^{-3}$, $(-3,1)$\\
    $2\varphi^{-2}$ & edge of the 120-cell & $4$ & no (a rational factor $2$)\\
    \bottomrule
  \end{tabular}
  \caption{The ratios of this article under the test of Proposition~\ref{prop:scale}
  (Appendix~\ref{app:checks}, block~C).  All depths are golden scales; only the hereditary depth
  carries the prime $\sqrt5$.}\label{tab:scales}
\end{table}

\subsection{The exact sign}\label{sec:sign}
Every decision in the computation of Section~\ref{sec:algorithm} is ``is $m+n\varphi$ positive?''
with $m,n\in\Z$.  Two derivations give the same rule.

\emph{Via $\sqrt5$.}  $m+n\varphi=\tfrac12(p+q\sqrt5)$ with $p=2m+n$, $q=n$.  If $p$ and $q$ have
the same sign, that is the sign.  If they disagree, $\sgn(p+q\sqrt5)=\sgn p$ exactly when
$p^2>5q^2$:
\begin{equation}\label{eq:sign}
  \sgn(m+n\varphi)=\begin{cases}\sgn p & \text{if } pq\ge0,\ (p,q)\ne(0,0),\\
  \sgn p & \text{if } pq<0 \text{ and } p^2>5q^2,\\ \sgn q & \text{if } pq<0 \text{ and } p^2<5q^2.\end{cases}
\end{equation}

\emph{Via trace and norm.}  $x$ and $\bar x$ have the same sign if and only if $\Nm(x)>0$, and then
$\sgn x=\sgn\Tr(x)=\sgn(2m+n)$.  If $\Nm(x)<0$ the conjugates have opposite signs and $x>\bar x$
exactly when $n>0$, since $x-\bar x=n\sqrt5$.  So $\sgn(m+n\varphi)=\sgn(2m+n)$ if $m^2+mn-n^2>0$
and $\sgn n$ otherwise.  The two rules are the same rule.

\emph{Why there is no tie.}  $p^2=5q^2$ with $(p,q)\neq(0,0)$ would make $\sqrt5$ rational.  This
is the whole reason there is no $\varepsilon$ in the algorithm: predicates are decided, not
estimated.  The decision is nevertheless delicate in floating point, because
$F_n\varphi-F_{n+1}=(-1)^{n+1}\varphi^{-n}$ is as small as one likes with coefficients as small as
Fibonacci numbers; Section~\ref{sec:float} shows what the estimate costs.  In the implementation
$m,n$ are $128$-bit integers, every operation is checked for overflow, and $p^2$ is compared with
$5q^2$ in $256$ bits; the coefficient size is bounded once and for all because every vertex is a
fixed-size determinant of \emph{input} hyperplanes (Section~\ref{sec:dd}).

\subsection{The ring as a lattice: \texorpdfstring{$x\mapsto(x,\bar x)$}{x -> (x, conj x)}}\label{sec:embedding}
As a subset of $\R$ the ring $\Zphi$ is dense, which hides its structure.  The map
$\iota(x)=(x,\bar x)$ sends it onto the lattice of $\R^2$ with basis $(1,1)$ and
$(\varphi,\bar\varphi)$, of area $|\varphi-\bar\varphi|=\sqrt5$ per cell (Figure~\ref{fig:mink}).
In this picture $x=m+n\varphi\leftrightarrow(m,n)\in\Z^2$, addition is addition of vectors, and
multiplication by $s\in\Zphi$ is the linear map $\operatorname{diag}(s,\bar s)$, of determinant
$\Nm(s)$.  Two consequences are used below.
\begin{itemize}
\item Multiplication by $s$ maps $\Zphi$ onto the sublattice $s\Zphi$ of index $|\Nm(s)|$.  For a
unit the index is $1$ and $s\Zphi=\Zphi$: the lattice is mapped onto itself.  For $s=\sqrt5$ the
index is $5$: $\sqrt5\Zphi$ is a proper sublattice, and $(\sqrt5\varphi^{-2})\Zphi^4$ has index
$5^4=625$ in $\Zphi^4$.  This is the arithmetic content of the sentence ``the hereditary core is
not similar to the hull by a unit'' in Table~\ref{tab:bifurcation}.
\item The subring $\Z[\sqrt5]$ is the sublattice $\{(m,n):n\text{ even}\}$ of index $2$, and
$\Zphi=\Z[\sqrt5]\sqcup(\varphi+\Z[\sqrt5])$: two cosets, the ``even'' and the ``odd'' points.
Equation~\eqref{eq:OK} is the statement that the odd coset consists of algebraic integers.
\end{itemize}

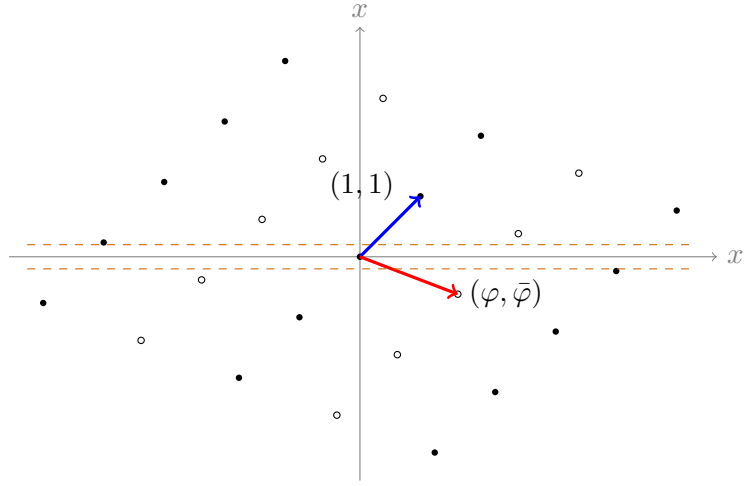
\begin{figure}[t]\centering
\begin{tikzpicture}[scale=0.8]
  \draw[->,gray] (-5.8,0)--(5.9,0) node[right,gray]{$x$};
  \draw[->,gray] (0,-3.7)--(0,3.8) node[above,gray]{$\bar x$};
  \foreach \m in {-2,...,2}{\foreach \n in {-2,...,2}{
      \pgfmathsetmacro{\px}{\m + \n*1.61803398875}\pgfmathsetmacro{\py}{\m - \n*0.61803398875}
      \pgfmathparse{int(mod(\n+100,2))}
      \ifnum\pgfmathresult=0 \fill (\px,\py) circle (1.5pt);\else \draw (\px,\py) circle (1.5pt);\fi}}
  \draw[->,very thick,blue] (0,0)--(1,1) node[pos=0.75,above left,black]{$(1,1)$};
  \draw[->,very thick,red] (0,0)--(1.61803398875,-0.61803398875) node[pos=1.0,right,black]{$(\varphi,\bar\varphi)$};
  \draw[dashed,orange!80!black] (-5.5,0.2)--(5.5,0.2) (-5.5,-0.2)--(5.5,-0.2);
\end{tikzpicture}
\caption{$\Zphi$ under $x\mapsto(x,\bar x)$.  Filled points form the subring $\Z[\sqrt5]$ ($n$ even),
open points its coset $\varphi+\Z[\sqrt5]$; together they are the ring of integers.  Projecting to
the $x$-axis gives the dense set $\Zphi\subset\R$; keeping only the points with $\bar x$ in a
window (dashed strip) gives a Fibonacci quasicrystal.  Multiplication by $\varphi$ stretches the
$x$-axis by $\varphi$ and squeezes the $\bar x$-axis by $\varphi^{-1}$ and maps the lattice onto
itself; multiplication by $\sqrt5$ maps it onto a sublattice of index $5$.}\label{fig:mink}
\end{figure}

\subsection{Golden coordinates in dimension \texorpdfstring{$d$}{d}: lattices, modules, icosians}\label{sec:modules}
A \emph{$\Zphi$-module} of rank $d$ in $\Qphi^d$ is a $\Z$-lattice of rank $2d$ (the coordinates
$(m_1,n_1,\dots,m_d,n_d)$ of $m_i+n_i\varphi$) closed under multiplication by $\varphi$.
Membership and index are integer linear algebra: clear denominators, bring the generators to
echelon form, compare pivots.  The modules that matter here are $\Zphi^d$ itself and the
\emph{vertex modules}, the $\Zphi$-spans of the vertex sets of the two polytopes.

\paragraph{Why the pentagon fails and the dodecahedron works.}  The rotation by $72^\circ$ has the
entries $\cos72^\circ=(\varphi-1)/2$ and $\sin72^\circ=\tfrac12\sqrt{\varphi+2}$.  The sine is not
in $\Qphi$: $\Nm\big((\varphi+2)/4\big)=5/16$ is not the square of a rational, so $\varphi+2$ is
not a square in $\Qphi$ and no regular pentagon has all its vertices in $\Qphi^2$.  In three
dimensions the same rotation, about the five-fold axis $n$ of the dodecahedron of
Section~\ref{sec:dodeca} ($|n|^2=\varphi+2$), is
\begin{equation}\label{eq:R72}
  R_{72}=\tfrac12\big((\varphi-1)\,I+[n]_\times+\varphi^{-2}\,nn^{\!\top}\big),\qquad n=(1,\varphi,0),
\end{equation}
where $[n]_\times$ is the cross-product matrix: Rodrigues' formula
$\cos\theta\,I+\sin\theta\,[k]_\times+(1-\cos\theta)\,kk^{\!\top}$ with $k=n/|n|$, in which
$\sin72^\circ/|n|=\tfrac12$ and $(1-\cos72^\circ)/|n|^2=\varphi^{-2}/2$: the sine has been absorbed
into the length of the axis vector, and all entries lie in $\tfrac12\Zphi$.  The full icosahedral group $I_h$ (order $120$) therefore acts on $\Qphi^3$
by matrices with entries in $\tfrac12\Zphi$; it does not preserve $\Zphi^3$, but it does preserve
the vertex module (Result~\ref{res:modules}).  The same mechanism gives $\Zphi^4$ coordinates to
the 120-cell and the 600-cell, and nothing beyond: the finite reflection groups with a golden
invariant are $H_2$, $H_3$, $\Hfour$, and there is no $H_5$ \cite{Coxeter1973}.

\begin{result}[the vertex modules; Appendix~\ref{app:checks}, blocks H and I]\label{res:modules}
Let $\M\subset\Zphi^3$ be the $\Zphi$-span of the $20$ vertices of the dodecahedron and
$\MV\subset\Zphi^4$ the $\Zphi$-span of the $600$ vertices of the 120-cell, in the coordinates of
Sections~\ref{sec:dodeca} and \ref{sec:120cell}.
\begin{enumerate}
\item $[\Zphi^3:\M]=4$; $\M$ is also the $\Zphi$-span of the twelve face normals; the unit vectors
are not in $\M$; the edge midpoints $m_e$ satisfy $2m_e\in\M$ but not all $m_e\in\M$.  $I_h$
preserves $\M$ and not $\Zphi^3$.  $\varphi^2\M=\M$, and $\sqrt5\M$ has index $5^3$ in $\M$.
\item $[\Zphi^4:\MV]=16$.  The $120$ facet normals $n$ of the 120-cell have $|n|=2$ and the
vectors $n/2$ are unit quaternions closed under multiplication: the binary icosahedral group $2I$,
the vertices of the 600-cell, the roots of $\Hfour$.  Their $\Zphi$-span is the \emph{icosian
ring} $\Ico$ \cite{ConwaySloane,Moody1993}, with $[\Ico:\Zphi^4]=16$ and $[\Ico:\MV]=256$.  The
$60$ reflections generating $\Hfour$ preserve $\MV$ and $\Ico$ but not $\Zphi^4$ (the reflection in
$(\varphi,1,\varphi^{-1},0)$ sends $e_1$ to $(1-\varphi^2/2,-\varphi/2,-\tfrac12,0)$).
$\varphi^{-4}\MV=\MV$, and $(\sqrt5\varphi^{-2})\MV$ has index $5^4=625$ in $\MV$.
\item With the bilinear form $2\Tr(\lambda\,x\cdot y)$, $\lambda=\varphi/\sqrt5=(2+\varphi)/5$,
the icosian ring is an even unimodular lattice of rank $8$: it is $E_8$, and its $240$ roots
(vectors of norm $2$) are the $120$ unit icosians $n/2$ together with their $120$ multiples by
$\varphi^{-1}$, two concentric 600-cells in $\R^4$
\cite[\S8.2.1]{ConwaySloane}; the factor $\lambda$ is forced because the dual of $\Zphi$ under
the trace form is $\tfrac1{\sqrt5}\Zphi$.
\end{enumerate}
\end{result}

Item 3 is the lattice from which Elser and Sloane project their four-dimensional quasicrystal
with point group $\Hfour$ \cite{Elser1987}; it is not used below, but it says where the
120-cell's arithmetic ultimately lives, and it is the reason the module $\MV$, not $\Zphi^4$, is
the object preserved by the symmetry.  The invariance in item 2 is what turns a computation on
one representative cell per orbit into a statement about every cell (Result~\ref{res:vert4}).

\section{The stellation dissection}\label{sec:definition}

\subsection{Stellation cores}
Let $C$ be a regular polytope in $\R^d$ and consider the arrangement of its facet hyperplanes.
The bounded cells of this arrangement are the \emph{stellation cells} of $C$, and a stellation of
$C$ is a union of such cells with the symmetry of $C$ \cite{Coxeter1938,Coxeter1973}.  Among the
stellations of the pentagon there is one regular star polygon, the pentagram $\{5/2\}$; among
those of the dodecahedron there are three regular ones, the small stellated dodecahedron
$\{5/2,5\}$, the great dodecahedron $\{5,5/2\}$ and the great stellated dodecahedron $\{5/2,3\}$;
among those of the 120-cell there are nine regular ones, the Schl\"afli--Hess polychora other
than the grand 600-cell \cite{Coxeter1973}.  In each dimension the last of them, the
\emph{final regular stellation} $S$, is the one whose vertices lie on a copy of $C$ itself:
$\conv S=\lambda C$ for some $\lambda>1$.

\begin{definition}\label{def:cut}
Let $C$ be a regular polytope with golden geometry, $S$ its final regular stellation and
$B=\conv S$.  The \emph{stellation dissection} of $B$ is the dissection of $B$ by the facet hyperplanes
of $C$, the arrangement whose bounded regions are the stellation cells of $C$
\cite{Coxeter1938,Wenninger1971}.  Its \emph{depth} is $t=r(C)/r(B)=\lambda^{-1}$, the ratio of
inradii.  The pieces of the dissection are the stellation cells of $C$ inside $B$ together with
the unbounded regions of the arrangement truncated by $B$.
\end{definition}

By construction the star $S$ is a union of pieces of the stellation dissection (indeed of
stellation cells), and the core $C$ is one of the pieces.

\begin{result}\label{res:depth}
The depth of the stellation dissection is $t=\varphi^{-2}$ for the pentagon, $\varphi^{-3}$ for the
dodecahedron and $\varphi^{-4}$ for the 120-cell.  In every case $t$ is a unit of $\Zphi$.
\end{result}

For the pentagon and the dodecahedron this is classical: the pentagram's outer vertices lie on the
pentagon $\varphi^2 C$ and the great stellated dodecahedron's on $\varphi^3C$
\cite{Coxeter1973,Wenninger1971}.  For the 120-cell we checked it with the tool
\texttt{stellation\_shells}: intersect all $4$-tuples of the $120$ facet hyperplanes of $C$
(inradius $1$), group the $752\,160$ intersection points by their distance from the centre
($142$ shells), and keep the shells that are congruent to the vertex set of $C$ and put a
facet-sized cluster of $20$ points on each facet hyperplane.  The $600$-point shells at powers of
$\varphi$ lie at $\varphi^0,\dots,\varphi^4$; of these, only $\varphi^4$ (besides $\varphi^0$, the
core itself) puts exactly $20$ points on each facet hyperplane, where they form a dodecahedron
$\varphi^6$ times the facet of $C$.  These are the $600$ vertices of the great grand stellated
120-cell $\{5/2,3,3\}$ \cite{Coxeter1973,WikiGGS}.  The shells at $\varphi^1,\varphi^2,\varphi^3$
put $60$ points on each hyperplane and belong to stars with $120$ vertices.  The incidence that
matters, that the hyperplane $n\cdot x=\varphi^{-4}\cdot r(B)$ contains a $20$-vertex layer of
$B$, is verified exactly in Section~\ref{sec:120cell}.

\subsection{The pentagon}\label{sec:pentagon}
With circumradius $1$ and vertices $v_k=(\cos\theta_k,\sin\theta_k)$, $\theta_k=90^\circ+72^\circ k$,
the five diagonal lines $v_kv_{k+2}$ are the side lines of the inner pentagon of circumradius
$\varphi^{-2}$, at distance $\varphi^{-2}\cos36^\circ=\varphi^{-1}/2$ from the centre.  The eleven
pieces (Figure~\ref{fig:pentagon}) are
\begin{itemize}
\item the inner pentagon, of area $\varphi^{-4}\approx14.6\%$ of the whole;
\item five \emph{points}: golden triangles with angles $36^\circ,72^\circ,72^\circ$, base a side
$s$ of the inner pentagon and legs $\varphi s$, together $(5-2\sqrt5)/\varphi\approx32.6\%$;
\item five \emph{gnomons}: golden gnomons with angles $36^\circ,36^\circ,108^\circ$, with equal
sides $\varphi s$ along the diagonals and long side a side $\varphi^2s$ of the large pentagon,
together $5-2\sqrt5\approx52.8\%$.
\end{itemize}
Each diagonal has length $\varphi^3s$ and is divided by the other diagonals into segments
$\varphi s,\ s,\ \varphi s$, so a point and a gnomon share the leg $\varphi s$ and their areas are in
the ratio $\sin36^\circ:\sin108^\circ=1:\varphi$; with $\vol(\text{core})=\varphi^{-4}$ this
fixes all the fractions.  The pieces come in three congruence classes, with multiplicities
$1,5,5$, which are the orbits of the dihedral group $D_5$ (order $10$) on the pieces.  All ratios
of lengths and areas are in $\Qphi$; the coordinates are not (Section~\ref{sec:modules}), so the
pentagon needs no computer and gets none.

\begin{figure}[t]
  \centering
  \begin{tikzpicture}[scale=2.6, line join=round]
    \def\R{1}\def\r{0.381966}
    \foreach \kk in {0,...,4} { \coordinate (v\kk) at ({90+72*\kk}:\R); \coordinate (u\kk) at ({54+72*\kk}:\r); }
    \fill[blue!25!white] (u0)--(u1)--(u2)--(u3)--(u4)--cycle;
    \foreach \kk/\kn in {0/1, 1/2, 2/3, 3/4, 4/0} { \fill[red!55!orange!60] (v\kk)--(u\kk)--(u\kn)--cycle; }
    \foreach \kk/\kn in {0/1, 1/2, 2/3, 3/4, 4/0} { \fill[green!50!teal!40] (v\kk)--(v\kn)--(u\kn)--cycle; }
    \draw[very thin,black!60] (v0)--(v1)--(v2)--(v3)--(v4)--cycle;
    \foreach \kk/\kn in {0/2, 1/3, 2/4, 3/0, 4/1} { \draw[thick,black!80] (v\kk)--(v\kn); }
  \end{tikzpicture}\hspace{1.5cm}%
  \begin{tikzpicture}[scale=2.6, line join=round]
    \def\R{1}\def\r{0.381966}
    \foreach \kk in {0,...,4} { \coordinate (v\kk) at ({90+72*\kk}:\R); \coordinate (u\kk) at ({54+72*\kk}:\r); }
    \def\ex{0.5}
    \fill[blue!25!white,draw=black!70,very thin] (u0)--(u1)--(u2)--(u3)--(u4)--cycle;
    \foreach \kk/\kn in {0/1, 1/2, 2/3, 3/4, 4/0} {
      \coordinate (c) at (barycentric cs:v\kk=1,u\kk=1,u\kn=1);
      \fill[red!55!orange!60,draw=black!70,very thin] ($(v\kk)+\ex*(c)$)--($(u\kk)+\ex*(c)$)--($(u\kn)+\ex*(c)$)--cycle;
    }
    \foreach \kk/\kn in {0/1, 1/2, 2/3, 3/4, 4/0} {
      \coordinate (c) at (barycentric cs:v\kk=1,v\kn=1,u\kn=1);
      \fill[green!50!teal!40,draw=black!70,very thin] ($(v\kk)+\ex*(c)$)--($(v\kn)+\ex*(c)$)--($(u\kn)+\ex*(c)$)--cycle;
    }
  \end{tikzpicture}
  \caption{The stellation dissection of the pentagon: the five diagonals (left) and the eleven pieces
  exploded (right): the core pentagon, five points (golden triangles) and five gnomons.}
  \label{fig:pentagon}
\end{figure}
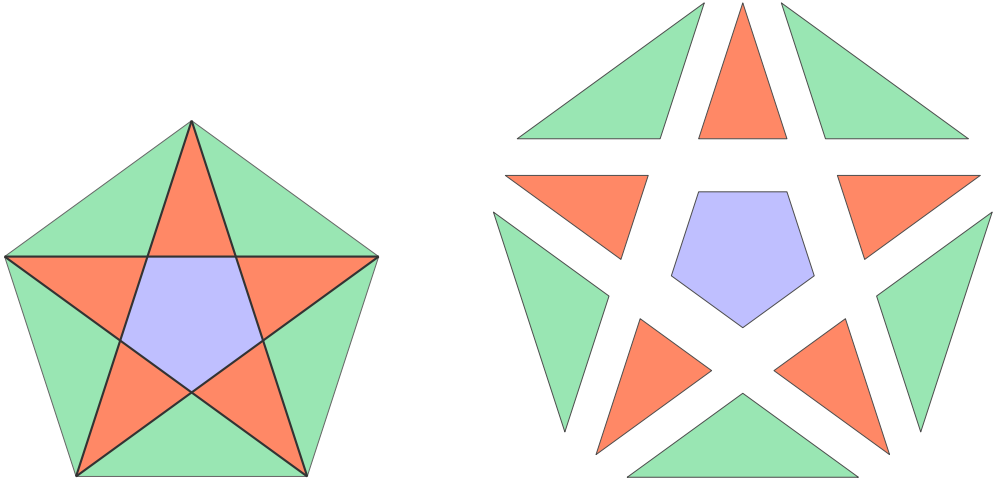

\subsection{Pieces as sign vectors}\label{sec:signvectors}
Let the cut hyperplanes be $h_k(x)=n_k\cdot x-c_k$, $k=1,\dots,m$.  A point $x$ on none of them
has a \emph{sign vector} $\sigma(x)=(\sgn h_k(x))_k\in\{\pm\}^m$, and the piece containing $x$ is
exactly $\{y\in B:\sigma(y)=\sigma(x)\}$: the pieces of an arrangement are its non-empty sign
classes.  Two things follow.  First, \emph{which piece a point lies in} needs no cutting: $m$ exact
signs of Section~\ref{sec:sign}, with no recursion and no state ($12$ signs for the dodecahedron,
$120$ for the 120-cell; Section~\ref{sec:algorithm}).  Second, a piece has a canonical name, and
the symmetry group of $B$, which permutes the cut hyperplanes, acts on names by permuting
coordinates, so that congruence classes of pieces become orbits of sign vectors
(Section~\ref{sec:orbits}).  With the convention that $+$ means ``on the far side of the
hyperplane from the centre'', the number of plus signs is the number of cut hyperplanes the piece
lies beyond, the \emph{layer} of the stellation diagram.  Figure~\ref{fig:pentasign} shows the
eleven names of the pentagon.

\begin{figure}[t]\centering
\begin{tikzpicture}[scale=2.2]
  \foreach \k in {0,...,4}{\coordinate (v\k) at (90+72*\k:1);}
  \draw[thick] (v0)--(v1)--(v2)--(v3)--(v4)--cycle;
  \foreach \k in {0,...,4}{\pgfmathtruncatemacro{\j}{mod(\k+2,5)}\draw[blue!60!black] (v\k)--(v\j);}
  \node[font=\scriptsize] at (0,0) {$-----$};
  \foreach \k in {0,...,4}{
    \pgfmathtruncatemacro{\a}{mod(\k+1,5)}
    \node[font=\scriptsize,fill=white,inner sep=1pt] at (90+72*\k:0.80) {$d_{\k}{+}$};
    \node[font=\scriptsize,fill=white,inner sep=1pt] at (126+72*\k:0.62) {$d_{\k}{+}d_{\a}{+}$};
  }
\end{tikzpicture}
\caption{The eleven pieces of the pentagon named by their sign vectors against the diagonals
$d_0,\dots,d_4$ ($d_k$ cuts off vertex $k$; only the $+$ entries are written).  Core: weight $0$;
the five points: weight $1$; the five gnomons: weight $2$.  Of the $2^5=32$ possible sign vectors,
$11$ are realised.  With $12$ planes the same naming gives the $165$ pieces of the dodecahedron
(Result~\ref{res:orbits}), with $120$ hyperplanes the $\NSTAR$ or $\NHERED$ pieces of the 120-cell.}
\label{fig:pentasign}
\end{figure}
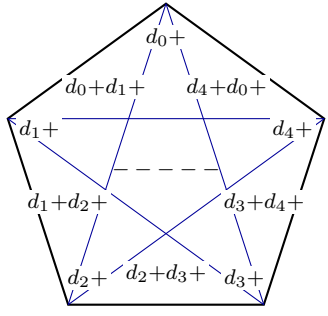

\subsection{Vertex layers}\label{sec:layers}
There is a second way to see the cut, from the outside.  Project the vertices of $B$ onto a facet
normal: they fall into layers, and the heights of the layers are elements of $\Qphi$.  For the
pentagon the layers along a side normal contain $2,2,1$ vertices; for the dodecahedron, $5,5,5,5$
at heights $\pm1,\pm\varphi^{-3}$ of the inradius; for the 120-cell, with heights as fractions of
the inradius,
\begin{equation}\label{eq:layers}
  \setlength{\arraycolsep}{3pt}\scriptstyle
  \begin{array}{c|cccccccc}
    \text{height} & \pm1 & \pm(1-\varphi^{-4}) & \pm2\varphi^{-2} & \pm\varphi^{-1} & \pm\varphi^{-2} & \pm\varphi^{-3} & \pm\varphi^{-4} & 0\\
    & \pm1 & \pm0.8541 & \pm0.7639 & \pm0.6180 & \pm0.3820 & \pm0.2361 & \pm0.1459 & 0\\ \hline
    \text{vertices} & 20 & 20 & 30 & 60 & 60 & 60 & 20 & 60
  \end{array}
\end{equation}
(all fifteen heights and counts are recomputed exactly in Appendix~\ref{app:checks}, block~I).
In every dimension the cut hyperplanes pass through the vertex layer of $B$ nearest to the centre
(the pentagon's diagonals through the two vertices next to a side, the dodecahedron's planes through
the five vertices adjacent to a face).  For the pentagon and the dodecahedron this layer is also the
layer \emph{adjacent to a facet}.  For the 120-cell the two layers are different: the layer nearest
the centre has $20$ vertices at height $\varphi^{-4}$, and the layer adjacent to a facet has $20$
vertices at height $1-\varphi^{-4}=3\varphi-4$.  The definition therefore bifurcates
(Table~\ref{tab:bifurcation}), and the ring tells the two branches apart before any cell is
computed: $\varphi^{-4}$ is a unit and $3\varphi-4=\sqrt5\,\varphi^{-2}$ is not.

\begin{table}[t]
  \centering\footnotesize
  \begin{tabular}{lll}
    \toprule
    & star cut & hereditary cut\\
    \midrule
    depth $t$ & $\varphi^{-4}=5-3\varphi\approx0.1459$ & $1-\varphi^{-4}=3\varphi-4\approx0.8541$\\
    $t$ as a golden scale & $\varphi^{-4}$, a unit, $\Nm=1$ & $\sqrt5\,\varphi^{-2}$, $\Nm=-5$\\
    hyperplanes (\S\ref{sec:120cell}) & $n\cdot x=2\varphi^{-2}=4-2\varphi$ & $n\cdot x=2\sqrt5$\\
    characterisation & $\{5/2,3,3\}$ is a union of cells & facets get the 3D dissection\\
    core $tB$ and the module $\MV$ & $tB$ has vertices in $\MV$ ($t\MV=\MV$) & $t\MV\subsetneq\MV$, index $5^4$\\
    hyperplanes meeting a facet & $42$ & $12$\\
    pieces on a facet & $3\,655$ & $165$\\
    vertices of the pieces (Result~\ref{res:vert4}) & denominators $1,2,3,5$; not all in $\MV$ & all in $\MV\subset\Zphi^4$\\
    cells / orbits & $\NSTAR$ / $\NSTARCLASSES$ & $\NHERED$ / $\NHEREDCLASSES$\\
    \bottomrule
  \end{tabular}
  \caption{The two stellation dissections of the 120-cell.}\label{tab:bifurcation}
\end{table}

The \emph{star cut} is Definition~\ref{def:cut}.  The \emph{hereditary cut} is the one that
passes the construction down to the facets: each dodecahedral facet of $B$ is cut by exactly the
twelve planes of its own three-dimensional stellation dissection, whose pentagonal faces are cut
by their diagonals.  We compute both.

\section{The dodecahedron}\label{sec:dodeca}

Take the dodecahedron $B$ with vertices $(\pm\varphi,\pm\varphi,\pm\varphi)$ and the cyclic
permutations of $(0,\pm\varphi^2,\pm1)$.  Its twelve faces are $n\cdot x\le\varphi^3$ for the cyclic
permutations $n$ of $(\pm1,\pm\varphi,0)$, with $|n|^2=\varphi+2=\sqrt5\varphi$; the inradius is
$\varphi^3/\sqrt{\varphi+2}$, the circumradius $\varphi\sqrt3$, the edge $2$.  The stellation
dissection is the cut by the twelve planes $n\cdot x=1$, which pass through five vertices of $B$
each: $n\cdot v\in\{\pm1,\pm\varphi^3\}$ for every vertex $v$ and every normal $n$, the two layers
of Section~\ref{sec:layers}, at $\varphi^{-3}$ and $1$ of the inradius.  Everything is in
$\Zphi^3$.  (The copy of midradius $1$, with vertices $(\pm\varphi^{-1})^3$ and the cyclic
permutations of $(0,\pm1,\pm\varphi^{-2})$, is this one divided by the unit $\varphi^2$; module
statements hold in both normalisations, and the checks of Appendix~\ref{app:checks} use the
smaller copy, with faces $n\cdot x=1$ and cut planes $n\cdot x=\varphi^{-3}$.)

\begin{result}[the seven pieces]\label{res:seven}
The twelve planes cut $B$ into $165$ convex pieces in the seven congruence classes of
Table~\ref{tab:dodeca}, which are the orbits of the icosahedral group $I_h$ (order $120$).  All
pieces are closed $2$-manifolds satisfying Euler's formula and their volumes add up to $\vol B$
with relative error $7\cdot10^{-16}$.  Counting the number of cut planes a piece lies beyond
(its \emph{layer}), the layers $0,1,2,3$ contain $1,12,30,20$ pieces, whose union is the great
stellated dodecahedron; the remaining $102$ pieces fill the gap between the star and $B$.
\end{result}

\begin{table}[t]
  \centering\footnotesize
  \begin{tabular}{rlrllcr}
\toprule
class & piece & copies & $V$, $F$ (face sizes) & edge lengths (count) & layer & volume \%\\
\midrule
1 & core dodecahedron & 1 & $20$, $12$ ($5^{12}$) & $1(30)$ & 0 & 1.32\\
5 & small tent & 12 & $6$, $6$ ($3^{5} 5$) & $1(5)\;\varphi(5)$ & 1 & 1.63\\
6 & small tetrahedron & 30 & $4$, $4$ ($3^{4}$) & $1(1)\;\varphi(4)\;\varphi^2(1)$ & 2 & 1.82\\
2 & spike (triangular bipyramid) & 20 & $5$, $6$ ($3^{6}$) & $\varphi(3)\;\varphi^2(3)\;\varphi^3(3)$ & 3 & 15.40\\
0 & wedge & 30 & $6$, $8$ ($3^{8}$) & $\varphi^2(7)\;\varphi^3(5)$ & 4 & 57.55\\
4 & tetrahedron $T_2$ & 60 & $4$, $4$ ($3^{4}$) & $\varphi(1)\;\varphi^2(4)\;\varphi^3(1)$ & 5 & 15.40\\
3 & tent (pentagonal pyramid) & 12 & $6$, $6$ ($3^{5} 5$) & $\varphi(5)\;\varphi^2(5)$ & 6 & 6.89\\
\bottomrule
\end{tabular}

  \caption{The seven pieces of the stellation dissection of the dodecahedron.  Edge lengths are given as
  powers of $\varphi$ times the edge of the core (the face sizes are listed as $k^m$ for $m$ faces
  with $k$ vertices).  Layer $k$ means the piece lies beyond $k$ of the twelve cut planes; layers
  $0$--$3$ form the great stellated dodecahedron.}\label{tab:dodeca}
\end{table}

The names in Table~\ref{tab:dodeca} are Kramer's where we could match them.  His ``tent'' is a
pentagonal pyramid with base edge $1$ and five Robinson triangles $(1,\tau,\tau)$ as sides
\cite{Kramer1982,Koca2020}; our tents have base $\varphi^k$ and lateral edges $\varphi^{k+1}$, the
same shape.  His tetrahedron $T_2$ has faces $2\times(1,\tau,\tau)$ and $2\times(\tau^2,\tau,\tau)$,
that is, edges $1,\tau^2$ and four of length $\tau$ \cite{Koca2020}; our $90$ tetrahedra have edges
$\varphi^k,\varphi^{k+2}$ and four of length $\varphi^{k+1}$.  Since his seven cells pack a
dodecahedron and inflate by $\tau^3$ exactly as these do, we are confident that the two sets
coincide, although we could not consult the original paper.  The second-layer pieces are the $30$
small tetrahedra (the ``wedges'' of the great dodecahedron in Wenninger's terminology), the
third-layer pieces are the $20$ triangular bipyramids that form the spikes of the great stellated
dodecahedron, and the first layer consists of the $12$ small tents that turn the core into the
small stellated dodecahedron.  Figure~\ref{fig:dodeca} shows the cut assembled, exploded, and with
the star highlighted.

\begin{figure}[t]
  \centering
  \fig[0.325\linewidth]{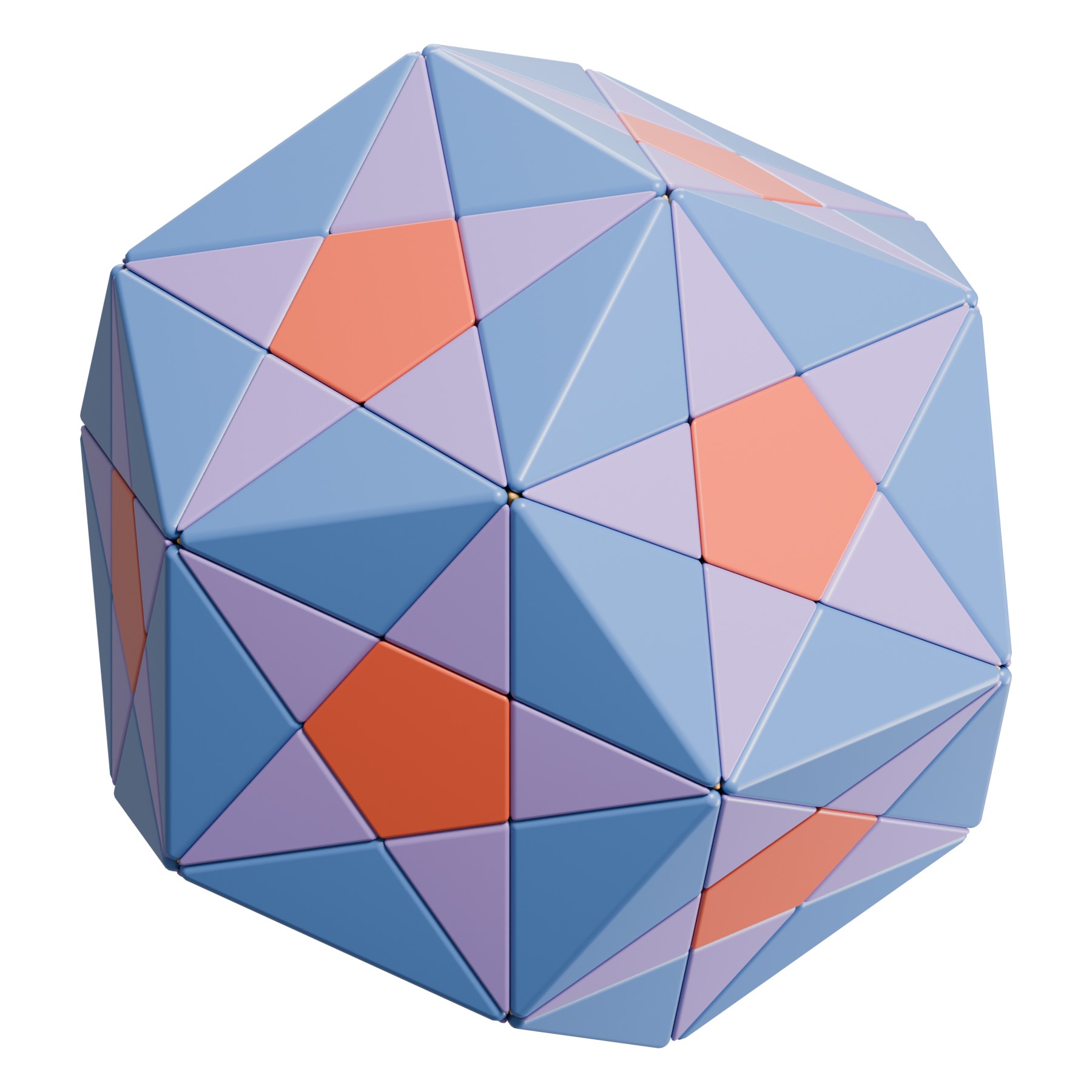}\hfill
  \fig[0.325\linewidth]{figures/dodeca_phi3_exploded.jpg}\hfill
  \fig[0.325\linewidth]{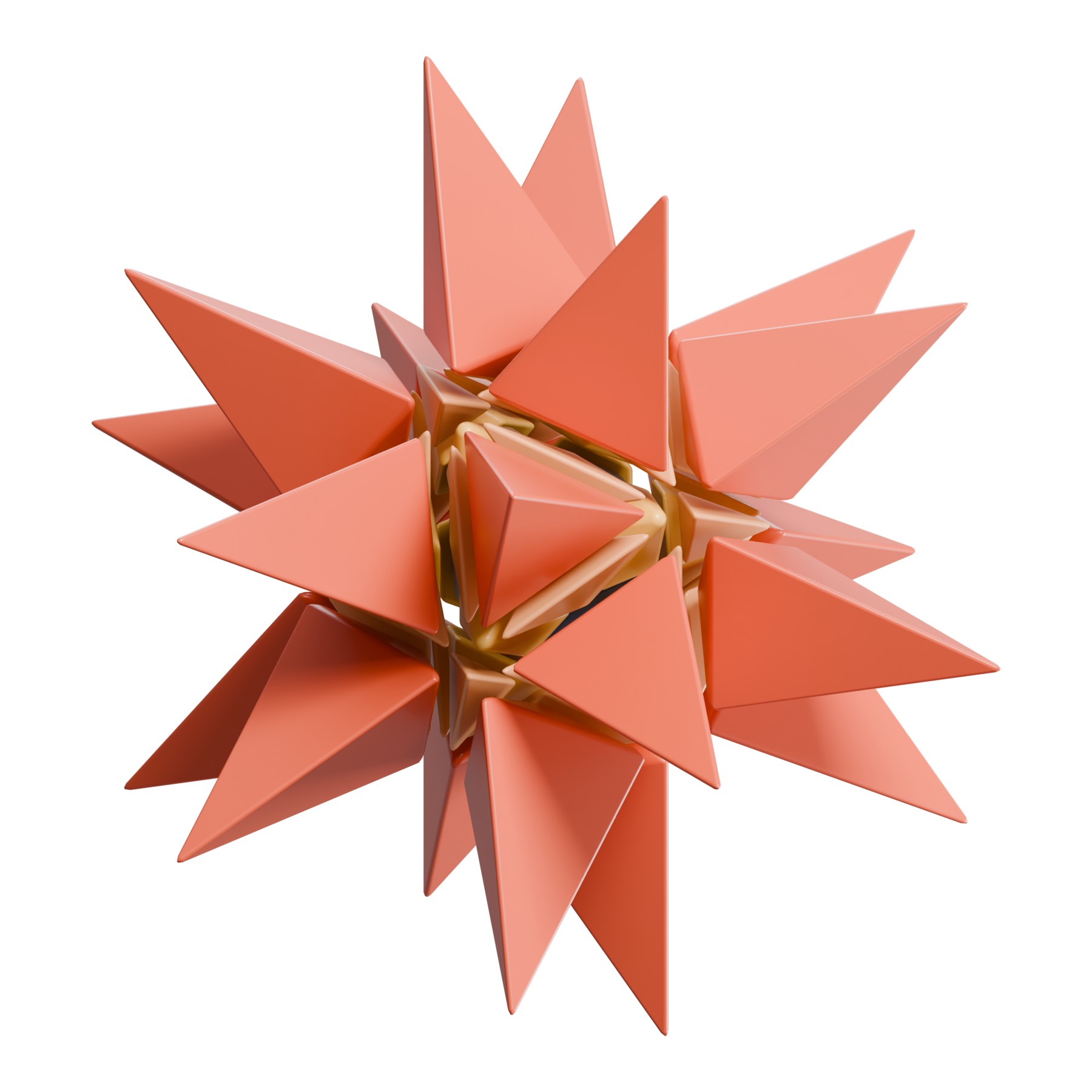}
  \caption{The stellation dissection of the dodecahedron: assembled, exploded (coloured by class), and
  exploded with the $63$ pieces of the great stellated dodecahedron in red and the core in blue.}
  \label{fig:dodeca}
\end{figure}

\begin{figure}[t]
  \centering
  \begin{tikzpicture}
    \node[anchor=south west,inner sep=0] (img) at (0,0) {\fig[0.9\linewidth]{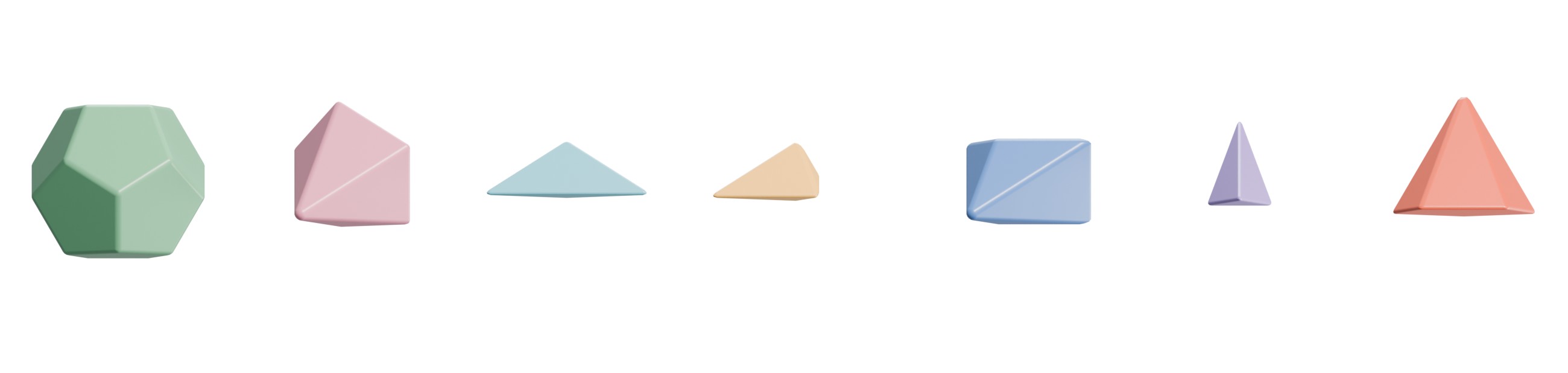}};
    \begin{scope}[x={(img.south east)},y={(img.north west)}]
      \IfFileExists{figures/dodeca_classes_labels.tex}{\input{figures/dodeca_classes_labels.tex}}{}
    \end{scope}
  \end{tikzpicture}
  \caption{One representative of each of the seven pieces, in the order of
  Table~\ref{tab:dodeca} and scaled to a common size (class number and multiplicity below each
  piece).}
  \label{fig:dodeca_classes}
\end{figure}

\subsection{The vertices of the pieces are module points}\label{sec:vert3}
The vertices of the pieces are intersections of three of the $24$ planes (twelve faces, twelve
cuts), that is, Cramer quotients $X/W$ with $X\in\Zphi^3$ and $W\in\Zphi$ a $3\times3$ determinant
of normals.  Nothing in this forces the quotient to be integral.  It is.

\begin{result}[Appendix~\ref{app:checks}, block~E]\label{res:vert3}
The $24$ planes meet in $112$ distinct points inside $B$, which are exactly the $112$ distinct
vertices of the $165$ pieces.  All $112$ have coordinates in $\Zphi$ with no denominators, and all
$112$ lie in the vertex module $\M$ of Result~\ref{res:modules}.  Of them $20$ are vertices of $B$,
$60$ lie on its boundary and $32$ are interior; $80$ are simple (three planes), $12$ lie on five
planes (one on each five-fold axis) and $20$ on six.  Every determinant $W$ that occurs has
$|\Nm(W)|=4$, so $W$ is $2$ times a unit and the quotient $X/W$ lies a priori in $\tfrac12\Zphi^3$;
the halves always cancel.
\end{result}

So the dissection of the dodecahedron is a dissection \emph{of the module $\M$}: its pieces are
lattice polytopes of $\M$ in the strict sense, and the whole figure, pieces and star included, can
be drawn on the twenty-point module without ever leaving it.  The same question has a different
answer for one of the two four-dimensional cuts (Result~\ref{res:vert4}).

\subsection{The seven classes as orbits of sign vectors}\label{sec:orbits3}
Section~\ref{sec:signvectors} names each piece by its sign vector against the twelve cut planes.
The group $I_h$ acts on the twelve planes by permutation (the set of cut planes is invariant),
hence on sign vectors by permuting coordinates, and the \emph{canonical form} of a piece is the
lexicographically least vector in its orbit.  Everything is exact and combinatorial once the
$120$ permutations are known, and they come from the exact matrices of Section~\ref{sec:modules}.

\begin{result}[Appendix~\ref{app:checks}, block~F]\label{res:orbits}
The $165$ pieces have $165$ distinct sign vectors (every centroid is at least $0.045$ from every
cut plane, so no sign is in doubt).  Under the $120$ permutations of the twelve cut planes induced
by $I_h$ they fall into exactly $7$ orbits, of sizes $1,12,12,20,30,30,60$, each equal to one of the
classes of Table~\ref{tab:dodeca}.  Writing $w(\sigma)$ for the number of $+$ signs:
\begin{center}\small
\begin{tabular}{@{}lrll@{}}\toprule
piece & size & $w$ & canonical sign vector\\\midrule
core & 1 & 0 & \lstinline[columns=fixed]|------------|\\
small tent & 12 & 1 & \lstinline[columns=fixed]|-----------+|\\
small tetrahedron & 30 & 2 & \lstinline[columns=fixed]|----------++|\\
spike & 20 & 3 & \lstinline[columns=fixed]|---------+++|\\
wedge & 30 & 4 & \lstinline[columns=fixed]|--------++++|\\
tetrahedron $T_2$ & 60 & 5 & \lstinline[columns=fixed]|-------+++++|\\
tent & 12 & 6 & \lstinline[columns=fixed]|-----+++++-+|\\\bottomrule
\end{tabular}
\end{center}
The weights $0,1,2,3$ are the four stellation layers (core, small stellated, great, great
stellated dodecahedron) and the gap fillers have weights $4,5,6$.  Sampling $4\cdot10^5$ uniform
random points of $B$ and recording their sign vectors recovers all $165$ names and nothing else.
\end{result}

This replaces, for the dodecahedron, the floating-point classification of
Section~\ref{sec:orbits} by an exact one.  Of the $2^{12}=4096$ conceivable sign vectors $165$
are realised; deciding whether a given vector is realised is a linear feasibility problem over
$\Qphi$, and listing all realised vectors is exactly what the cutting algorithm does.  So the cut
is unavoidable for the \emph{catalogue} and avoidable for \emph{membership}, which is what
rendering, sampling and ray-marching need.

\subsection{Other depths}\label{sec:depths}
Nothing forces the twelve planes to pass through the vertices.  Cutting $B$ by the planes parallel
to its faces at signed distance $t$ times the inradius, $0<t<1$, gives a one-parameter family in
which the stellation dissection is one member.

\begin{result}[phase diagram]\label{res:phase}
The number of pieces is $185$ for $0<t<\varphi^{-3}$, $165$ for $\varphi^{-3}\le t<1/\sqrt5$, $93$
for $1/\sqrt5\le t<\varphi^{-1}$ and $63$ for $\varphi^{-1}\le t<1$ (and $32$ at $t=0$, where the
twelve planes become six).  The combinatorial type is constant on each interval: we checked it at
$2\,000$ equally spaced rational depths and exactly at the three critical depths.  At every depth
twelve points of the arrangement lie on five cut planes at once (one on each five-fold axis), so
the arrangement is never simple.
\end{result}

The critical depths are incidence events, and the ring predicts them: along a face normal the
vertices of $B$ lie at heights $\pm1,\pm\varphi^{-3}$ of the inradius, the face centres at
$\pm1,\pm1/\sqrt5$ (adjacent unit normals have inner product $1/\sqrt5$) and the edge midpoints at
$\pm1,\pm\varphi^{-1},0$ (Appendix~\ref{app:checks}, block~D).  At $t=\varphi^{-3}$ the cut planes
reach the $20$ vertices; at $t=1/\sqrt5$ the five-fold concurrency points reach the $12$ face
centres; at $t=\varphi^{-1}$ pairs of cut planes meet the $30$ edge midpoints.  Each of the three
numbers is a golden scale (Table~\ref{tab:scales}), and they are ordered
$\varphi^{-3}<1/\sqrt5<\varphi^{-1}$ by the exact sign test.
For $t<\varphi^{-3}$ the $185$ pieces fall into eight classes ($30$ eight-vertex pieces, $12$
pentagonal pyramids, the core, $20$ triangular bipyramids, $60$ triangular prisms, $12$ smaller
pentagonal pyramids, $30$ tetrahedra and $20$ small bipyramids): the $60$ prisms are what the
tetrahedra of the stellation dissection become when the planes no longer pass through the vertices.
Figure~\ref{fig:regimes} shows the four regimes.

\begin{figure}[t]
  \centering
  \fig[0.245\linewidth]{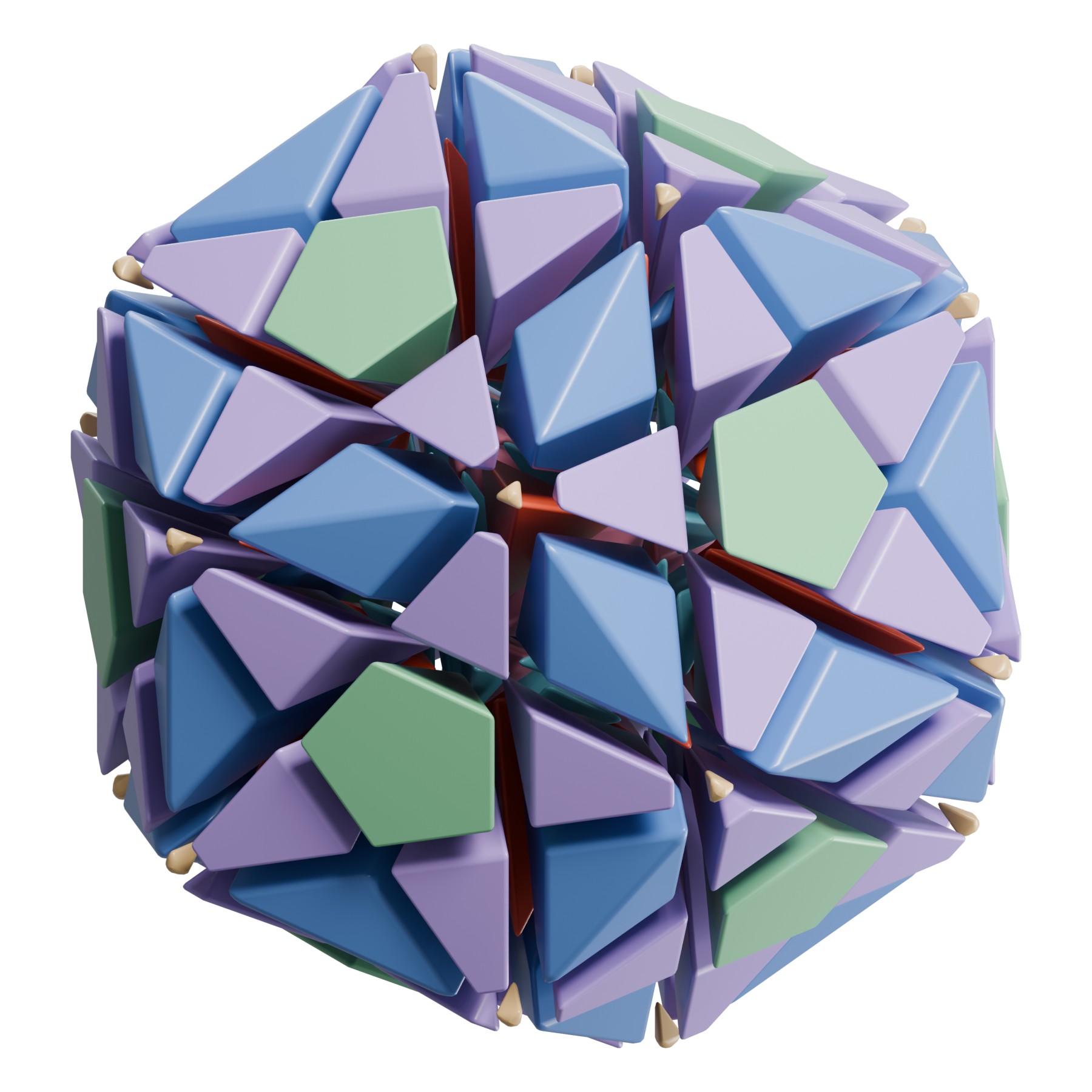}\hfill
  \fig[0.245\linewidth]{figures/dodeca_phi3_exploded.jpg}\hfill
  \fig[0.245\linewidth]{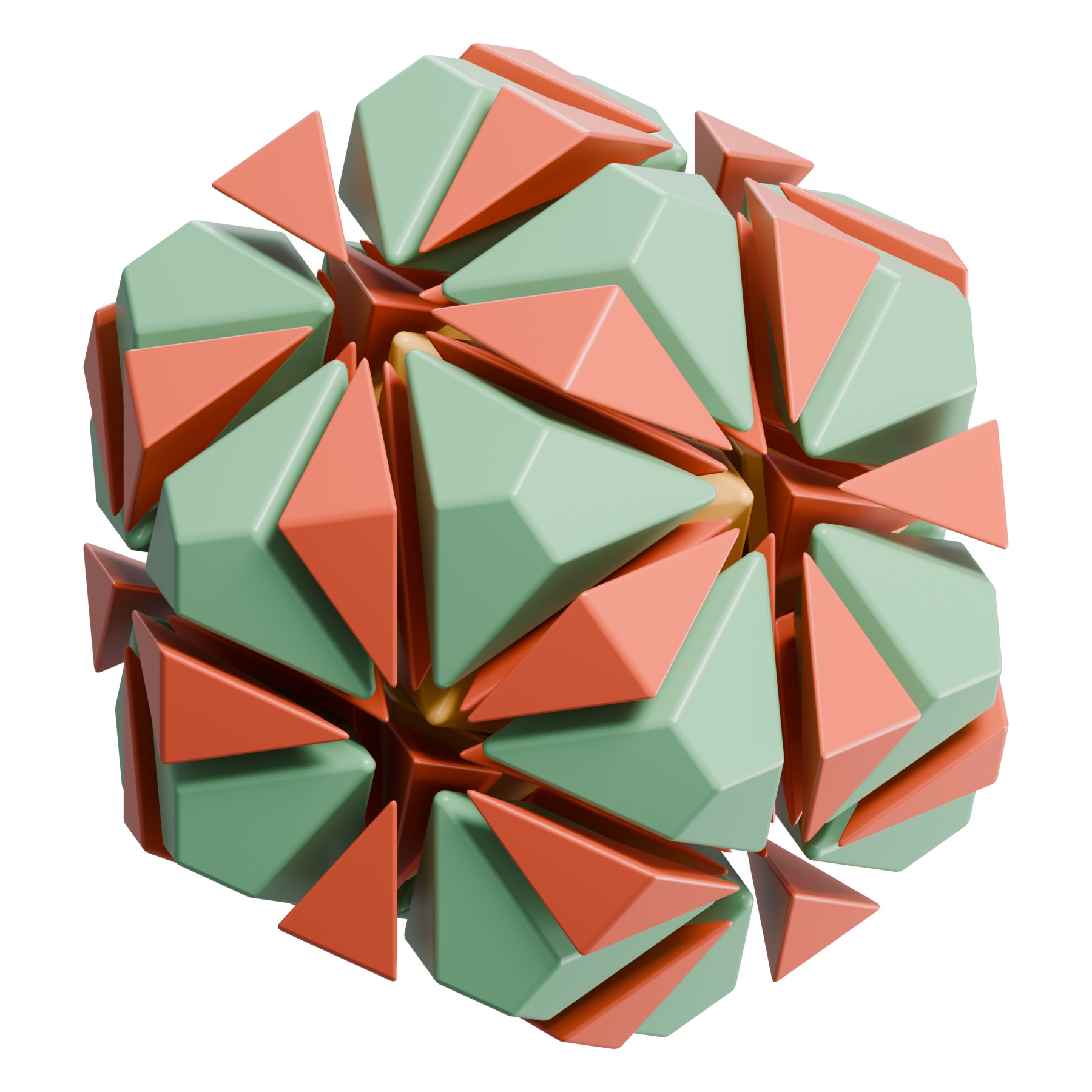}\hfill
  \fig[0.245\linewidth]{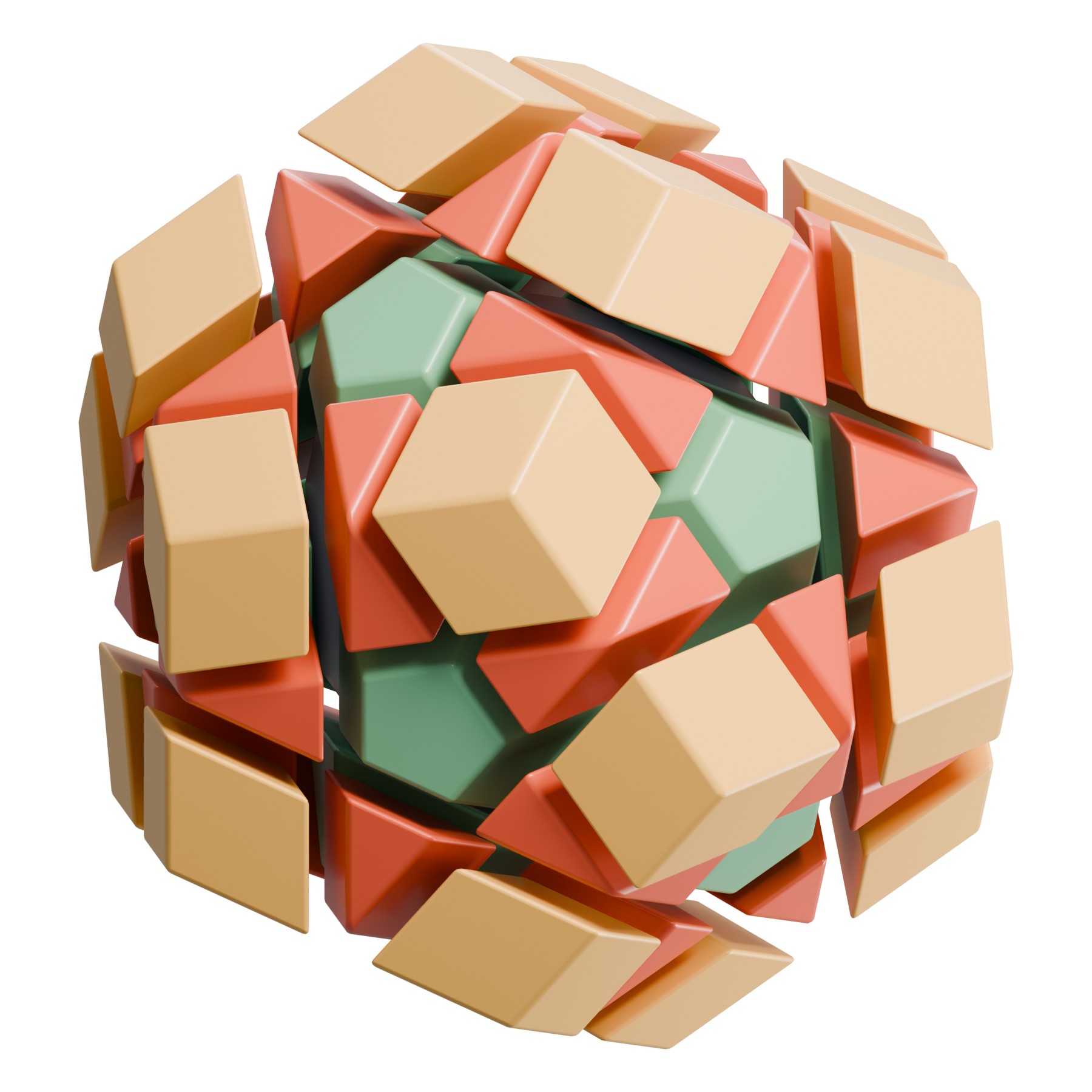}
  \caption{The four regimes of the face-parallel cut of the dodecahedron, exploded:
  $t=1/5$ ($185$ pieces), $t=\varphi^{-3}$ (the stellation dissection, $165$), $t=1/\sqrt5$ ($93$),
  $t=\varphi^{-1}$ ($63$).}
  \label{fig:regimes}
\end{figure}

\subsection{Why exact arithmetic}\label{sec:float}
The five-fold concurrencies are the reason this innocent-looking computation defeats floating
point.  We ran the algorithm of Section~\ref{sec:algorithm} in double precision with raw signs:
it produced $18$ non-manifold pieces at $t=1/5$, crashed at $t=\varphi^{-3}$ and reported $14$
pieces at $t=1/\sqrt5$.  With a tolerance $\varepsilon=10^{-14}$ it crashed at $t=17/240$;
tolerances between $10^{-12}$ and $10^{-4}$ work at the depths we tried but are input-dependent,
and a large tolerance erases genuine pieces: at $t=236\,067\,977/10^9$, within
$5\cdot10^{-10}$ of $\varphi^{-3}$, the exact computation correctly finds $20$ slivers of volume
$5.6\cdot10^{-28}$.  (This project began with a floating-point ``blade'' program written by one
of us in 2022; revived, it reproduced its archived $235$ pieces byte for byte, but none of them is
watertight and a single typo in four plane normals had created $36$ spurious slivers.  The
correct count at its depth, $t\approx0.205$, is $185$.)  In the language of
Section~\ref{sec:sign}: a tolerance guesses whether $m+n\varphi$ is zero, and the ring knows.  In
dimension four the number of degenerate incidences is far larger (up to $20$ hyperplanes through
one vertex, Result~\ref{res:vert4}), and we did not attempt a floating-point computation at all.

\section{The 120-cell and its cut hyperplanes}\label{sec:120cell}

This section contains everything needed to set up the four-dimensional cut.

\paragraph{Vertices.}  The 120-cell $B$ of circumradius $2\sqrt2$ and edge $2\varphi^{-2}$ has the
$600$ vertices
\begin{equation}\label{eq:vertices}
\begin{array}{ll}
\text{all permutations of} & (0,0,\pm2,\pm2),\ (\pm1,\pm1,\pm1,\pm\sqrt5),\\
  & (\pm\varphi^{-2},\pm\varphi,\pm\varphi,\pm\varphi),\ (\pm\varphi^{-1},\pm\varphi^{-1},\pm\varphi^{-1},\pm\varphi^2),\\[2pt]
\text{even permutations of} & (0,\pm\varphi^{-2},\pm1,\pm\varphi^2),\ (0,\pm\varphi^{-1},\pm\varphi,\pm\sqrt5),\\
  & (\pm\varphi^{-1},\pm1,\pm\varphi,\pm2),
\end{array}
\end{equation}
($24+64+64+64+96+96+192=600$) \cite{Coxeter1973,Wiki120}.  All coordinates lie in $\Zphi$, by
\eqref{eq:powers}; their $\Zphi$-span is the module $\MV$ of index $16$ in $\Zphi^4$.

\paragraph{Facets.}  The $120$ facet hyperplanes are $n\cdot x\le2\varphi^2$, where $n$ runs over
\begin{equation}\label{eq:normals}
  \text{all permutations of } (\pm2,0,0,0),\ (\pm1,\pm1,\pm1,\pm1),\qquad
  \text{\emph{odd} permutations of } (\pm\varphi,\pm1,\pm\varphi^{-1},0)
\end{equation}
($8+16+96=120$), the vertices of the dual 600-cell scaled by $2$.  All have $|n|=2$, so the
inradius of $B$ is $\varphi^2$.  The parity matters: with the \emph{even} permutations in
\eqref{eq:normals}, $96$ of the hyperplanes touch $B$ at single vertices instead of supporting
facets (we learned this the hard way).  Correctness is checked exactly: every vertex lies on
exactly $4$ facets, every facet contains exactly $20$ vertices, and $n\cdot v\le2\varphi^2$ for all
$600\times120$ pairs (block~I).  The vectors $n/2$ are the $120$ unit icosians of
Result~\ref{res:modules}.

\paragraph{Vertex layers.}  For the normal $n=(2,0,0,0)$ the heights $n\cdot v/|n|=v_1$ of the
vertices take the values $0,\pm\varphi^{-2},\pm\varphi^{-1},\pm1,\pm\varphi,\pm2,\pm\sqrt5,\pm\varphi^2$,
which divided by the inradius $\varphi^2$ give the layers \eqref{eq:layers}: in particular
$\varphi^{-2}/\varphi^2=\varphi^{-4}$ and $\sqrt5/\varphi^2=1-\varphi^{-4}$.

\paragraph{Cut hyperplanes.}  The star cut uses the $120$ hyperplanes
\begin{equation}\label{eq:starcut}
  n\cdot x=\varphi^{-4}\cdot2\varphi^2=2\varphi^{-2}=4-2\varphi ,
\end{equation}
the facet hyperplanes of the core $C=\varphi^{-4}B$; the hereditary cut uses
\begin{equation}\label{eq:heredcut}
  n\cdot x=(1-\varphi^{-4})\cdot2\varphi^2=2\sqrt5 .
\end{equation}
Each star-cut hyperplane contains the $20$ vertices of $B$ in its layer at height $\varphi^{-4}$
(for $n=(2,0,0,0)$ these are the vertices with $x_1=\varphi^{-2}$); they form a regular
dodecahedron $\varphi^2$ times larger than a facet of $B$, the dodecahedron of
Result~\ref{res:depth}.  Each hereditary-cut hyperplane contains the $20$ vertices at height
$1-\varphi^{-4}$ ($x_1=\sqrt5$), which form a dodecahedron congruent to a facet.  For a general
depth $t$ the cut hyperplanes are $n\cdot x=2\varphi^2t$, and for $t\in\Qphi$ all coefficients can
be cleared to $\Zphi$ by multiplying with the conjugate of the denominator
(Section~\ref{sec:ring}).

\paragraph{The two cores in the module.}  The star core $C=\varphi^{-4}B$ has its vertices in
$\varphi^{-4}\MV=\MV$: it is a copy of $B$ on the same module, and so is every further
$\varphi^{-4k}B$.  The hereditary core $(3\varphi-4)B=(\sqrt5\varphi^{-2})B$ has its vertices in
$\sqrt5\,\MV$, a sublattice of index $5^4=625$ (Result~\ref{res:modules}); iterating the
hereditary construction on its own core produces vertices in $5^{-k}\MV$ and leaves the module.
This is the difference the last column of Table~\ref{tab:bifurcation} records, and it is decided
by one norm.

\paragraph{Restriction to a facet.}  A facet $F$ of $B$ is a regular dodecahedron of edge
$2\varphi^{-2}$, midradius $1$, inradius $\varphi^2/\sqrt{\varphi+2}\approx0.851$ and circumradius
$\sqrt{6-3\varphi}\approx1.071$, lying at distance $\varphi^2$ from the centre.  Of the $120$
hyperplanes, $42$ meet the interior of a given facet in the star cut: $12$ of them are parallel to the faces of $F$
at $5-2\sqrt5=\sqrt5\varphi^{-3}\approx0.528$ of its inradius and $30$ are orthogonal to its two-fold axes at
$\varphi^{-2}$ of its midradius; together they cut $F$ into $3\,655$ pieces.  In the hereditary cut
exactly $12$ hyperplanes meet $F$, in the twelve planes of $F$'s own stellation dissection, and $F$ is cut
into $165$ pieces with the signature of Result~\ref{res:seven} ($1\times V_{20}$, $90\times V_4$,
$20\times V_5$, $54\times V_6$).  Both statements were verified by restricting the hyperplanes to
the facet's $3$-space and running the three-dimensional code independently.

\section{The algorithm}\label{sec:algorithm}

The cells of $B$ cut by hyperplanes $h_1,\dots,h_m$ are the cells of a hyperplane arrangement
restricted to a convex polytope.  Arrangements are classically built incrementally
\cite{Edelsbrunner1986} and counted by Zaslavsky's theory \cite{Zaslavsky1975}: $120$ hyperplanes
in general position in $\R^4$ bound at most $8\,503\,671$ cells, $12$ planes in $\R^3$ at most
$299$.  A binary tree of half-space splits is a binary space partition
\cite{Thibault1987}; representing a convex cell by both its inequalities and its vertices is the
double description method \cite{Motzkin1953,Fukuda1996}, and cutting convex cells by planes is
also the core operation of Voronoi codes such as Voro++ \cite{Rycroft2009}.  The difficulty is
robustness: rounding errors in orientation tests produce inconsistent topology \cite{Yap1997}, and
the modern remedies are exact or filtered predicates \cite{Shewchuk1997,Attene2020,Cherchi2020,Levy2024}
and \emph{plane-based} representations in which every point is remembered as the intersection of
input planes \cite{Campen2010,NehringWirxel2021,Trettner2022}.  Our setting allows a fully exact
method with no filter at all, because every coefficient is in $\Zphi$ and Section~\ref{sec:sign}
decides every sign.

\subsection{Membership without cutting}
Before the arrangement is built, the ring already answers the question a renderer asks: which
piece contains the point $x$?  By Section~\ref{sec:signvectors} the answer is the sign vector.
\begin{lstlisting}[caption={Naming the piece that contains $x$: $m$ exact signs, no cutting, one thread per point.},label=lst:member]
procedure PIECE(x):                 # x in Q(sqrt5)^d as pairs of rationals
    if some facet inequality of B fails at x: return OUTSIDE
    sigma <- ()
    for k in 1..m: sigma.append( sign(n[k].x - c[k]) )   # the exact sign of Section 2.5
    return sigma                    # the piece; its class is the canonical form of sigma
\end{lstlisting}
The slicer below labels every cell by this vector; the point of writing the procedure down is that
the label can be computed \emph{before} the cell, in parallel, with no state.  What the cut adds
is the list of realised labels, the vertices (Cramer quotients $X/W$), the face lattices and the
volumes.

\subsection{Exact signs in \texorpdfstring{$\Zphi$}{Z[phi]}}
A number $m+n\varphi$ is stored as the pair $(m,n)$ of $128$-bit integers; sums and products stay
in the ring, $(m+n\varphi)(m'+n'\varphi)=(mm'+nn')+(mn'+nm'+nn')\varphi$, and its sign is
\eqref{eq:sign}.  The Galois conjugate $\overline{m+n\varphi}=(m+n)-n\varphi$ and the norm
$\Nm(m+n\varphi)=m^2+mn-n^2$ clear denominators.  Nothing else from the ring is needed at run
time; the modules of Section~\ref{sec:modules} are used afterwards, to interpret the output.

\subsection{Plane-based double description}\label{sec:dd}
Hyperplanes $n\cdot x=c$ with $n\in\Zphi^d$, $c\in\Zphi$ are stored once in a global list.  A
vertex is the intersection of $d$ of them and is stored in homogeneous coordinates
$(X,W)\in\Zphi^{d+1}$ computed by Cramer's rule \emph{from the original hyperplanes}; nothing
accumulates from one cut to the next, and the algebraic size of a vertex is bounded once and for
all.  Each vertex carries its \emph{incidence set} $Z(v)$, the set of hyperplanes through it, as a
bitset.  A cell is a list of oriented half-spaces together with its vertices.

To split a cell $P$ by a hyperplane $h$, evaluate the exact sign of $h$ at every vertex.  If no
vertex is strictly on one side, $P$ is unchanged.  Otherwise vertices with sign $0$ go to both
halves with $h$ added to their incidence set, and every \emph{edge} $uv$ with
$\sgn h(u)<0<\sgn h(v)$ contributes the new vertex $h\cap p_1\cap\dots\cap p_{d-1}$, where
$p_1,\dots,p_{d-1}$ are independent hyperplanes in $Z(u)\cap Z(v)$.  Edges are found without any
geometry:

\begin{lemma}[combinatorial adjacency \cite{Fukuda1996}]\label{lem:adj}
Two vertices $u\neq v$ of a $d$-polytope given by an inequality system are joined by an edge if
and only if $|Z(u)\cap Z(v)|\ge d-1$ and no third vertex $w$ satisfies $Z(u)\cap Z(v)\subseteq Z(w)$.
\end{lemma}

In dimension $3$ the first condition suffices, because two distinct supporting planes meet in a
line; in dimension $4$ it does not, since three hyperplanes through a $2$-plane are common, and the
second condition is essential.  Degenerate incidences, which are everywhere in these arrangements
(Results~\ref{res:phase} and \ref{res:vert4}), need no special treatment: a vertex on five planes
simply has five planes in its incidence set.

\subsection{Face lattice, checks, volume}
The face lattice of a cell is built top-down from the incidence sets: a subset of a $k$-face $F$
is a $(k-1)$-face if it is $F\cap\{v:p\in Z(v)\}$ for some hyperplane $p$ and has affine dimension
$k-1$.  (The rank test is the one place where we use floating point, on exactly computed
vertices; it never came close to its tolerance.)  Every cell is checked against the Euler relation
$\sum_{k<d}(-1)^kf_k=1-(-1)^d$ and for positive volume, computed recursively as
$\vol_k(F)=\tfrac1k\sum_G\operatorname{dist}(c_F,\operatorname{aff}G)\,\vol_{k-1}(G)$ over the
facets $G$ of $F$, and the volumes of all cells are summed and compared with $\vol B$.

\subsection{Streaming the binary space partition}
Cutting breadth-first, keeping all cells after each hyperplane, exhausts memory in dimension four
(more than $3$ million exact vertex records after $70$ of the $120$ cuts).  We traverse the tree
depth-first instead: a cell is split by cut $k$, its two children continue with cut $k+1$, and a
leaf (a cell that has met all $m$ hyperplanes) is analysed and discarded at once.  Memory is
proportional to the depth $m$.  A cell is identified by its sign vector
$\sigma\in\{\pm1\}^m$.  Parallelism comes from cutting breadth-first to a
frontier ($26$ cuts, $7\,346$ cells for the star cut) and distributing the frontier over threads.

\begin{lstlisting}[caption={The exact cut. All arithmetic in $\Zphi$; \texttt{sign} is \eqref{eq:sign}.},label=lst:algo]
procedure CUT(B, h[1..m]):
    P <- cell(B)               # half-spaces of B, its vertices with incidence sets
    DFS(P, 1)
procedure DFS(P, k):
    if k > m: emit(P); return  # face lattice, Euler check, volume, invariant
    s[v] <- sign(h[k](v)) for each vertex v of P
    if all s >= 0 or all s <= 0: DFS(P, k+1); return   # h[k] misses the interior
    N <- {v : s[v] <= 0}, Q <- {v : s[v] >= 0}         # zeros go to both sides
    for each pair (u, v) with s[u] < 0 < s[v]:
        Z <- Z(u) & Z(v)
        if |Z| < d-1 or (Z <= Z(w) for some third vertex w): continue   # not an edge
        w <- intersection of h[k] with d-1 independent planes of Z (Cramer)
        Z(w) <- Z + {h[k]}; add w to N and Q
    DFS((P + {h[k] <= 0}, N), k+1); DFS((P + {h[k] >= 0}, Q), k+1)
\end{lstlisting}

\subsection{Orbits}\label{sec:orbits}
The symmetry group of $B$ acts on the cells, and we want the orbits.  Two cells in the same
orbit are congruent and lie at the same distance from the centre; we group cells by the invariant
($f$-vector, multiset of facet sizes, volume to nine significant digits, a hash of the sorted
pairwise vertex distances, a hash of the sorted vertex radii), computed in double precision from
the exact vertices.  A necessary condition for a class to be one orbit is that its size divides the
order of the group, $120$ for $I_h$ and $14\,400$ for $\Hfour$; a coarser invariant (volume only)
failed this test for $28$ classes of the star cut, the invariant above passes it for all.  In
dimension three Result~\ref{res:orbits} confirms the classification exactly: canonical sign vectors
under the $120$ permutations give the same seven classes.  The same procedure in dimension four
needs the $14\,400$ permutations of the $120$ hyperplanes induced by $\Hfour$, which the exact
reflections of Result~\ref{res:modules} generate, and the sign vectors of the $\NSTAR$ cells,
which the slicer already stores; it would turn the counts $\NSTARCLASSES$ and $\NHEREDCLASSES$
into exact statements and is left for future work.

\subsection{Exploding a dissection}\label{sec:explode}
The pictures in this article move the pieces apart.  The following choice never makes two pieces
collide, in any dimension and also in sections.

\begin{proposition}\label{prop:explode}
Let the cells of an arrangement of hyperplanes $\hat n_k\cdot x=c_k$, $|\hat n_k|=1$, be translated
by $A\sum_k\sigma_k\hat n_k$, where $\sigma_k=\pm1$ is the side of the cell with respect to
hyperplane $k$ and $A>0$.  Any two cells $\sigma\neq\tau$ are then separated by a slab of width at
least $2A|m|$, where $m=\sum_{k\in S}\sigma_k\hat n_k$ and $S$ is the set of hyperplanes separating
them.  The same holds, with $m$ replaced by its orthogonal projection, for the sections of the
exploded cells by an affine subspace.
\end{proposition}
\begin{proof}
Put $f(x)=\sum_{k\in S}\sigma_k(\hat n_k\cdot x-c_k)$.  Each term is $\ge0$ on $\sigma$ and
$\le0$ on $\tau$, so $f\ge0$ on $\sigma$ and $f\le0$ on $\tau$.  The translation of $\sigma$
relative to $\tau$ is $2Am$, which raises $f$ by $2A|m|^2$ on $\sigma$; the exploded cells lie on
opposite sides of the parallel hyperplanes $f=0$ and $f=2A|m|^2$, at distance $2A|m|$.  Restricted
to an affine subspace, $f$ is affine with gradient the projection of $m$, and the argument repeats.
\end{proof}

\section{The 120-cell cut}\label{sec:results}

\subsection{The star cut}
\begin{result}[star cut]\label{res:star}
The 120-cell cut by the $120$ hyperplanes \eqref{eq:starcut} has $\NSTAR$ cells.  Every cell
satisfies $f_0-f_1+f_2-f_3=0$ and has positive volume; the volumes add up to $268.3281572997$,
against the exact $\vol B=\tfrac{15}{4}(105+47\sqrt5)\,a^4=268.3281572999\ldots$ for the edge
$a=2\varphi^{-2}$ (relative difference $10^{-12}$).  The cells fall into $\NSTARCLASSES$ classes
whose sizes all divide $14\,400$, with the histogram (orbit size : number of classes)
\ORBITHIST.  There are $\NSTARSHAPES$ shapes up to congruence (two pairs of classes are congruent
cells at different radii) and $\NSIMPLEXCLASSES$ classes of $4$-simplices, $\NSIMPLEXCELLS$ cells
in all.  The computation took $246$\,s on four threads and $290$\,MB of memory.
\end{result}

The complete catalogue is Appendix~\ref{app:catalogue}; Table~\ref{tab:star_top} lists the
twelve largest classes and Figure~\ref{fig:gallery} shows representatives of the largest
$24$.  The two largest pieces are the core $C=\varphi^{-4}B$ (one cell, volume
$\varphi^{-16}\vol B=0.1216$) and the $120$ \emph{caps}, one under each facet, with
$f=(53,190,180,43)$: $30$ facets with $6$ vertices, $12$ with $7$, and one $32$-vertex facet on
the boundary of $B$, which is the central piece of the facet pattern.  Beyond these the volumes
drop quickly: the third class has volume $0.0048$, and the smallest piece, a $4$-simplex of volume
$9\cdot10^{-8}$, is six orders of magnitude smaller than the cap.  Vertex counts of the classes
range from $5$ ($66$ classes) through $6$ ($97$), $7$ ($40$), $8$ ($25$), $9$ ($39$) and up to
$64$, with the core at $600$.

\begin{table}[t]
  \centering\small
  \begin{tabular}{rrllr}
\toprule
class & cells & $f$-vector & facets (vertex counts, with multiplicity) & $4$-volume\\
\midrule
0 & 120 & $(53,190,180,43)$ & $6^{30} 7^{12} 32$ & 0.141153 \\
1 & 1 & $(600,1200,720,120)$ & $20^{120}$ & 0.121581 \\
2 & 3\,600 & $(19,46,41,14)$ & $4^{4} 7^{4} 8^{4} 10^{2}$ & 0.004837 \\
3 & 1\,440 & $(8,22,25,11)$ & $4^{10} 7$ & 0.004187 \\
4 & 2\,400 & $(26,72,63,17)$ & $6^{3} 7^{3} 8^{5} 9^{6}$ & 0.003717 \\
5 & 3\,600 & $(17,43,42,16)$ & $4^{4} 5^{2} 6^{4} 8^{6}$ & 0.003324 \\
6 & 2\,400 & $(6,15,18,9)$ & $4^{9}$ & 0.002512 \\
7 & 2\,400 & $(15,37,33,11)$ & $6^{6} 7^{3} 8^{2}$ & 0.002426 \\
8 & 1\,440 & $(19,53,50,16)$ & $5^{5} 7^{11}$ & 0.002283 \\
9 & 1\,440 & $(19,54,55,20)$ & $4^{5} 5^{5} 7^{10}$ & 0.002138 \\
10 & 3\,600 & $(16,36,32,12)$ & $6^{8} 8^{4}$ & 0.001911 \\
11 & 600 & $(16,48,48,16)$ & $5^{12} 8^{4}$ & 0.001812 \\
\bottomrule
\end{tabular}

  \caption{The twelve largest of the $\NSTARCLASSES$ orbits of the star cut of the 120-cell
  (facet entries $k^m$: $m$ facets with $k$ vertices).}\label{tab:star_top}
\end{table}

\begin{figure}[t]
  \centering
  \begin{tikzpicture}
    \node[anchor=south west,inner sep=0] (img) at (0,0) {\fig[\linewidth]{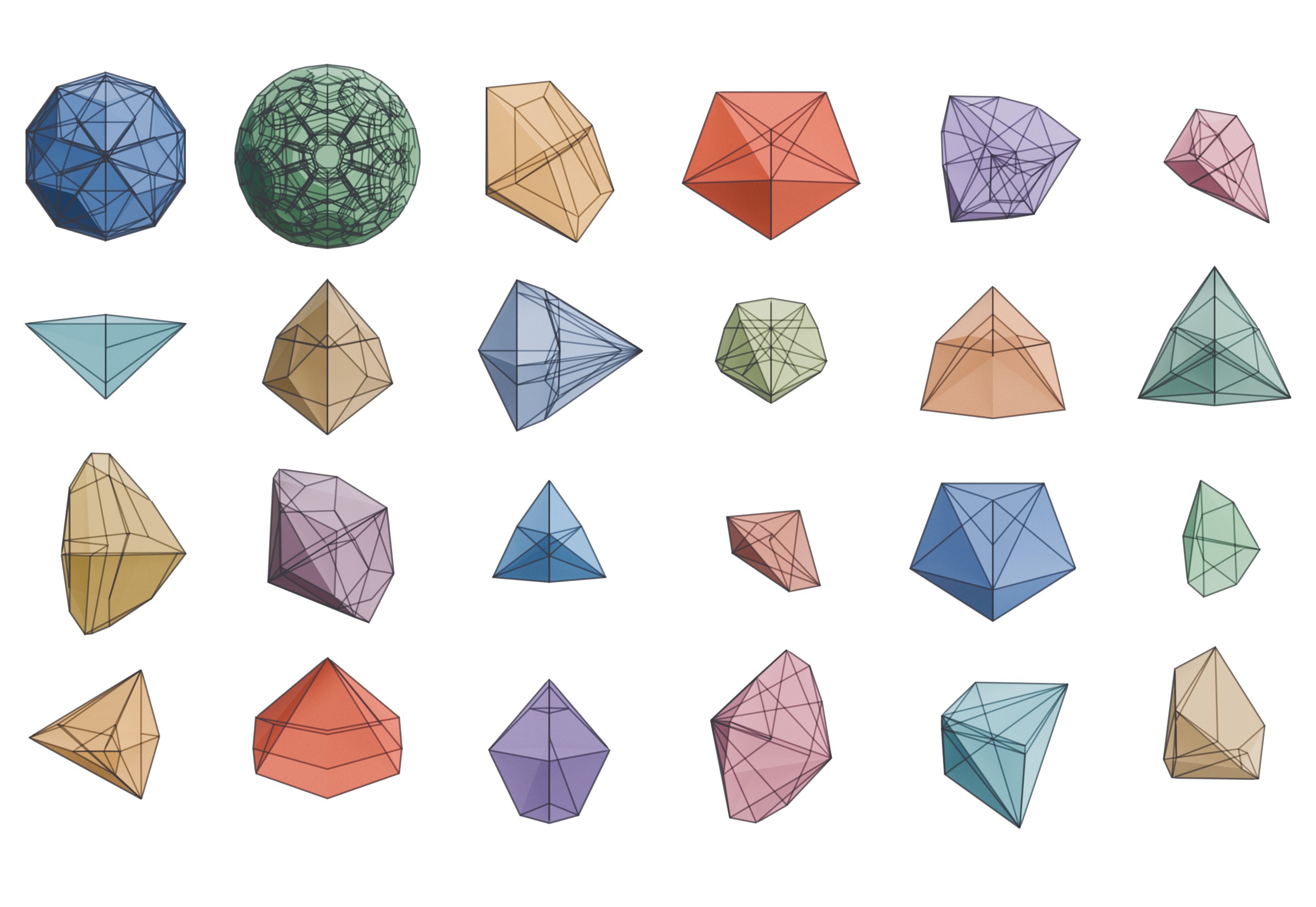}};
    \begin{scope}[x={(img.south east)},y={(img.north west)}]
      \IfFileExists{figures/cell120_gallery_labels.tex}{\input{figures/cell120_gallery_labels.tex}}{}
    \end{scope}
  \end{tikzpicture}
  \caption{The $24$ largest classes of the star cut, one four-dimensional cell each, drawn in
  perspective from a point outside the cell along its centroid direction and scaled to a common
  size (class number and orbit size under each piece).}
  \label{fig:gallery}
\end{figure}

\paragraph{The facet pattern.}  The $42$ hyperplanes that meet a facet $F$ cut it into $3\,655$
three-dimensional pieces (Figure~\ref{fig:facets}, left), each the facet of exactly one cell of
the cut; the $32$-vertex central piece belongs to the cap.  This is the pattern one would see
on the surface of a physical model.

\paragraph{Peeling the facet pattern.}  The pattern is layered (Figure~\ref{fig:peel}).  Measured
by the distance of a piece's centroid from the centre of the facet, the $3\,655$ pieces occupy a
central piece, a shell of $1\,882$ pieces at $0.41$--$0.71$ of the outermost centroid distance,
and an outer shell of $1\,772$ pieces beyond $0.75$ (there are no centroids between $0.71$ and
$0.75$).  Cutting the pattern in half along a mirror plane of the facet, or removing the outer
pieces, shows how they fit: the central piece, the facet of the cap, has $32$ vertices and $30$
quadrilateral faces and is a rhombic triacontahedron (the dual of the icosidodecahedron, with the
icosahedral symmetry of the facet); on each of its faces sits a square pyramid (class $31$,
$30$ pieces), then come $60$ six-vertex pieces of class $61$, and the remaining layers nest around
them.  The same peeling applied to the central section of the
120-cell (Figure~\ref{fig:peel_section}) shows the same organisation one dimension up: the
central section of the inner 120-cell $C$ is alone within $0.15$ of the outermost centroid
distance, a closed first shell of $622$ pieces follows at $0.15$--$0.18$, and the layers
beyond are dense ($1\,595$ pieces within $0.20$, $4\,809$ within $0.25$, $13\,411$ within
$0.40$).

\begin{figure}[p]
  \centering
  \fig[0.32\linewidth]{figures/cell120_facet.jpg}\hfill
  \fig[0.32\linewidth]{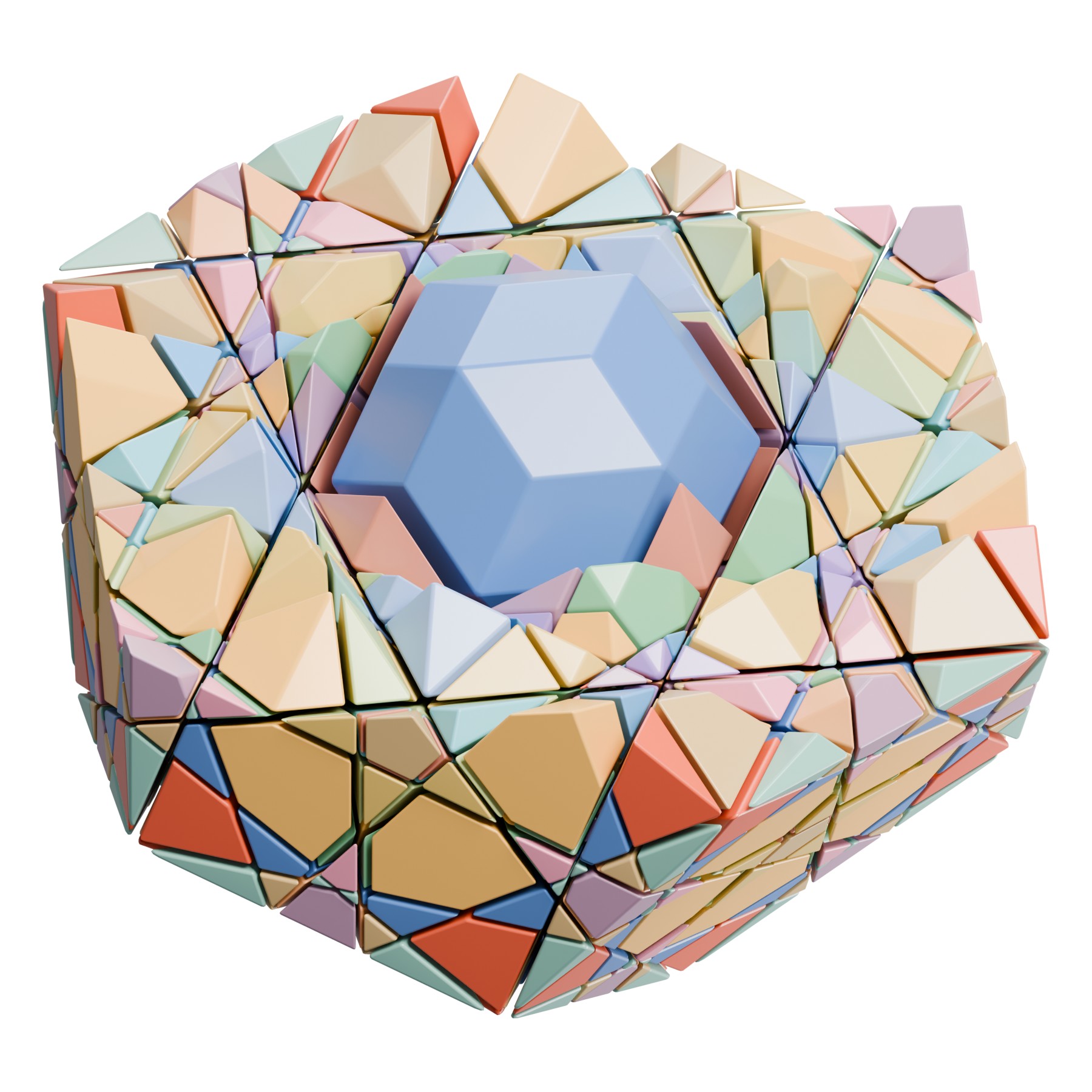}\hfill
  \fig[0.32\linewidth]{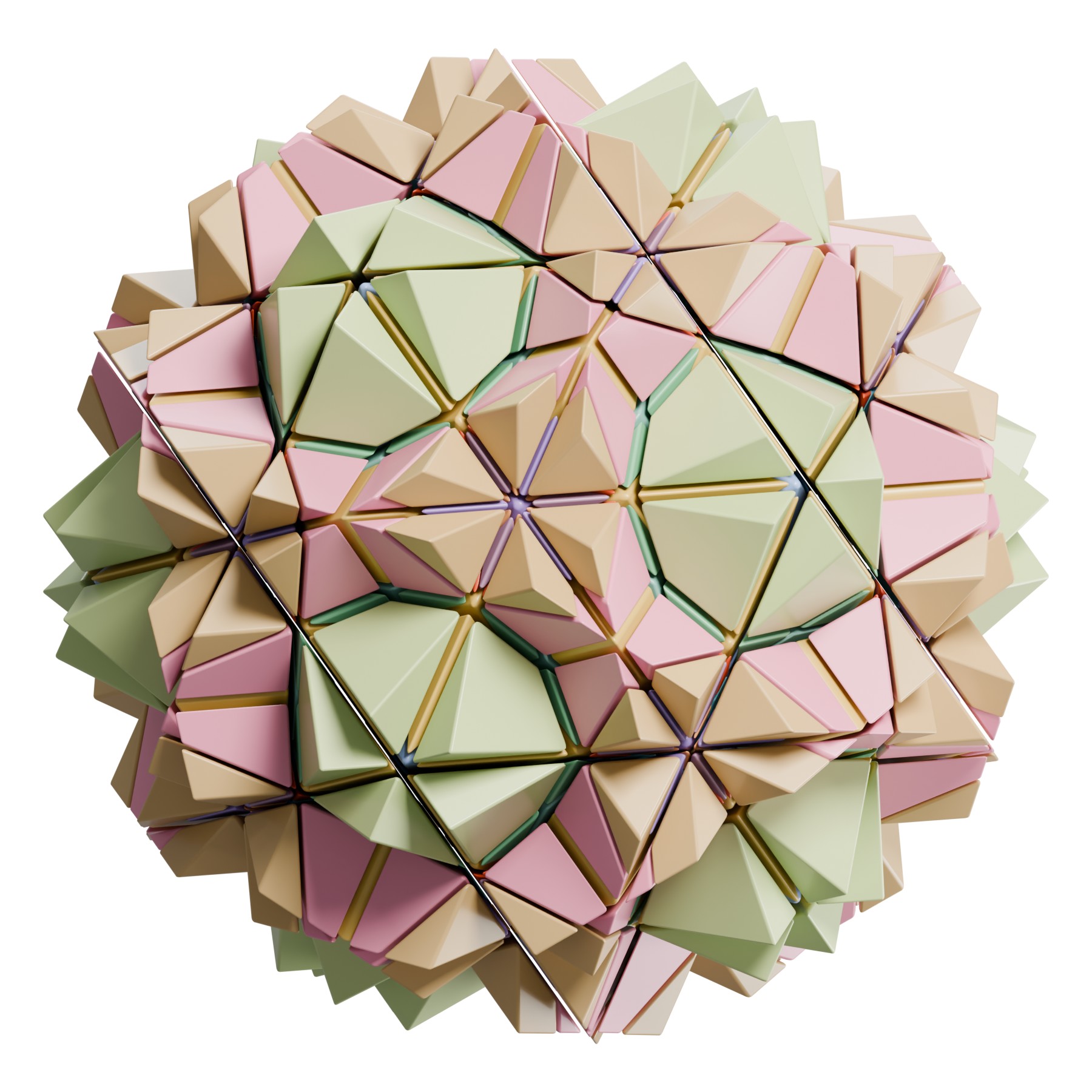}\\[3pt]
  \fig[0.32\linewidth]{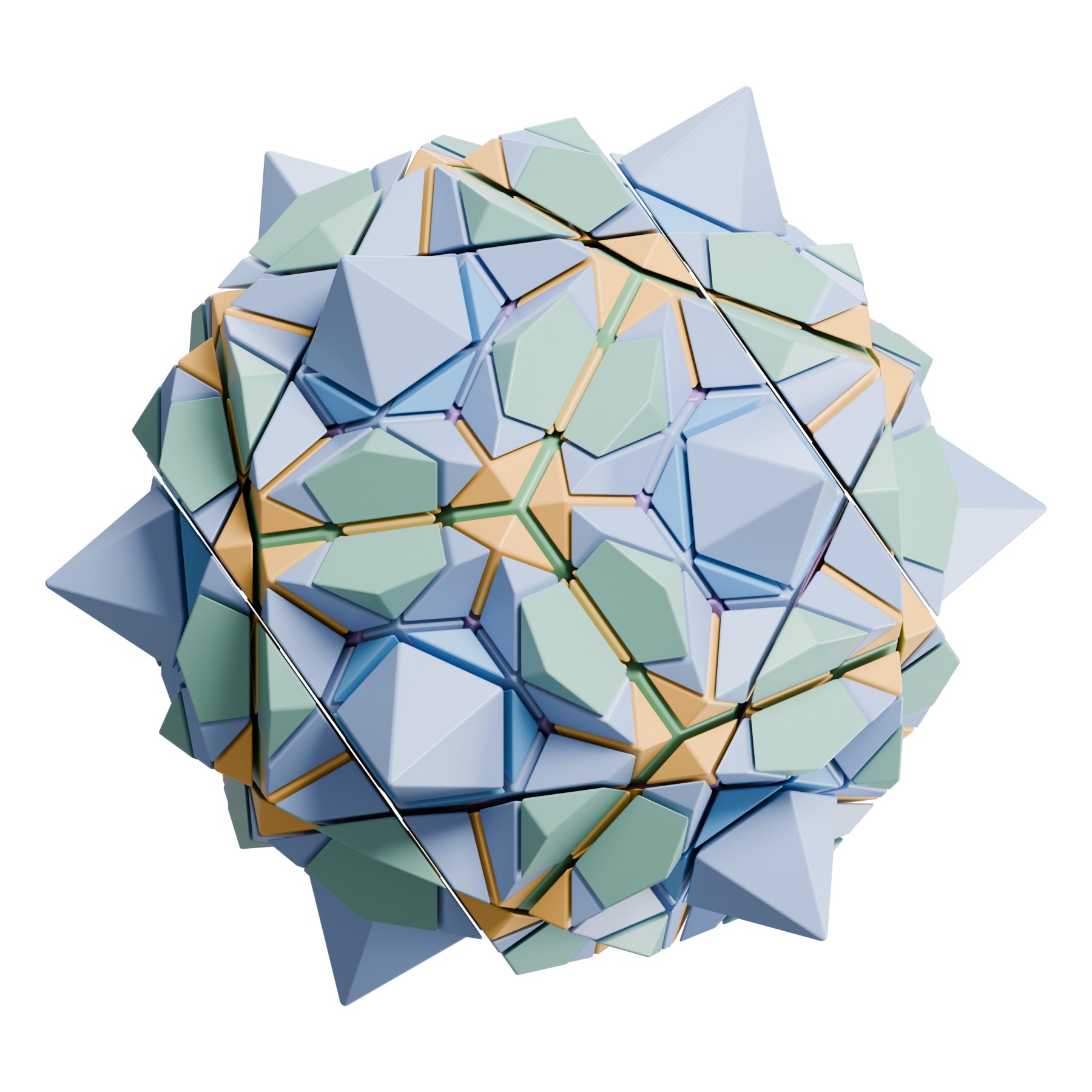}\hfill
  \fig[0.32\linewidth]{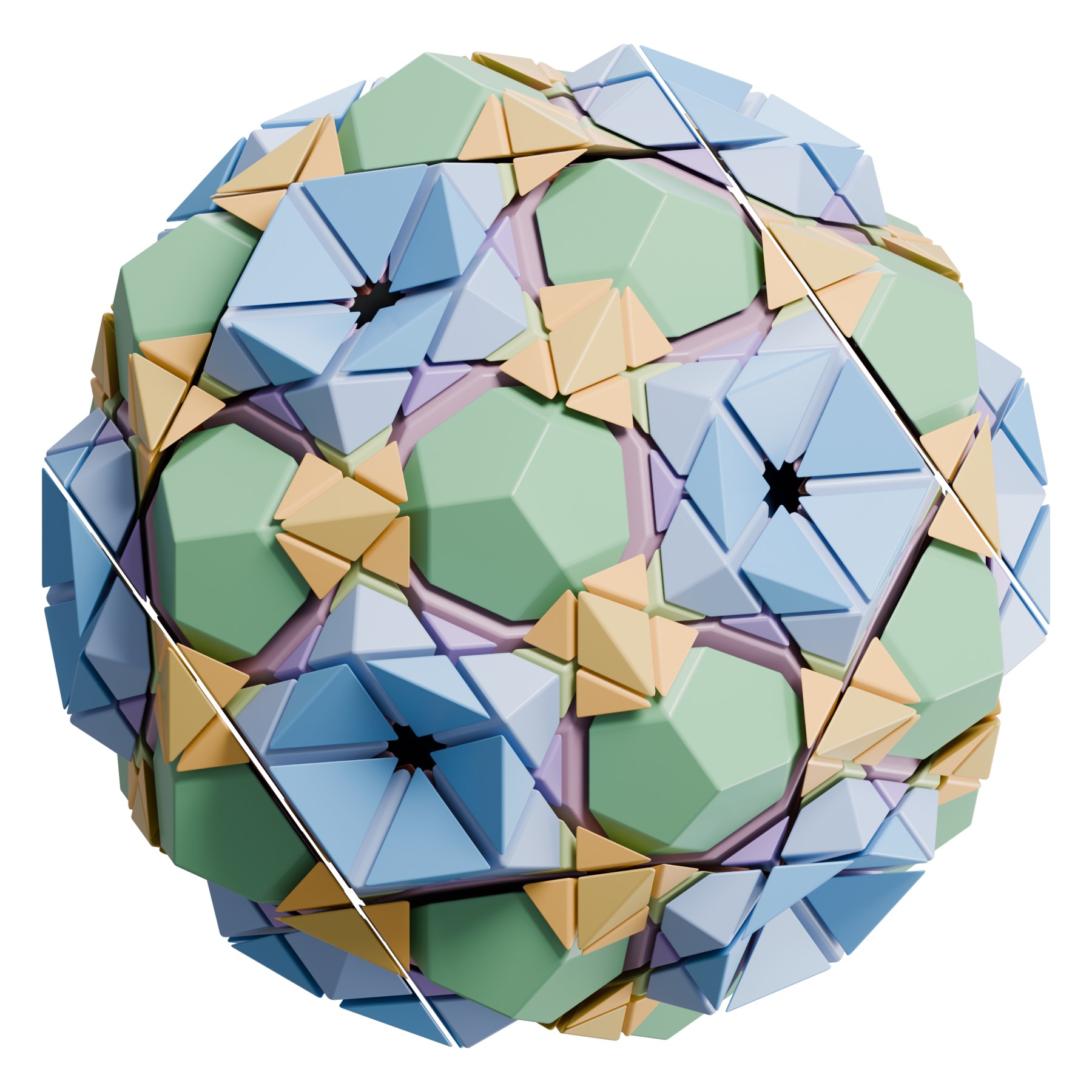}\hfill
  \fig[0.32\linewidth]{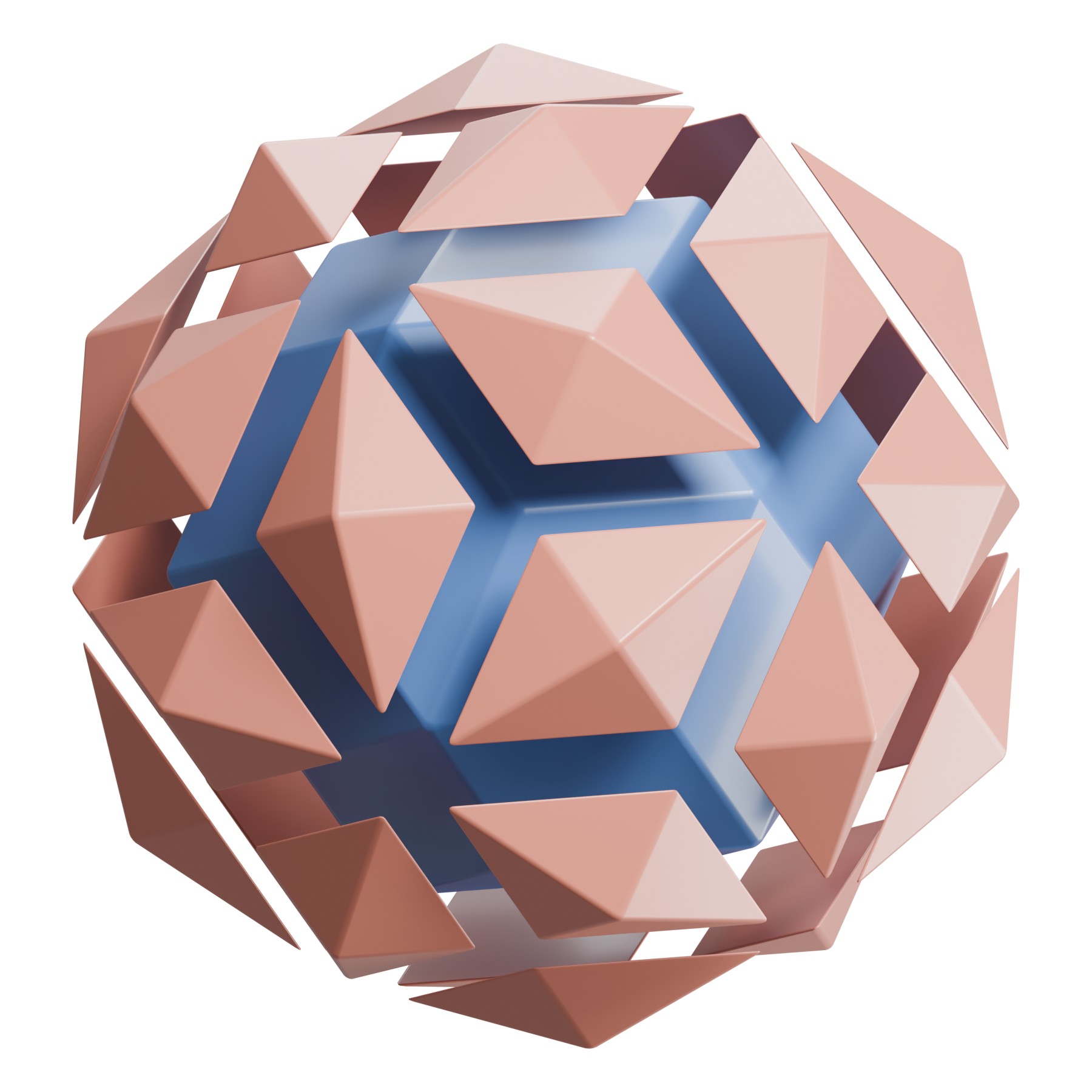}
  \caption{Peeling the $3\,655$-piece pattern of a facet of the 120-cell (star cut).  Top: the
  whole pattern; the pattern cut in half along a mirror plane of the facet (the pieces whose
  centroids lie on one side are removed, so the central rhombic triacontahedron stays whole);
  the pieces whose centroids lie within $0.73$ of the outermost centroid distance.  Bottom: within $0.575$, within
  $0.52$, and within $0.43$, which leaves the rhombic triacontahedron at the centre and the $30$
  square pyramids on its faces (explosion increased in the last two panels).  Colours are the orbit classes of
  the four-dimensional cells the pieces belong to.}
  \label{fig:peel}
\end{figure}

\begin{figure}[t]
  \centering
  \fig[0.32\linewidth]{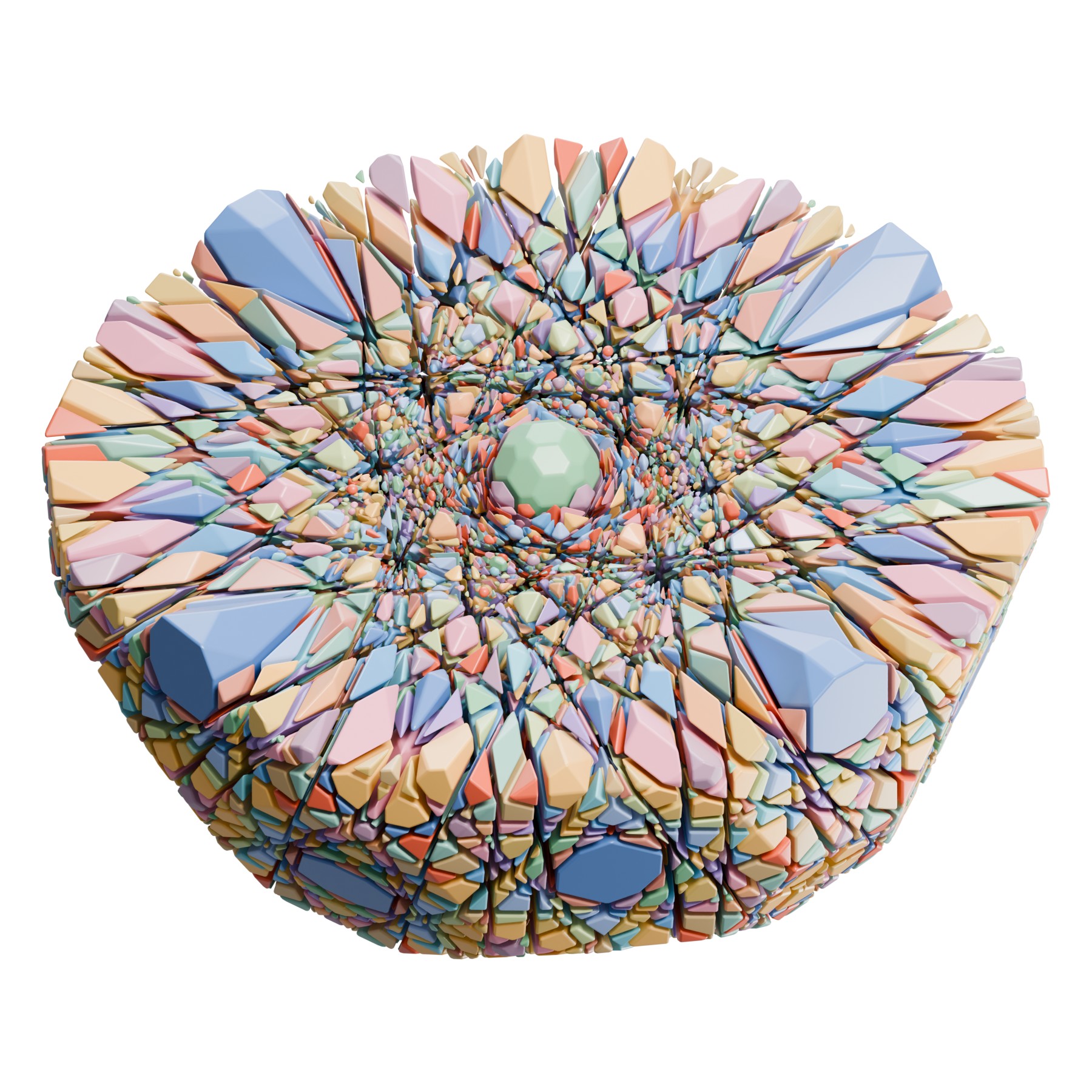}\hfill
  \fig[0.32\linewidth]{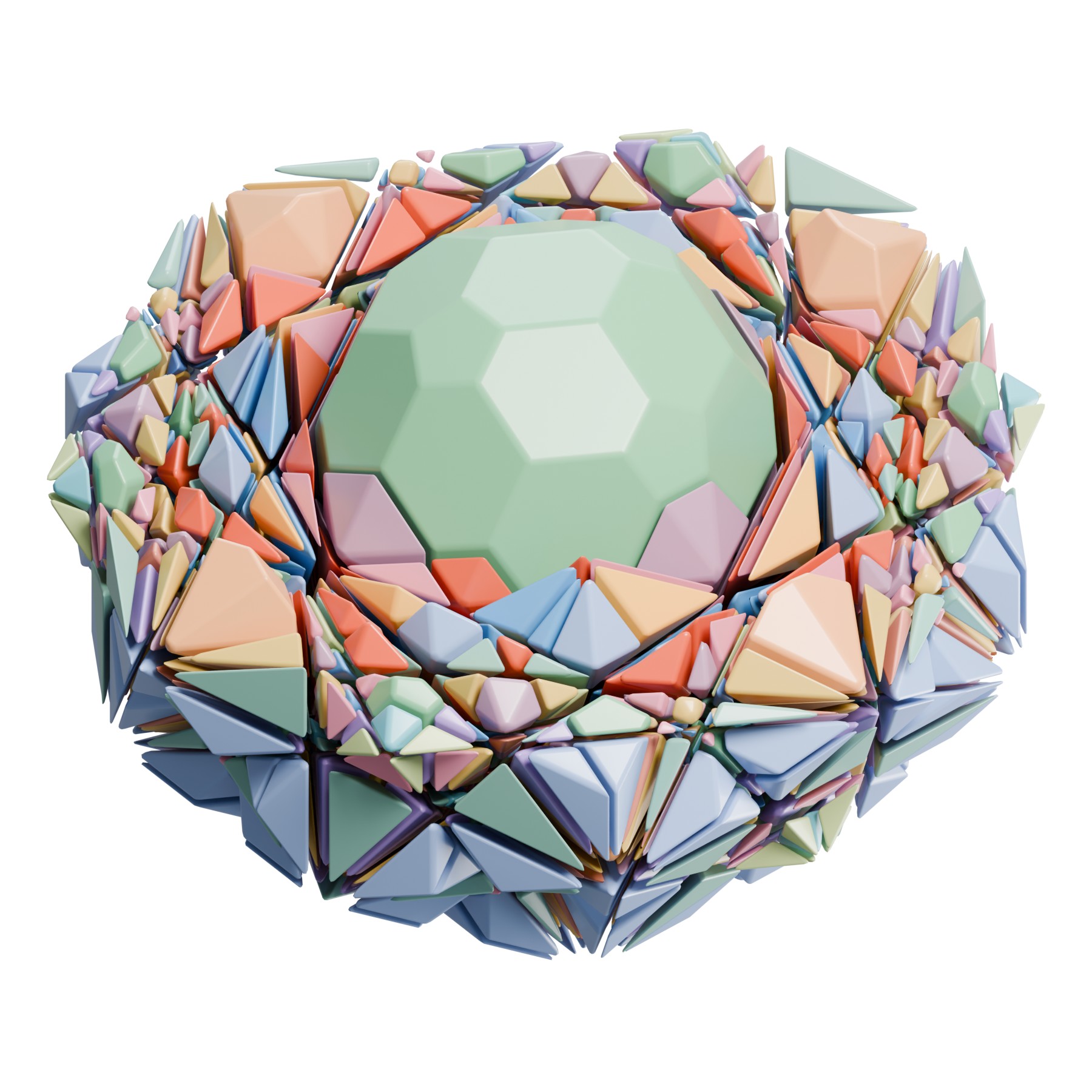}\hfill
  \fig[0.32\linewidth]{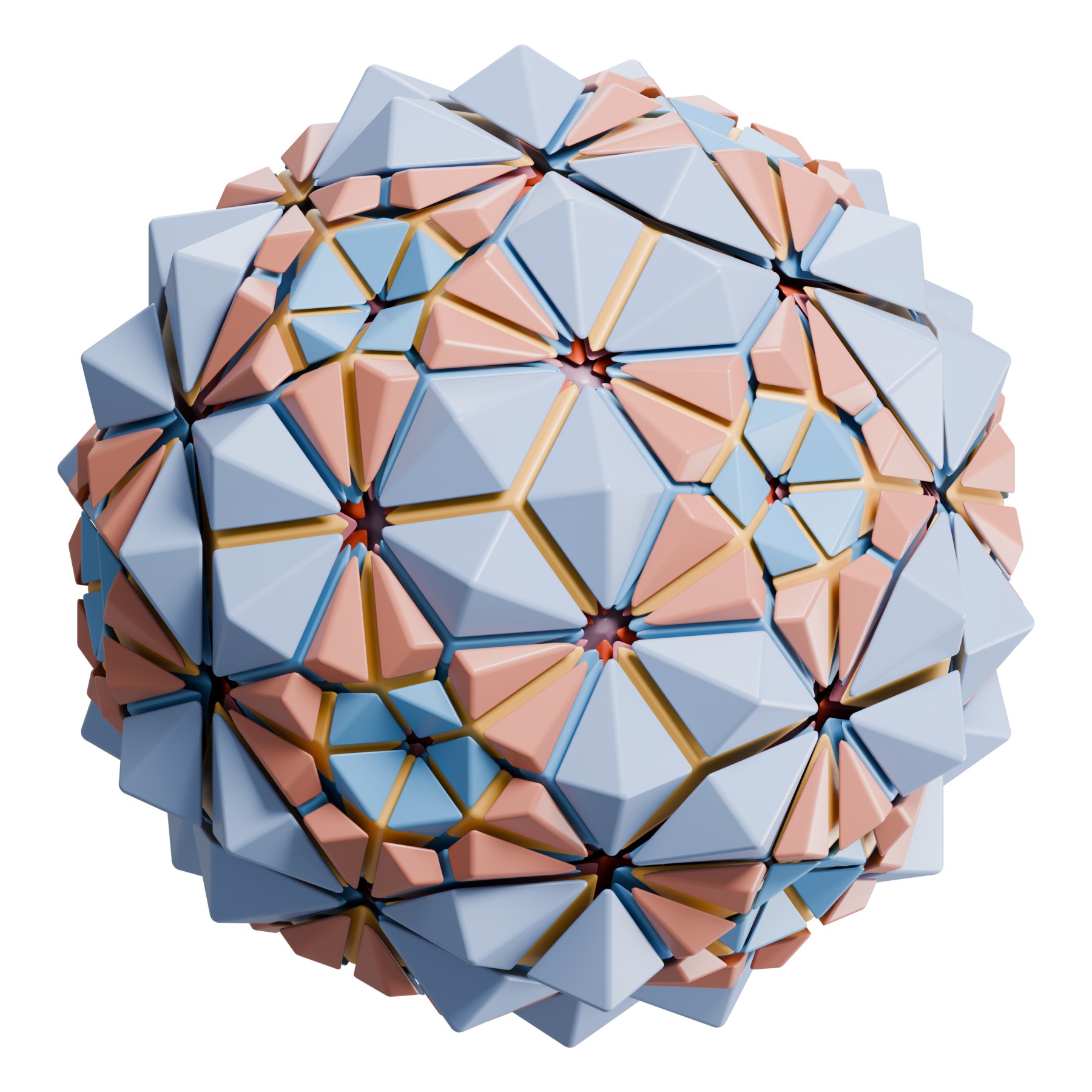}
  \caption{Peeling the central three-dimensional section of the star cut ($29\,617$ pieces,
  Figure~\ref{fig:sections} top left).  Left: cut in half along a mirror plane.  Middle: the
  pieces whose centroids lie within $0.25$ of the outermost centroid distance, cut in half
  ($2\,491$ pieces): the central section of the inner 120-cell $C$ sits in a cradle of the sections
  of the cells around it.  Right: within $0.18$, the central piece and the closed first shell of
  $622$ pieces around it (explosion reduced).}
  \label{fig:peel_section}
\end{figure}

\paragraph{Sections.}  A three-dimensional hyperplane $u\cdot x=s$ with $u$ a unit facet normal
meets $29\,617$ cells for $s=0$, $112\,225$ for $s=0.4$, $28\,573$ for $s=0.8$ and $6\,929$ for
$s=0.95$ (fractions of the inradius); the sections of the exploded cells are shown in
Figure~\ref{fig:sections}, with the explosion of Proposition~\ref{prop:explode}.  The central
section is the least crowded of the four for a reason worth stating: the hyperplane $u\cdot x=0$
is a mirror of $\Hfour$, so every cell is either symmetric under the reflection or forms a mirror
pair with another cell, and a cell of a mirror pair lies on one side of the hyperplane
(a cell that crosses a mirror contains a fixed interior point and is therefore its own image).
Only the $29\,617$ symmetric cells are cut; the hyperplane at $s=0.014$, just beside the
mirror, cuts $161\,869$.  Sweeping $s$ from $-1$ to $1$ produces the animation in the repository.

\begin{figure}[tb]
  \centering
  \fig[0.49\linewidth]{figures/cell120_facet.jpg}\hfill
  \fig[0.49\linewidth]{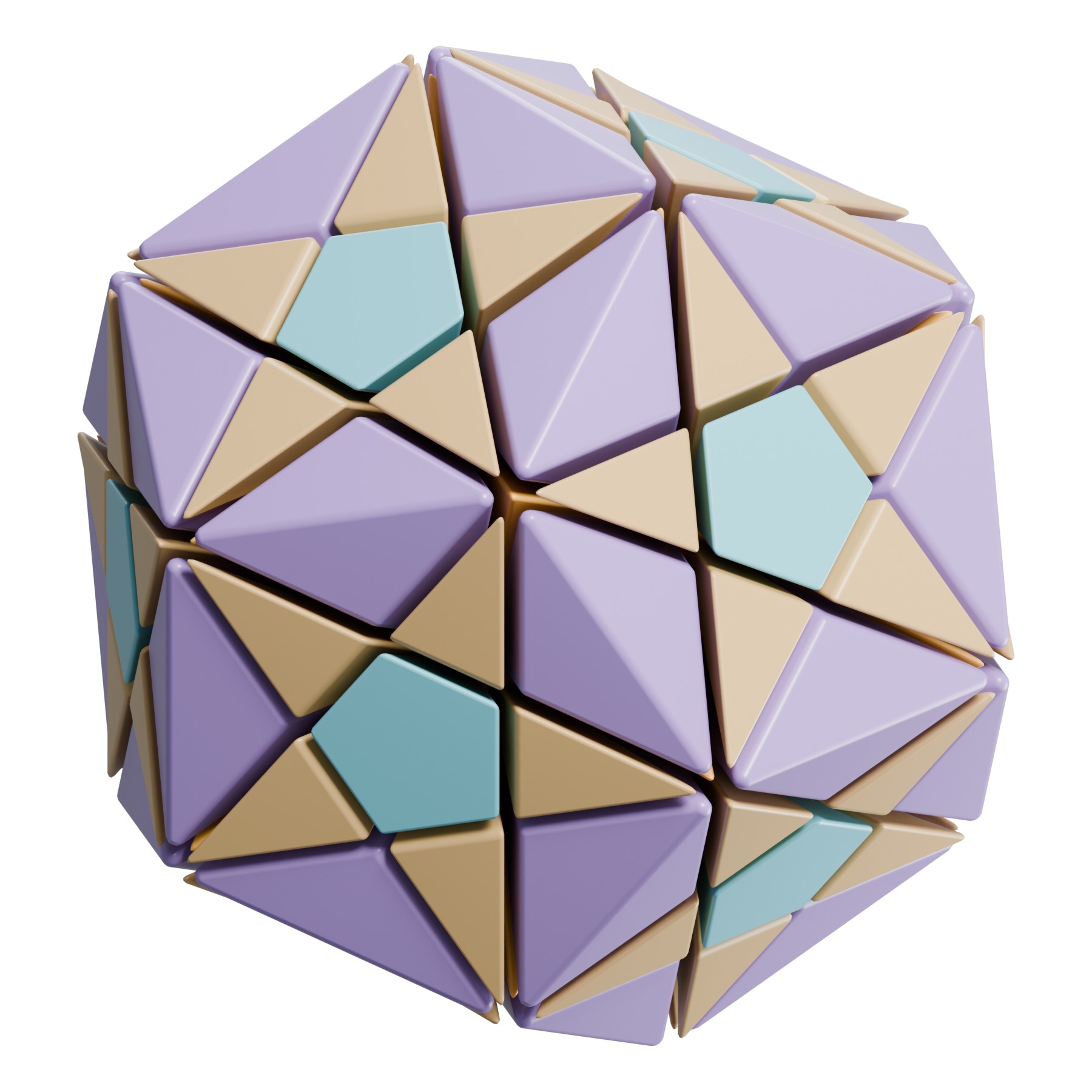}
  \caption{The pattern induced on one dodecahedral facet of the 120-cell, exploded inside the
  facet's $3$-space.  Left: star cut, $3\,655$ pieces.  Right: hereditary cut, $165$ pieces, the
  three-dimensional stellation dissection of Figure~\ref{fig:dodeca}.}
  \label{fig:facets}
\end{figure}

\begin{figure}[tb]
  \centering
  \fig[0.49\linewidth]{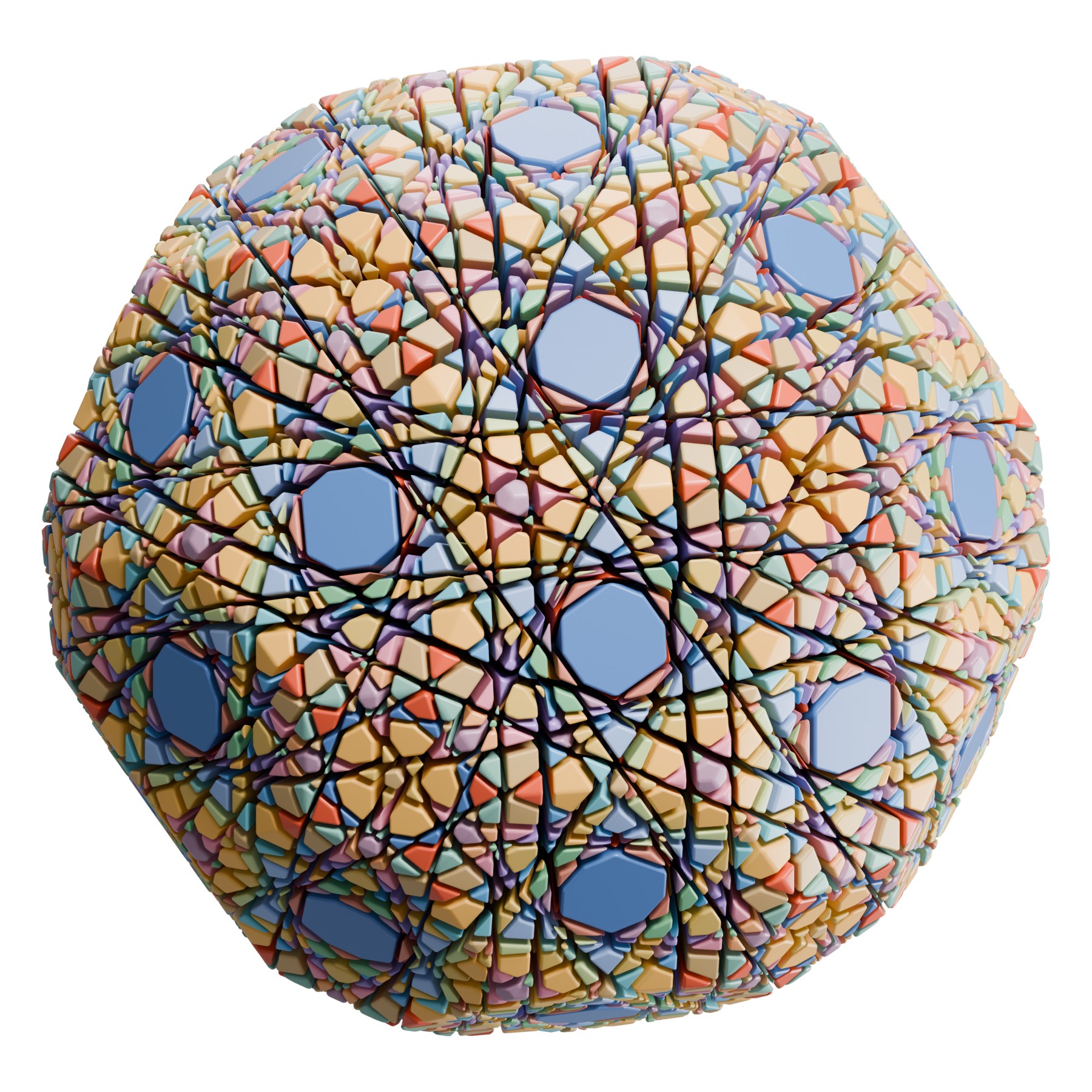}\hfill
  \fig[0.49\linewidth]{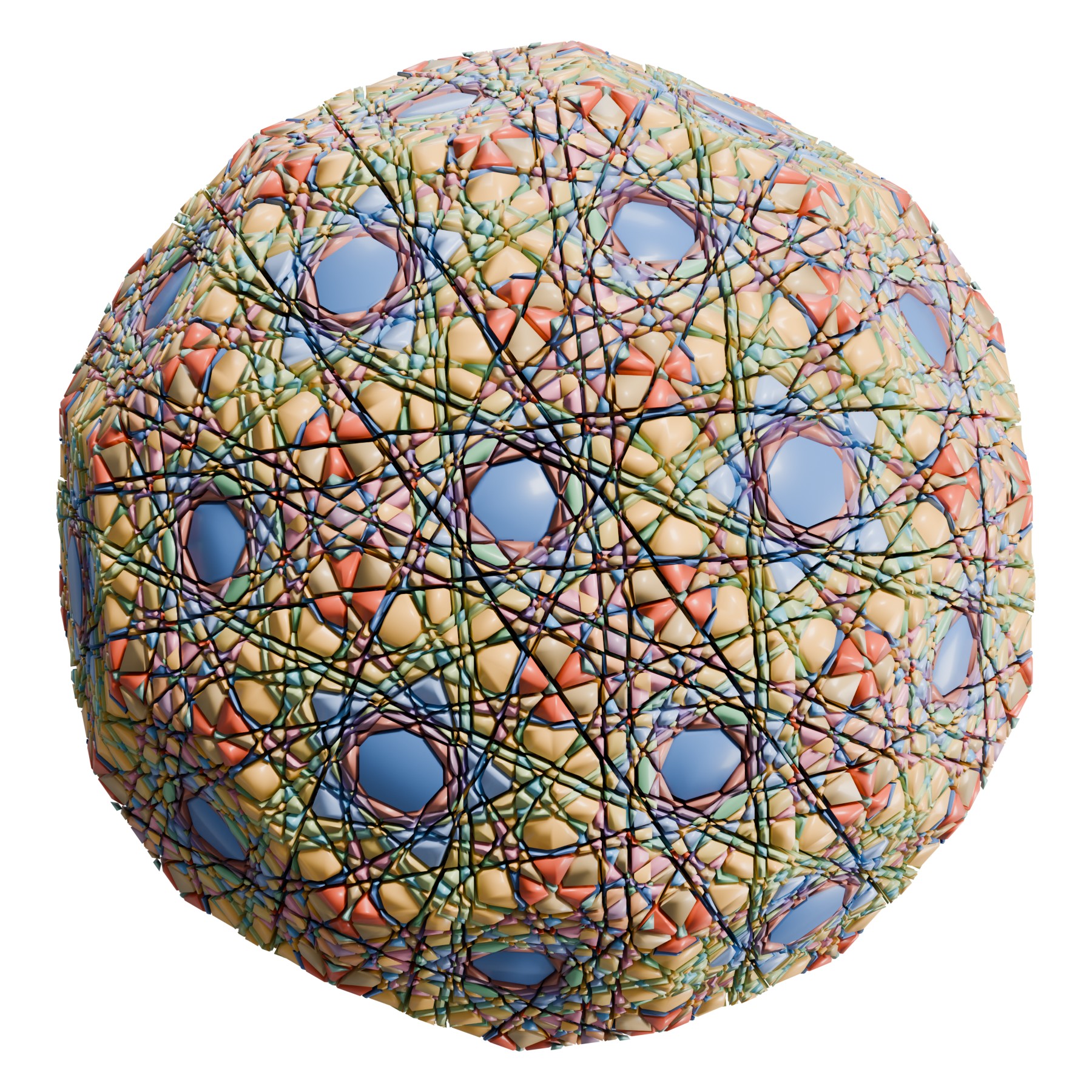}\\[3pt]
  \fig[0.49\linewidth]{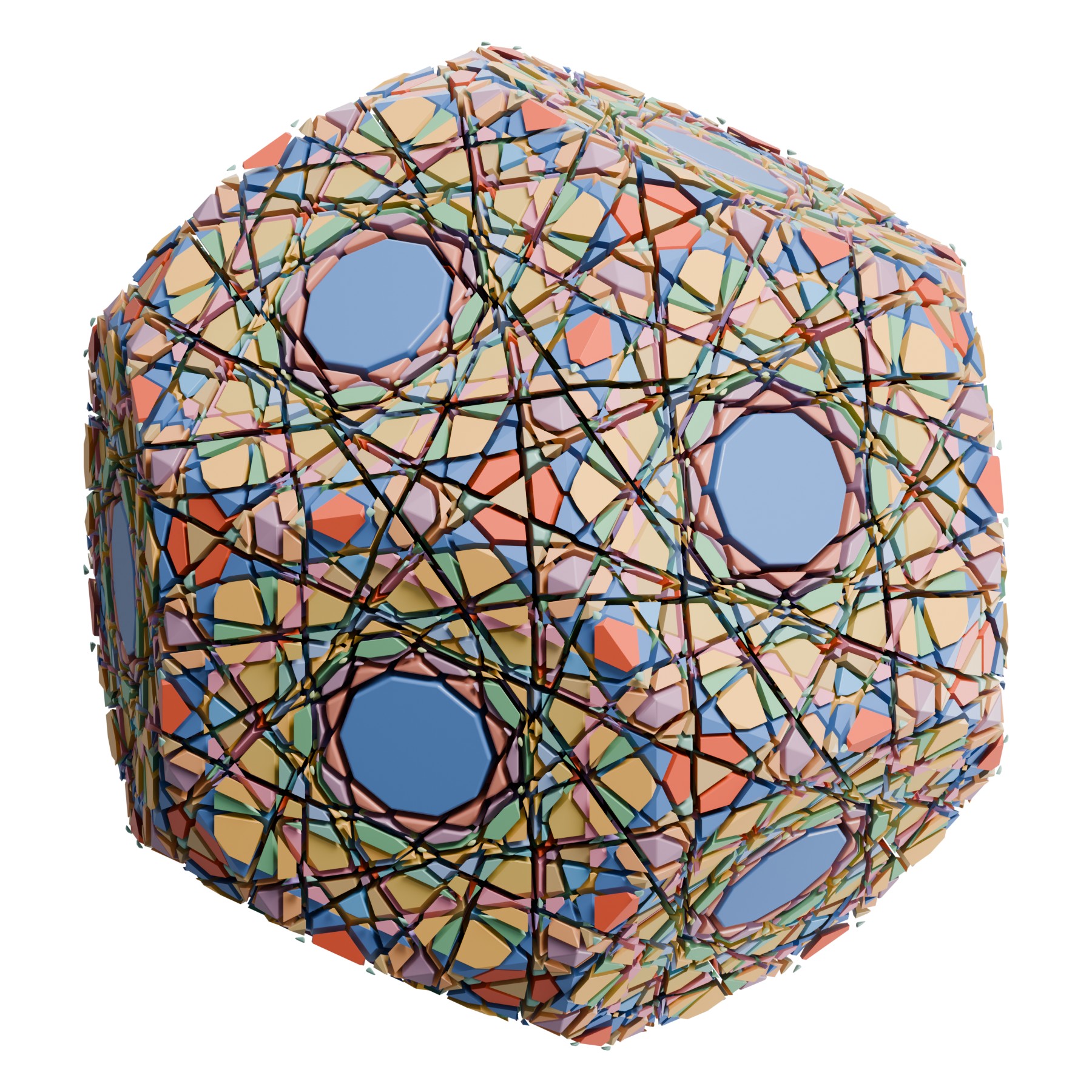}\hfill
  \fig[0.49\linewidth]{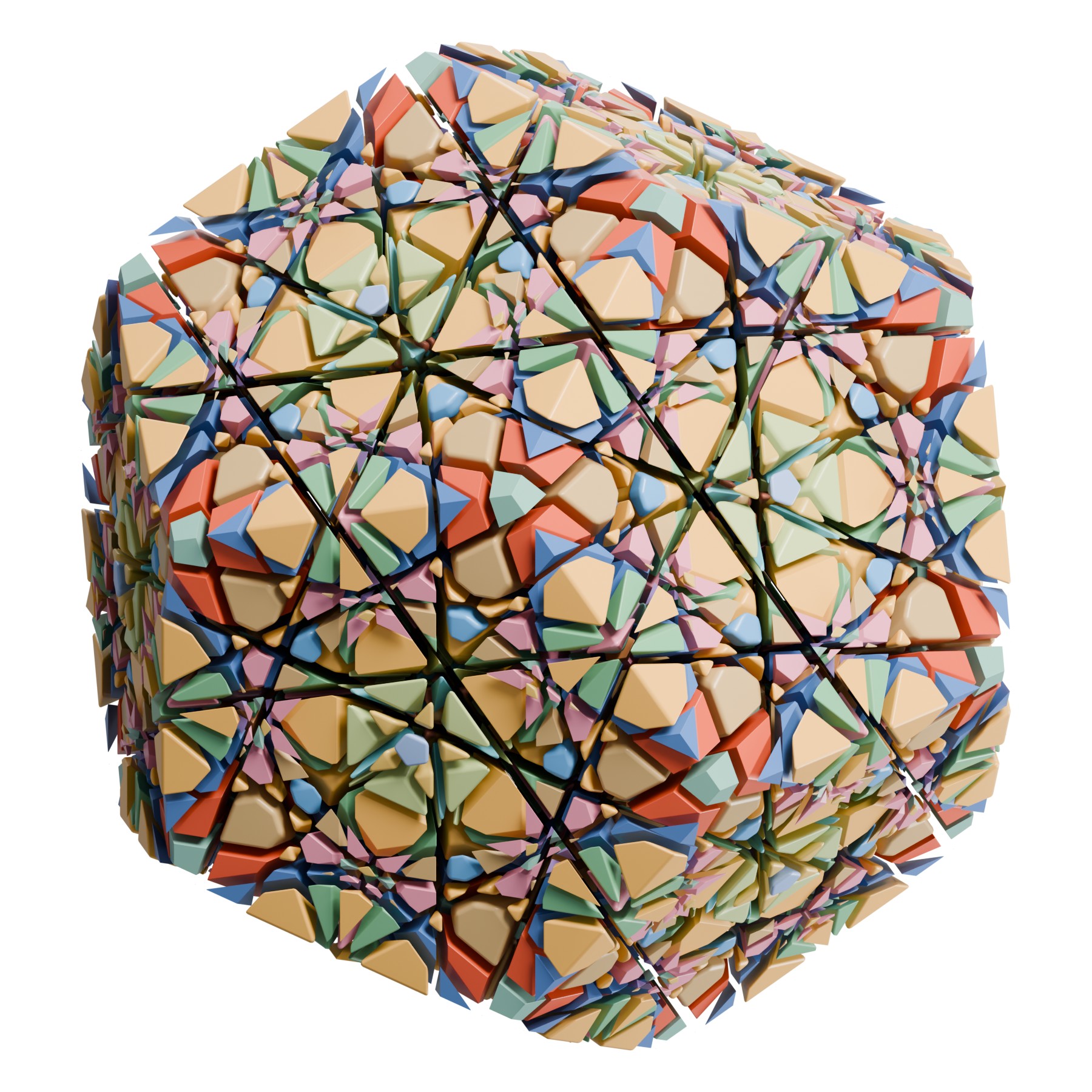}
  \caption{Three-dimensional sections of the star cut of the 120-cell by hyperplanes orthogonal to
  a facet normal, at $s=0$, $0.4$ (top) and $0.8$, $0.95$ (bottom) of the inradius.  Each piece is
  the section of one four-dimensional cell, coloured by class and slightly exploded.}
  \label{fig:sections}
\end{figure}

\subsection{The hereditary cut}
\begin{result}[hereditary cut]\label{res:hered}
The 120-cell cut by the $120$ hyperplanes \eqref{eq:heredcut} has $\NHERED$ cells in
$\NHEREDCLASSES$ orbits (Table~\ref{tab:hered}).  Volumes add up to $\vol B$ within
$4\cdot10^{-15}$.
\end{result}

\begin{table}[t]
  \centering\small
  \begin{tabular}{rrllr}
\toprule
class & cells & $f$-vector & facets (vertex counts) & $4$-volume\\
\midrule
0 & 1 & $(600,1200,720,120)$ & $20^{120}$ & 142.792414 \\
1 & 120 & $(40,80,54,14)$ & $10^{12} 20^{2}$ & 0.279267 \\
2 & 600 & $(18,46,40,12)$ & $5^{4} 7^{4} 10^{4}$ & 0.056763 \\
3 & 720 & $(17,36,28,9)$ & $6^{2} 8^{5} 10^{2}$ & 0.028597 \\
4 & 1\,200 & $(11,31,31,11)$ & $4^{6} 6^{3} 7^{2}$ & 0.016186 \\
5 & 1\,200 & $(14,28,22,8)$ & $4^{3} 8^{3} 10^{2}$ & 0.012714 \\
6 & 720 & $(7,16,16,7)$ & $4^{5} 6^{2}$ & 0.001157 \\
7 & 3\,600 & $(5,10,10,5)$ & $4^{5}$ & 0.000518 \\
\bottomrule
\end{tabular}

  \caption{The eight orbits of the hereditary cut.}\label{tab:hered}
\end{table}

The pieces are easy to describe.  Class $0$ is the core $(1-\varphi^{-4})B$.  Class $1$ consists of
$120$ \emph{frusta}, $f=(40,80,54,14)$, one under each facet $F$ of $B$: a facet of the core (a
dodecahedron of circumradius $0.914$) is joined to the central dodecahedron of $F$'s own
stellation dissection (circumradius $0.253=\varphi^{-3}\cdot1.071$, on the boundary of $B$) by twelve
cells with $10$ vertices, pentagonal frusta, one under each face of $F$; the cell narrows towards
the boundary.  Class
$7$ consists of $3\,600$ $4$-simplices, the four-dimensional relatives of the $90$ tetrahedra of
the dodecahedron.  The remaining five orbits have sizes $600$, $720$, $1200$, $1200$ and $720$,
the numbers of vertices, pentagons, edges, edges and pentagons of the 120-cell.  Each of the $600$
cells of class $2$ ($f=(18,46,40,12)$) has exactly one vertex of $B$ among its vertices, and each
of the $1200$ cells of class $4$ has two, the ends of an edge of $B$; the cells of classes $3$,
$5$ and $6$ have no vertex of $B$ and lie between the core and the boundary.  Figure~\ref{fig:hered} shows the eight pieces and a central section.

\begin{figure}[tb]
  \centering
  \begin{tikzpicture}
    \node[anchor=south west,inner sep=0] (img) at (0,0) {\fig[0.6\linewidth]{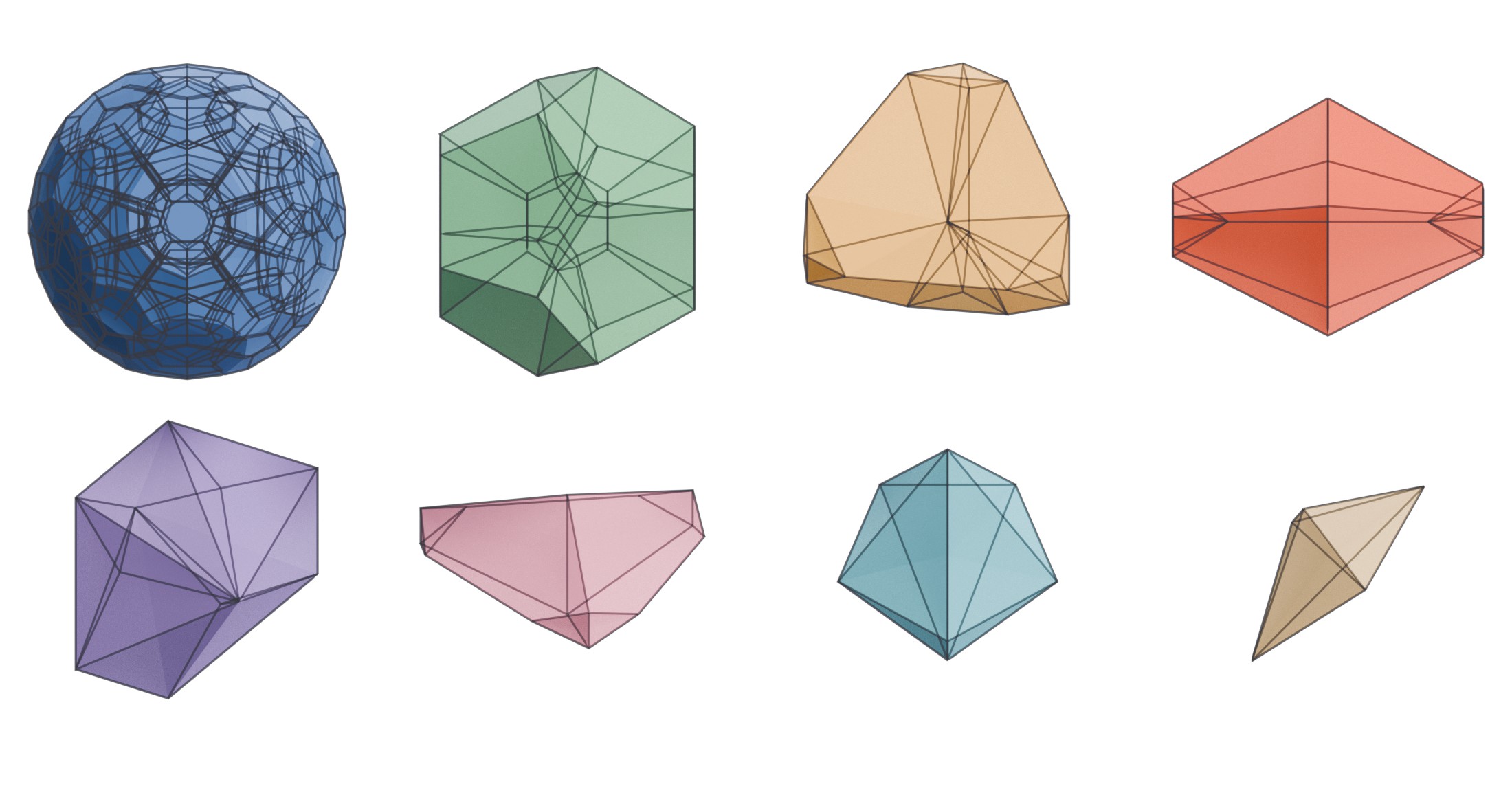}};
    \begin{scope}[x={(img.south east)},y={(img.north west)}]
      \IfFileExists{figures/cell120_galleryB_labels.tex}{\input{figures/cell120_galleryB_labels.tex}}{}
    \end{scope}
  \end{tikzpicture}\hfill
  \fig[0.38\linewidth]{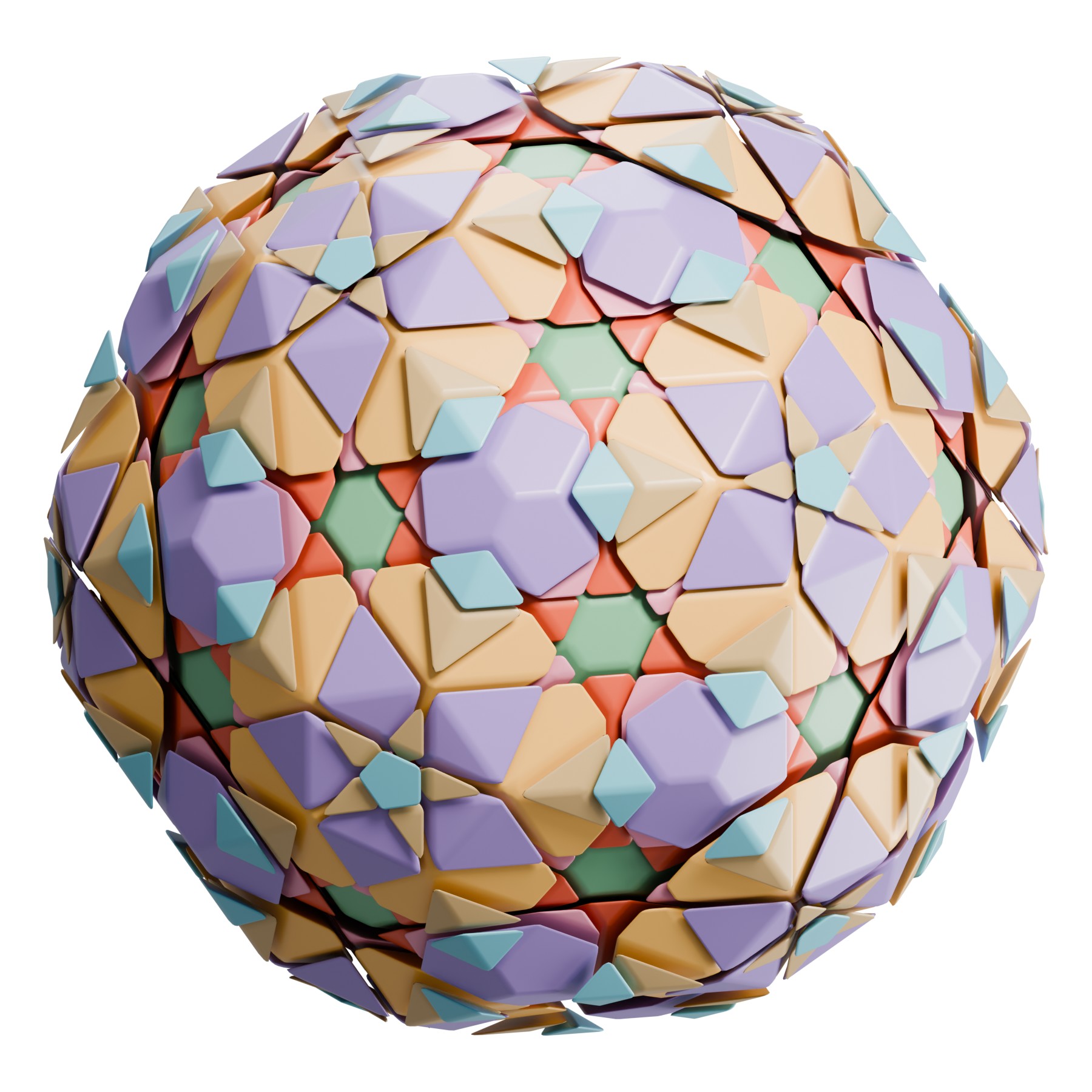}
  \caption{The hereditary cut: the eight orbits (left, perspective projections) and the central
  three-dimensional section, exploded (right).}
  \label{fig:hered}
\end{figure}

\subsection{The vertices of the four-dimensional pieces}\label{sec:vert4}
The slicer writes one representative cell per orbit with floating-point vertices.  For each
representative vertex we identified the hyperplanes through it from the floats, intersected four
independent ones by Cramer's rule in $\Qphi$, and checked the exact point against every incident
hyperplane exactly (Appendix~\ref{app:checks}, blocks~J and K).

\begin{result}\label{res:vert4}
\begin{itemize}
\item \textbf{Hereditary cut.}  The $8$ representatives have $712$ vertices, $665$ distinct, all
in $\Zphi^4$ with no denominators and all in the vertex module $\MV$ (hence in $\Ico$).  Since the
reflections of $\Hfour$ preserve $\MV$ (Result~\ref{res:modules}), \emph{every vertex of every one
of the $\NHERED$ cells lies in $\MV\subset\Zphi^4\subset\Ico$}.  Incidences: $678$ simple vertices,
$30$ on six hyperplanes, $4$ on eight.
\item \textbf{Star cut.}  The $\NSTARCLASSES$ representatives have $3\,461$ vertices, $1\,748$
distinct, with coordinate denominators $1$ ($1\,362$ vertices), $2$ ($198$), $3$ ($23$) and $5$
($165$).  Of them $1\,362$ lie in $\Zphi^4$, $1\,560$ in $\Ico$ and $1\,295$ in $\MV$; $30x\in\Zphi^4$
and $15x\in\Ico$ for every vertex $x$.  Up to $20$ hyperplanes pass through one vertex.
\end{itemize}
\end{result}

So the hereditary dissection, the one that restricts to the three-dimensional dissection on every
facet, is a dissection \emph{of the module $\MV$}, exactly as the dodecahedron's is a dissection
of $\M$ (Result~\ref{res:vert3}).  The star cut is not: it needs the inert prime $3$ and the
ramified prime $5$ in its denominators, besides the $2$ that already divides the normals.  We did
not expect the $3$; it is the signature of vertices where many mirrors of $\Hfour$ meet (up to
$20$ hyperplanes through one point), and it makes the star cut the less arithmetic of the two, the
opposite of what its unit depth suggests.  Taken together with the depths in
Table~\ref{tab:bifurcation}: the star cut has the unit \emph{scale} and the hereditary cut has the
module \emph{vertices}; neither has both.

\subsection{Cross-checks}
Independent checks of the four-dimensional computation: (i) the same program with $d=3$
reproduces Results~\ref{res:seven} and \ref{res:phase}; (ii) the facet counts $3\,655$ and $165$ were
obtained a second time by restricting the hyperplanes to the facet and running the
three-dimensional code; (iii) all volumes are positive, all cells satisfy Euler's relation and the
total volume is right to $10^{-12}$; (iv) all class sizes divide $|\Hfour|$; (v) the orbit sizes of the hereditary cut are
numbers of elements of the 120-cell ($120$ facets, $600$ vertices, $1200$ edges, $720$
pentagons, $3600=3\cdot1200$), and the cells of the orbits of sizes $600$ and $1200$ do contain a
vertex, respectively an edge, of $B$; (vi) every representative vertex of both cuts is reproduced
exactly from its incident hyperplanes (Result~\ref{res:vert4}), which also confirms the incidence
sets the slicer recorded.

\section{Coda: the pentagram in a circle, and its relatives}\label{sec:balls}

The pentagram is usually drawn inside a circle.  The circle is the circumcircle of the pentagon,
and the five diagonals are its chords, so the picture is the stellation dissection applied to the
\emph{disc} instead of the pentagon.  It has eleven pieces as well: the core pentagon, the five
points, and five circular caps that are the gnomons with the circular segments beyond the sides
attached.  The five regions of the plane beyond two diagonal lines at a vertex touch the circle only
at that vertex and contribute nothing.

\paragraph{A cultural aside.}  The pentagram in a circle is also, of course, the most recognisable
occult emblem of the Western world, and it is worth recording how a figure of Euclidean geometry
acquired that reputation.  To the Pythagoreans the pentagram was a sign of recognition and of
health \cite{Huntley1970}; in the Renaissance Agrippa drew the human figure inscribed in it as the
microcosm \cite{Agrippa1533}; and in the occult revival of the nineteenth century \'Eliphas L\'evi
placed a pentagram on the forehead of his Baphomet, the goat-headed figure of \emph{Dogme et
rituel de la haute magie} (1854--56), while assigning the \emph{inverted} pentagram, ``two points
up'', to the goat of the witches' sabbath \cite{Levi1856}.  The inverted pentagram with a goat's
head drawn inside it and the Hebrew letters of \emph{Leviathan} at its points first appears in
Stanislas de Guaita's \emph{La clef de la magie noire} (1897) \cite{Guaita1897}; Anton LaVey's
Church of Satan, founded in 1966, adopted that drawing as its emblem, printed it on the cover of
\emph{The Satanic Bible} (1969) and named it the Sigil of Baphomet \cite{LaVey1969,WikiSigil}, which
is how the chain inverted pentagram $\to$ Baphomet $\to$ Satanism $\to$ evil entered popular
culture.  None of this has anything to do with the golden ratio.  The reader who has followed the
pentagram from the pentagon to the dodecahedron and the 120-cell may decide for themselves
whether its four-dimensional relative, drawn on the $3$-sphere by $120$ hyperplanes, is any more
sinister than the plane one.

The same happens one dimension up, but not two (Figure~\ref{fig:balls}).

\begin{result}\label{res:balls}
Cut the circumscribed ball of the dodecahedron by the twelve planes of the stellation dissection: $165$
pieces in $7$ orbits of $I_h$ of sizes $1,12,12,20,30,30,60$, the same numbers as for the
dodecahedron.  Cut the circumscribed $4$-ball of the 120-cell by the $120$ hyperplanes
\eqref{eq:starcut}: \NBALLFOUR\ pieces in \NBALLFOURCLASSES\ classes.  Of these,
$\NSTAR$ pieces correspond one-to-one to the cells of the star cut (they have the same sign
vectors), and $\NBALLFOUREXTRA$ pieces in $\NBALLFOUREXTRACLASSES$ orbits, of sizes $1440$,
$7200$, $7200$, $7200$ and $14\,400$, lie entirely in the gap between the 120-cell and its
circumsphere.
\end{result}

In the plane and in space the counts agree for the same reason: the cut hyperplanes pass through
the vertices of $B$, which lie on the sphere, and every cell of the arrangement outside $B$ either
touches the sphere at a single vertex, contributing nothing, or meets it in a cap that is glued to
a piece of $B$'s cut.  The pieces that reach the sphere change shape, of course: the $30$ wedges
of the dodecahedron acquire spherical caps.  In four dimensions the second alternative is not
exhaustive: $42$ hyperplanes cross each facet of $B$ and continue into the gap between the facet
and the sphere, where they intersect, and the five orbits of cells they cut out there are
entirely new.  The largest of them, $1440$ cells with $f=(14,37,35,12)$ and volume $0.012$
before clipping, would be the $11$th largest class of the star cut; the central section of the
$4$-ball has $30\,097$ pieces against the $29\,617$ of the 120-cell.  (The computation treats the
ball exactly like the polytope: the arrangement is built inside the polytope $2B$, whose facets
are also in $\Zphi$, and a cell is kept when its distance from the centre is smaller than the
circumradius, which is decided by a small quadratic program on its exact vertices; the sign
vectors identify the cells across the two computations.)

\paragraph{The stars inside.}  The first stellation of the core is a union of pieces in every
dimension: the pentagon with its five points, the small stellated dodecahedron (the core and the
twelve pentagonal pyramids, $13$ of the $165$ pieces of the ball), and the small stellated 120-cell
$\{5/2,5,3\}$, the core and the $120$ dodecahedral pyramids of class $45$, which are single cells
with $f=(21,50,42,13)$.  The stars lie well inside the balls: the spikes of the dodecahedron reach
$0.42$ of the radius of the sphere, those of the 120-cell only $2\varphi^{-3}/(2\sqrt2)=0.17$,
because the apex of a pyramid over a facet of $C$ is where the twelve neighbouring facet
hyperplanes meet, at distance $2\varphi^{-3}$ from the centre against the inradius $\varphi^{-2}$
of $C$: the dihedral angle of the 120-cell is $144^\circ$ and its spikes are shallow, of height
$2\varphi^{-3}-\varphi^{-2}=\varphi^{-5}=0.09$.
Figure~\ref{fig:balls} shows the balls cut in half along a mirror plane with the star left whole.
The mirror hyperplane $u\cdot x=0$ of the four-dimensional picture contains the axes of the $30$
spikes whose facet normals are orthogonal to $u$ and crosses no other (a spike that crossed it
would be its own mirror image), so what one sees of the $120$-pointed star in the central section
is a $30$-pointed one: the central section of $C$, which has $80$ vertices, with $30$ hexagonal
pyramids of height $\varphi^{-5}$ on it.

\begin{figure}[tb]
  \centering
  \begin{tikzpicture}[scale=2.0, line join=round]
    \def\R{1}\def\r{0.381966}
    \foreach \kk in {0,...,4} { \coordinate (v\kk) at ({90+72*\kk}:\R); \coordinate (u\kk) at ({54+72*\kk}:\r); }
    \fill[blue!25!white] (u0)--(u1)--(u2)--(u3)--(u4)--cycle;
    \foreach \kk/\kn in {0/1, 1/2, 2/3, 3/4, 4/0} { \fill[red!55!orange!60] (v\kk)--(u\kk)--(u\kn)--cycle; }
    \foreach \kk/\kn in {0/1, 1/2, 2/3, 3/4, 4/0} { \fill[green!50!teal!40] (v\kk)--(u\kn)--(v\kn) arc ({162+72*\kk}:{90+72*\kk}:\R) -- cycle; }
    \draw[thick,black!80] (0,0) circle (\R);
    \foreach \kk/\kn in {0/2, 1/3, 2/4, 3/0, 4/1} { \draw[thick,black!80] (v\kk)--(v\kn); }
  \end{tikzpicture}\hfill
  \fig[0.32\linewidth]{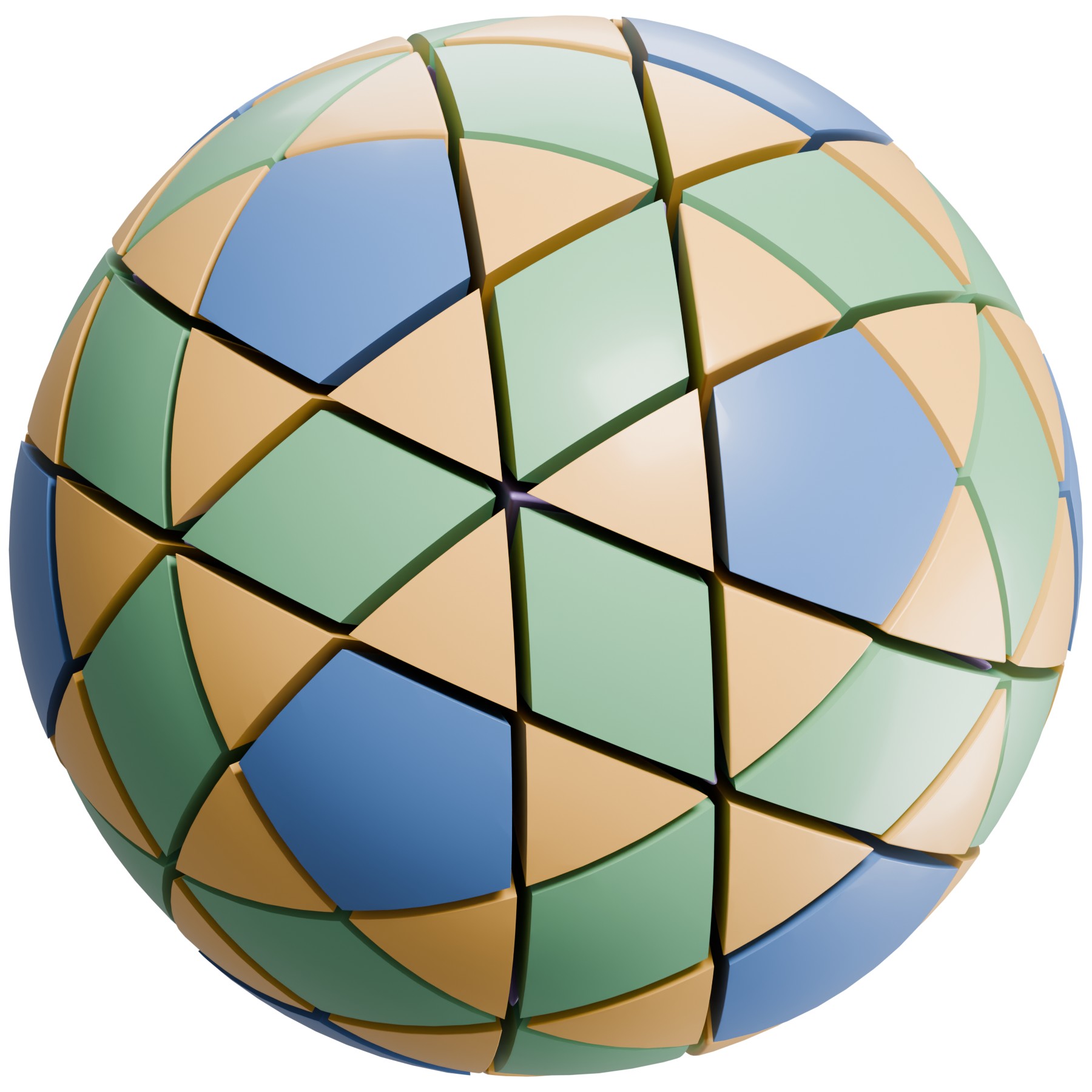}\hfill
  \fig[0.32\linewidth]{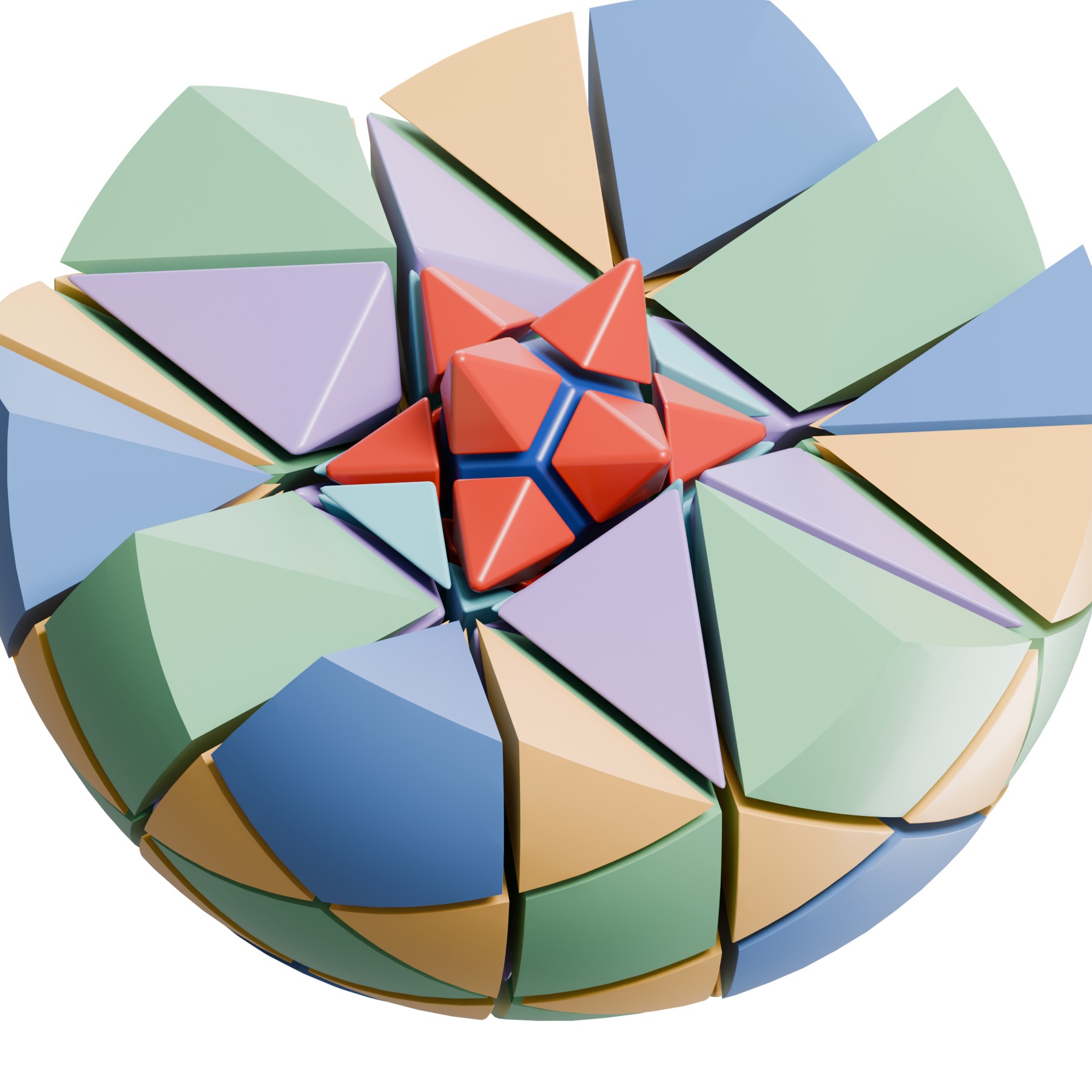}\\[3pt]
  \fig[0.32\linewidth]{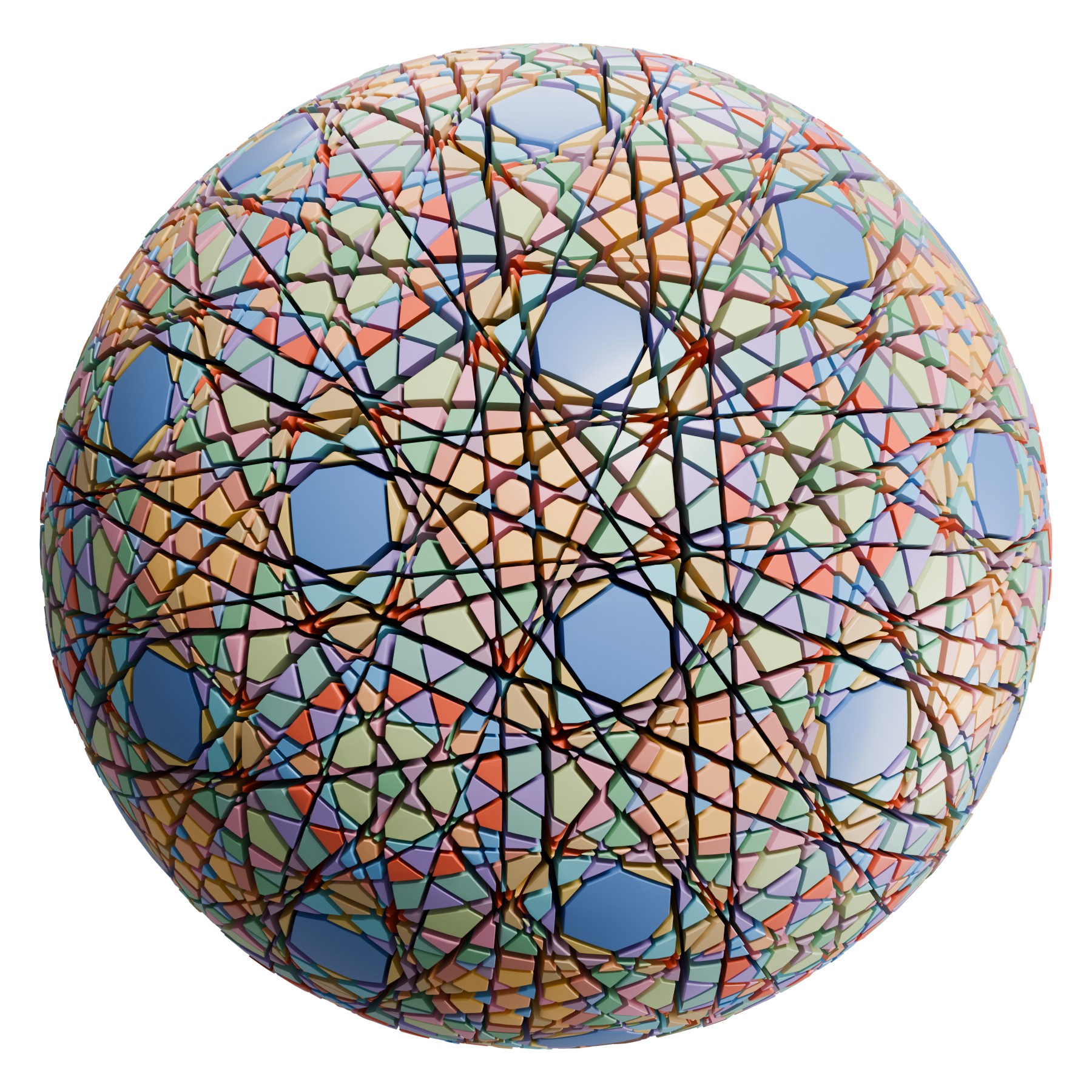}\hfill
  \fig[0.32\linewidth]{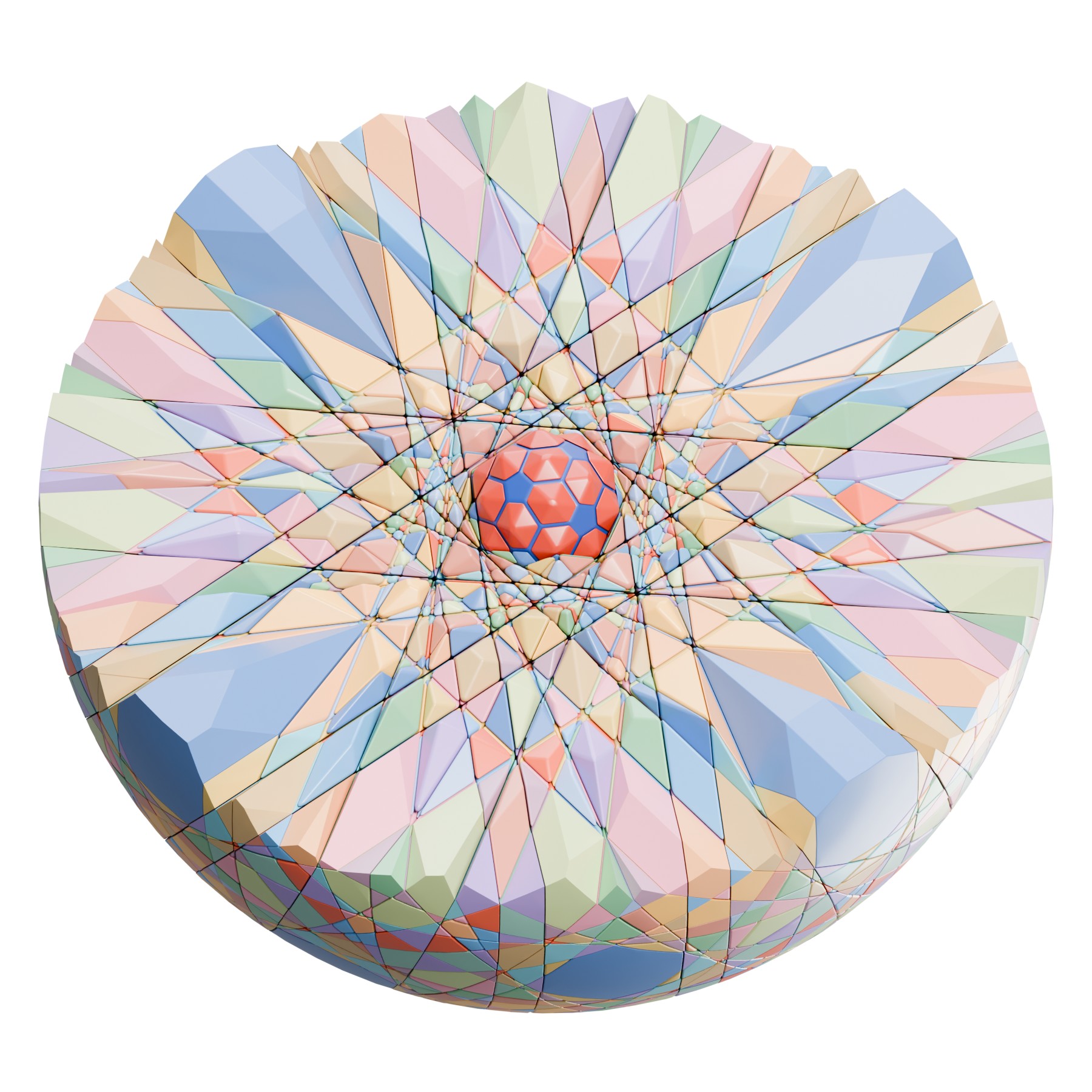}\hfill
  \fig[0.32\linewidth]{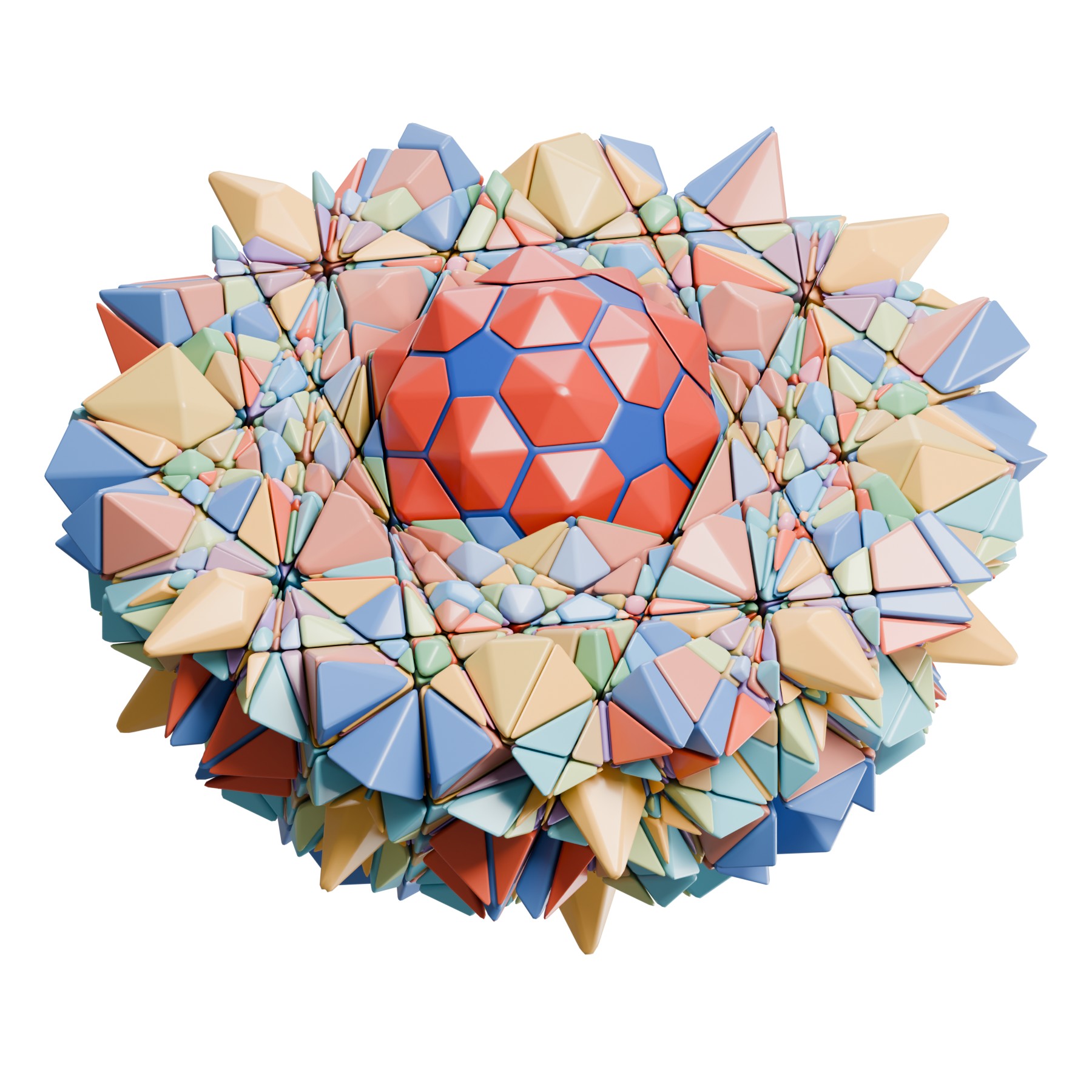}
  \caption{Top: the stellation dissection of the disc ($11$ pieces) and of the ball ($165$ pieces,
  exploded), and the ball cut in half along a mirror plane with the small stellated dodecahedron
  inside it left whole (core in blue, the twelve spikes in red).  Bottom: the central
  three-dimensional section of the dissected $4$-ball (exploded pieces, each cut off by the sphere);
  the same section cut in half along a mirror plane, with the section of the small stellated
  120-cell left whole at the centre; and the pieces within $0.22$ of the outermost centroid
  distance around that section, cut in half.  The hyperplane of the section passes through the
  axes of $30$ of the $120$ spikes, so the section of the four-dimensional star has $30$ points
  (in red) on the $80$-vertex central section of the core (in blue).}
  \label{fig:balls}
\end{figure}

\section{Discussion}\label{sec:discussion}

The stellation dissection is a natural object: in dimensions two and three it is classical, and in
dimension four it was announced in 1995 but, as far as we can tell, never carried out.  The
obstacle was never conceptual.  It was the combination of degenerate incidences (five planes
through a point everywhere in dimension three, up to twenty hyperplanes through a point in
dimension four) with the size of the computation, and the obstacle disappears once one notices
that everything lives in $\Zphi$, where a sign is a fact and not an estimate, and that cells can
be streamed instead of stored.  The two numbers $\NSTAR$ and $\NHERED$, and the
$\NSTARCLASSES$ and $\NHEREDCLASSES$ orbits, are the main results; the catalogue in
Appendix~\ref{app:catalogue} is the raw material for whatever comes next.

The ring did more than make the computation possible.  It explains the bifurcation
(a unit against the ramified prime), it predicts the critical depths of the phase diagram as
heights of vertex, face-centre and edge layers, and it supplied the two statements about the
pieces that we find most suggestive: the dissection of the dodecahedron and the hereditary
dissection of the 120-cell are dissections of the vertex modules $\M$ and $\MV$, while the star
cut needs the primes $2$, $3$ and $5$ in its denominators.  Several questions remain open.
\begin{enumerate}
\item \emph{Rigorous orbits in dimension four.}  Replace the invariant-plus-divisibility argument
of Section~\ref{sec:orbits} by canonical sign vectors under the $14\,400$ permutations induced by
$\Hfour$, as Result~\ref{res:orbits} does for $I_h$; the exact reflections of
Result~\ref{res:modules} and the stored sign vectors are all that is needed.
\item \emph{Four-dimensional phase diagram.}  The cell count as a function of the depth $t$; the
candidate critical depths are the layer heights \eqref{eq:layers}, all elements of $\Qphi$, and the
heights of the other $\Hfour$-orbits of centres (edges, pentagons, facets).
\item \emph{Exact volumes} in $\Qphi$ from the exact boundary representation, which would turn
the volume check into an identity.  The vertices are available exactly (Result~\ref{res:vert4}).
\item \emph{The denominators $2$, $3$, $5$.}  Which flats of $\Hfour$ carry the star-cut vertices
with denominator $3$, and is there a module between $\MV$ and $\tfrac1{30}\Zphi^4$ that contains
all of them?
\item \emph{The modules.}  The icosahedral $\Z$-lattices of rank six in $\R^3$ are classified
(primitive, face-centred, body-centred \cite{RMW1987}); which of them $\Zphi^3$ and $\M$ are, and
how $\M$ relates to the $\Z^6$ projection of Kramer and Neri \cite{KramerNeri1984}, is bookkeeping
we have not done.
\item \emph{Self-similarity.}  Kramer's seven cells re-pack themselves at scale $\varphi^3$, a unit,
inside the module $\M$.  Whether the $\NSTARCLASSES$ or the $\NHEREDCLASSES$ four-dimensional
pieces do anything comparable is open; Table~\ref{tab:bifurcation} says that the star cut has the
unit scale and the hereditary cut has the module vertices, and that no single cut has both.
\item \emph{A physical model.}  The $3\,655$-piece facet pattern of Figure~\ref{fig:facets} is a
three-dimensional object and could be printed; the $165$-piece dodecahedron already exists as a
set of blocks \cite{GoodmanStrauss2008}.
\end{enumerate}

\section*{Code and data}
The code is a few thousand lines of C++17 with no dependencies beyond the standard library
(OpenGL for the interactive viewer, Blender for the figures), plus the Python script
\texttt{check\_ring.py} (standard library only) that recomputes every identity, index and census
quoted here and whose output is Appendix~\ref{app:checks}.  The repository
contains, for each computation, a script that rebuilds the binaries, reruns it and regenerates the
figures: the dodecahedron ($40$\,s), the 120-cell ($20$\,min including sections and animations),
the balls, and the exact checks ($1$\,min).  Outputs include one representative per orbit with
its full face lattice, the facet patterns, the sections, and an index of all $\NSTAR$ cells
(sign vector, class, centroid, volume).

\clearpage  

\clearpage
\appendix
\section{Catalogue of the star cut}\label{app:catalogue}
\footnotesize
\begin{longtable}{rrllr}
\caption{The $338$ orbits of the star cut of the 120-cell, in order of decreasing volume. A class is identified by its $f$-vector $(f_0,f_1,f_2,f_3)$, the vertex counts of its facets (exponents are multiplicities) and its $4$-volume; the number of cells is the orbit size.}\label{tab:catalogue}\\
\toprule
class & cells & $f$-vector & facets & $4$-volume\\
\midrule
\endfirsthead
\toprule
class & cells & $f$-vector & facets & $4$-volume\\
\midrule
\endhead
\midrule\multicolumn{5}{r}{\small continued}\\
\endfoot
\bottomrule
\endlastfoot
0 & 120 & $(53,190,180,43)$ & $6^{30} 7^{12} 32$ & 0.141153 \\
1 & 1 & $(600,1200,720,120)$ & $20^{120}$ & 0.121581 \\
2 & 3\,600 & $(19,46,41,14)$ & $4^{4} 7^{4} 8^{4} 10^{2}$ & 0.004837 \\
3 & 1\,440 & $(8,22,25,11)$ & $4^{10} 7$ & 0.004187 \\
4 & 2\,400 & $(26,72,63,17)$ & $6^{3} 7^{3} 8^{5} 9^{6}$ & 0.003717 \\
5 & 3\,600 & $(17,43,42,16)$ & $4^{4} 5^{2} 6^{4} 8^{6}$ & 0.003324 \\
6 & 2\,400 & $(6,15,18,9)$ & $4^{9}$ & 0.002512 \\
7 & 2\,400 & $(15,37,33,11)$ & $6^{6} 7^{3} 8^{2}$ & 0.002426 \\
8 & 1\,440 & $(19,53,50,16)$ & $5^{5} 7^{11}$ & 0.002283 \\
9 & 1\,440 & $(19,54,55,20)$ & $4^{5} 5^{5} 7^{10}$ & 0.002138 \\
10 & 3\,600 & $(16,36,32,12)$ & $6^{8} 8^{4}$ & 0.001911 \\
11 & 600 & $(16,48,48,16)$ & $5^{12} 8^{4}$ & 0.001812 \\
12 & 3\,600 & $(24,53,42,13)$ & $7^{2} 8^{10} 17$ & 0.001675 \\
13 & 1\,200 & $(23,59,51,15)$ & $7^{6} 8^{6} 10^{3}$ & 0.001508 \\
14 & 1\,200 & $(10,31,39,18)$ & $4^{12} 5^{6}$ & 0.001466 \\
15 & 3\,600 & $(8,21,22,9)$ & $4^{4} 5^{4} 6$ & 0.001337 \\
16 & 720 & $(9,29,40,20)$ & $4^{20}$ & 0.001222 \\
17 & 14\,400 & $(11,23,19,7)$ & $4 6^{2} 7^{3} 8$ & 0.001117 \\
18 & 7\,200 & $(13,33,32,12)$ & $4^{6} 6 7^{4} 9$ & 0.001080 \\
19 & 3\,600 & $(13,28,23,8)$ & $5^{2} 7^{4} 8^{2}$ & 0.001077 \\
20 & 7\,200 & $(14,29,23,8)$ & $5^{2} 8^{6}$ & 0.001011 \\
21 & 720 & $(14,44,50,20)$ & $5^{20}$ & 0.000933 \\
22 & 600 & $(10,34,48,24)$ & $4^{24}$ & 0.000906 \\
23 & 14\,400 & $(13,30,26,9)$ & $5 6^{4} 7^{2} 8^{2}$ & 0.000864 \\
24 & 3\,600 & $(5,10,10,5)$ & $4^{5}$ & 0.000837 \\
25 & 600 & $(16,58,66,24)$ & $5^{24}$ & 0.000692 \\
26 & 120 & $(64,152,120,32)$ & $8^{20} 12^{12}$ & 0.000626 \\
27 & 3\,600 & $(10,29,32,13)$ & $4^{6} 5^{4} 6^{3}$ & 0.000618 \\
28 & 120 & $(64,152,120,32)$ & $8^{20} 12^{12}$ & 0.000591 \\
29 & 7\,200 & $(16,43,39,12)$ & $5^{2} 6^{6} 7^{2} 9^{2}$ & 0.000568 \\
30 & 7\,200 & $(10,21,18,7)$ & $5^{2} 6^{4} 8$ & 0.000489 \\
31 & 3\,600 & $(8,20,20,8)$ & $4^{2} 5^{5} 6$ & 0.000471 \\
32 & 7\,200 & $(13,29,25,9)$ & $5 6^{4} 7^{2} 8^{2}$ & 0.000468 \\
33 & 2\,400 & $(18,41,36,13)$ & $4^{6} 8^{6} 14$ & 0.000447 \\
34 & 3\,600 & $(6,15,18,9)$ & $4^{9}$ & 0.000442 \\
35 & 7\,200 & $(7,18,20,9)$ & $4^{8} 6$ & 0.000395 \\
36 & 7\,200 & $(9,22,22,9)$ & $4^{5} 6^{3} 7$ & 0.000393 \\
37 & 7\,200 & $(14,33,29,10)$ & $5^{3} 6^{2} 7^{2} 8^{3}$ & 0.000383 \\
38 & 7\,200 & $(8,18,17,7)$ & $4^{3} 5^{2} 6 7$ & 0.000305 \\
39 & 7\,200 & $(7,17,18,8)$ & $4^{6} 5 6$ & 0.000297 \\
40 & 7\,200 & $(12,29,26,9)$ & $5^{2} 6^{4} 7^{3}$ & 0.000295 \\
41 & 7\,200 & $(8,20,20,8)$ & $4^{4} 5^{2} 6^{2}$ & 0.000286 \\
42 & 14\,400 & $(12,25,21,8)$ & $4^{3} 6^{2} 8 9^{2}$ & 0.000268 \\
43 & 2\,400 & $(9,24,24,9)$ & $5^{9}$ & 0.000265 \\
44 & 7\,200 & $(8,18,17,7)$ & $4^{3} 5^{2} 6 7$ & 0.000259 \\
45 & 120 & $(21,50,42,13)$ & $6^{12} 20$ & 0.000239 \\
46 & 7\,200 & $(10,22,20,8)$ & $4^{2} 5 6^{4} 8$ & 0.000238 \\
47 & 7\,200 & $(9,20,18,7)$ & $5^{6} 8$ & 0.000231 \\
48 & 7\,200 & $(10,26,28,12)$ & $4^{6} 5^{2} 6^{4}$ & 0.000231 \\
49 & 14\,400 & $(8,16,14,6)$ & $4^{2} 6^{4}$ & 0.000227 \\
50 & 7\,200 & $(11,23,19,7)$ & $6^{4} 7^{2} 8$ & 0.000211 \\
51 & 7\,200 & $(10,23,21,8)$ & $4^{2} 5^{2} 6^{2} 7^{2}$ & 0.000209 \\
52 & 7\,200 & $(5,10,10,5)$ & $4^{5}$ & 0.000198 \\
53 & 7\,200 & $(5,10,10,5)$ & $4^{5}$ & 0.000198 \\
54 & 7\,200 & $(9,22,22,9)$ & $4^{2} 5^{5} 6^{2}$ & 0.000187 \\
55 & 720 & $(8,22,26,12)$ & $4^{10} 6^{2}$ & 0.000169 \\
56 & 1\,200 & $(9,24,24,9)$ & $5^{9}$ & 0.000164 \\
57 & 3\,600 & $(9,25,27,11)$ & $4^{4} 5^{7}$ & 0.000164 \\
58 & 7\,200 & $(7,18,20,9)$ & $4^{6} 5^{3}$ & 0.000151 \\
59 & 7\,200 & $(10,22,20,8)$ & $4^{4} 6^{2} 8^{2}$ & 0.000148 \\
60 & 7\,200 & $(11,27,25,9)$ & $4^{2} 6^{5} 7^{2}$ & 0.000146 \\
61 & 7\,200 & $(9,20,18,7)$ & $4^{2} 5^{2} 6^{2} 8$ & 0.000142 \\
62 & 7\,200 & $(11,28,26,9)$ & $5^{4} 6^{4} 7$ & 0.000140 \\
63 & 14\,400 & $(9,20,18,7)$ & $4 5 6^{5}$ & 0.000133 \\
64 & 14\,400 & $(10,21,18,7)$ & $4 6^{4} 7^{2}$ & 0.000133 \\
65 & 2\,400 & $(16,45,42,13)$ & $5^{7} 7^{3} 8^{3}$ & 0.000128 \\
66 & 14\,400 & $(8,16,14,6)$ & $4^{2} 6^{4}$ & 0.000128 \\
67 & 7\,200 & $(6,14,16,8)$ & $4^{8}$ & 0.000122 \\
68 & 3\,600 & $(6,14,16,8)$ & $4^{8}$ & 0.000122 \\
69 & 14\,400 & $(10,21,18,7)$ & $4^{2} 5 6 7^{2} 8$ & 0.000120 \\
70 & 14\,400 & $(10,21,18,7)$ & $5^{2} 6^{4} 8$ & 0.000119 \\
71 & 7\,200 & $(12,27,24,9)$ & $6^{8} 8$ & 0.000118 \\
72 & 7\,200 & $(12,28,25,9)$ & $5^{4} 6^{4} 10$ & 0.000117 \\
73 & 3\,600 & $(7,17,18,8)$ & $4^{4} 5^{4}$ & 0.000115 \\
74 & 7\,200 & $(9,19,17,7)$ & $4^{3} 6^{2} 7^{2}$ & 0.000114 \\
75 & 2\,400 & $(6,14,15,7)$ & $4^{6} 5$ & 0.000113 \\
76 & 14\,400 & $(7,15,14,6)$ & $4^{2} 5^{3} 6$ & 0.000111 \\
77 & 3\,600 & $(6,14,16,8)$ & $4^{8}$ & 0.000109 \\
78 & 7\,200 & $(9,19,17,7)$ & $4^{3} 6^{2} 7^{2}$ & 0.000108 \\
79 & 2\,400 & $(9,23,24,10)$ & $4^{6} 5^{3} 8$ & 0.000107 \\
80 & 2\,400 & $(9,24,24,9)$ & $5^{9}$ & 0.000107 \\
81 & 3\,600 & $(9,25,28,12)$ & $4^{6} 5^{6}$ & 0.000102 \\
82 & 7\,200 & $(10,24,23,9)$ & $4^{2} 5^{4} 6^{2} 8$ & 0.000100 \\
83 & 3\,600 & $(9,25,28,12)$ & $4^{6} 5^{6}$ & 0.000099 \\
84 & 7\,200 & $(6,13,13,6)$ & $4^{4} 5^{2}$ & 0.000099 \\
85 & 3\,600 & $(12,31,32,13)$ & $4^{4} 6^{9}$ & 0.000092 \\
86 & 7\,200 & $(6,14,15,7)$ & $4^{6} 5$ & 0.000088 \\
87 & 7\,200 & $(10,22,20,8)$ & $4^{2} 5 6^{4} 8$ & 0.000087 \\
88 & 14\,400 & $(10,21,18,7)$ & $4 6^{4} 7^{2}$ & 0.000087 \\
89 & 7\,200 & $(6,14,15,7)$ & $4^{6} 5$ & 0.000076 \\
90 & 14\,400 & $(5,10,10,5)$ & $4^{5}$ & 0.000076 \\
91 & 3\,600 & $(6,14,16,8)$ & $4^{8}$ & 0.000076 \\
92 & 3\,600 & $(7,18,21,10)$ & $4^{8} 5^{2}$ & 0.000071 \\
93 & 3\,600 & $(7,19,24,12)$ & $4^{12}$ & 0.000071 \\
94 & 2\,400 & $(9,24,24,9)$ & $5^{9}$ & 0.000066 \\
95 & 3\,600 & $(9,24,24,9)$ & $5^{9}$ & 0.000066 \\
96 & 3\,600 & $(11,27,27,11)$ & $4^{4} 5 6^{4} 7^{2}$ & 0.000063 \\
97 & 7\,200 & $(15,36,32,11)$ & $5^{2} 6^{4} 7^{2} 8^{3}$ & 0.000061 \\
98 & 7\,200 & $(16,42,38,12)$ & $5^{3} 6^{5} 7^{2} 8 10$ & 0.000061 \\
99 & 7\,200 & $(9,19,17,7)$ & $4^{3} 6^{2} 7^{2}$ & 0.000061 \\
100 & 14\,400 & $(5,10,10,5)$ & $4^{5}$ & 0.000061 \\
101 & 7\,200 & $(5,10,10,5)$ & $4^{5}$ & 0.000061 \\
102 & 2\,400 & $(13,43,45,15)$ & $5^{12} 6^{3}$ & 0.000058 \\
103 & 7\,200 & $(6,13,13,6)$ & $4^{4} 5^{2}$ & 0.000058 \\
104 & 720 & $(8,23,30,15)$ & $4^{15}$ & 0.000055 \\
105 & 14\,400 & $(5,10,10,5)$ & $4^{5}$ & 0.000055 \\
106 & 7\,200 & $(10,26,27,11)$ & $4^{6} 5^{2} 6^{2} 8$ & 0.000055 \\
107 & 7\,200 & $(9,19,17,7)$ & $4^{3} 6^{2} 7^{2}$ & 0.000054 \\
108 & 1\,200 & $(9,27,33,15)$ & $4^{12} 5^{3}$ & 0.000053 \\
109 & 7\,200 & $(7,17,18,8)$ & $4^{6} 5 6$ & 0.000052 \\
110 & 600 & $(6,14,16,8)$ & $4^{8}$ & 0.000052 \\
111 & 7\,200 & $(11,27,25,9)$ & $4^{2} 6^{5} 7^{2}$ & 0.000051 \\
112 & 2\,400 & $(6,14,15,7)$ & $4^{6} 5$ & 0.000051 \\
113 & 14\,400 & $(7,16,16,7)$ & $4^{4} 5^{2} 6$ & 0.000050 \\
114 & 1\,200 & $(9,24,27,12)$ & $4^{6} 5^{6}$ & 0.000048 \\
115 & 720 & $(20,53,50,17)$ & $6^{10} 7^{2} 8^{5}$ & 0.000047 \\
116 & 14\,400 & $(5,10,10,5)$ & $4^{5}$ & 0.000047 \\
117 & 7\,200 & $(6,14,15,7)$ & $4^{6} 5$ & 0.000047 \\
118 & 7\,200 & $(6,14,16,8)$ & $4^{8}$ & 0.000047 \\
119 & 7\,200 & $(6,14,16,8)$ & $4^{8}$ & 0.000047 \\
120 & 1\,440 & $(18,42,36,12)$ & $5^{5} 6 7^{5} 16$ & 0.000046 \\
121 & 14\,400 & $(6,13,13,6)$ & $4^{4} 5^{2}$ & 0.000044 \\
122 & 7\,200 & $(9,21,20,8)$ & $4^{2} 5^{5} 8$ & 0.000044 \\
123 & 7\,200 & $(6,13,13,6)$ & $4^{4} 5^{2}$ & 0.000042 \\
124 & 3\,600 & $(6,14,16,8)$ & $4^{8}$ & 0.000042 \\
125 & 14\,400 & $(8,19,19,8)$ & $4^{4} 5^{2} 6^{2}$ & 0.000041 \\
126 & 2\,400 & $(9,27,33,15)$ & $4^{12} 5^{3}$ & 0.000041 \\
127 & 7\,200 & $(7,17,17,7)$ & $4^{3} 5^{4}$ & 0.000039 \\
128 & 7\,200 & $(6,13,13,6)$ & $4^{4} 5^{2}$ & 0.000038 \\
129 & 7\,200 & $(5,10,10,5)$ & $4^{5}$ & 0.000038 \\
130 & 14\,400 & $(5,10,10,5)$ & $4^{5}$ & 0.000038 \\
131 & 14\,400 & $(8,18,17,7)$ & $4^{3} 5 6^{3}$ & 0.000036 \\
132 & 3\,600 & $(7,18,21,10)$ & $4^{8} 5^{2}$ & 0.000036 \\
133 & 7\,200 & $(6,13,13,6)$ & $4^{4} 5^{2}$ & 0.000036 \\
134 & 7\,200 & $(9,22,21,8)$ & $5^{6} 6^{2}$ & 0.000034 \\
135 & 7\,200 & $(6,13,13,6)$ & $4^{4} 5^{2}$ & 0.000034 \\
136 & 1\,440 & $(24,58,50,16)$ & $5^{5} 8^{10} 12$ & 0.000033 \\
137 & 3\,600 & $(7,18,21,10)$ & $4^{8} 5^{2}$ & 0.000032 \\
138 & 7\,200 & $(7,17,18,8)$ & $4^{5} 5^{3}$ & 0.000031 \\
139 & 1\,440 & $(13,32,30,11)$ & $5^{10} 12$ & 0.000030 \\
140 & 2\,400 & $(9,24,24,9)$ & $5^{9}$ & 0.000030 \\
141 & 600 & $(16,58,60,18)$ & $5^{12} 6^{6}$ & 0.000030 \\
142 & 7\,200 & $(6,14,15,7)$ & $4^{6} 5$ & 0.000029 \\
143 & 14\,400 & $(6,13,13,6)$ & $4^{4} 5^{2}$ & 0.000029 \\
144 & 7\,200 & $(12,26,22,8)$ & $6^{6} 8^{2}$ & 0.000028 \\
145 & 720 & $(38,82,61,17)$ & $8^{10} 11^{5} 16^{2}$ & 0.000027 \\
146 & 1\,200 & $(6,14,16,8)$ & $4^{8}$ & 0.000026 \\
147 & 720 & $(7,16,16,7)$ & $4^{5} 6^{2}$ & 0.000025 \\
148 & 14\,400 & $(5,10,10,5)$ & $4^{5}$ & 0.000023 \\
149 & 7\,200 & $(5,10,10,5)$ & $4^{5}$ & 0.000023 \\
150 & 14\,400 & $(8,19,19,8)$ & $4^{4} 5^{3} 7$ & 0.000022 \\
151 & 7\,200 & $(7,18,20,9)$ & $4^{6} 5^{3}$ & 0.000022 \\
152 & 7\,200 & $(6,13,13,6)$ & $4^{4} 5^{2}$ & 0.000022 \\
153 & 7\,200 & $(6,13,13,6)$ & $4^{4} 5^{2}$ & 0.000022 \\
154 & 7\,200 & $(6,13,13,6)$ & $4^{4} 5^{2}$ & 0.000021 \\
155 & 7\,200 & $(5,10,10,5)$ & $4^{5}$ & 0.000021 \\
156 & 3\,600 & $(8,20,20,8)$ & $4^{4} 5^{2} 6^{2}$ & 0.000020 \\
157 & 7\,200 & $(7,17,18,8)$ & $4^{6} 5 6$ & 0.000020 \\
158 & 7\,200 & $(6,14,15,7)$ & $4^{6} 5$ & 0.000020 \\
159 & 2\,400 & $(7,18,21,10)$ & $4^{9} 5$ & 0.000019 \\
160 & 14\,400 & $(7,15,14,6)$ & $4^{2} 5^{3} 6$ & 0.000019 \\
161 & 7\,200 & $(7,17,18,8)$ & $4^{5} 5^{3}$ & 0.000019 \\
162 & 14\,400 & $(7,15,14,6)$ & $4^{2} 5^{3} 6$ & 0.000019 \\
163 & 14\,400 & $(9,20,18,7)$ & $4 5 6^{5}$ & 0.000018 \\
164 & 2\,400 & $(9,24,27,12)$ & $4^{6} 5^{6}$ & 0.000018 \\
165 & 7\,200 & $(6,13,13,6)$ & $4^{4} 5^{2}$ & 0.000018 \\
166 & 7\,200 & $(7,17,17,7)$ & $4^{3} 5^{4}$ & 0.000018 \\
167 & 7\,200 & $(6,14,15,7)$ & $4^{6} 5$ & 0.000018 \\
168 & 3\,600 & $(6,14,16,8)$ & $4^{8}$ & 0.000018 \\
169 & 14\,400 & $(5,10,10,5)$ & $4^{5}$ & 0.000018 \\
170 & 3\,600 & $(12,30,28,10)$ & $5^{6} 6^{2} 8^{2}$ & 0.000017 \\
171 & 3\,600 & $(9,24,25,10)$ & $4^{2} 5^{8}$ & 0.000017 \\
172 & 7\,200 & $(10,26,26,10)$ & $4^{2} 5^{4} 6^{4}$ & 0.000017 \\
173 & 14\,400 & $(8,16,14,6)$ & $4^{2} 6^{4}$ & 0.000017 \\
174 & 7\,200 & $(8,18,17,7)$ & $4^{3} 5^{2} 6 7$ & 0.000017 \\
175 & 7\,200 & $(5,10,10,5)$ & $4^{5}$ & 0.000017 \\
176 & 14\,400 & $(6,13,13,6)$ & $4^{4} 5^{2}$ & 0.000017 \\
177 & 14\,400 & $(7,15,14,6)$ & $4^{2} 5^{3} 6$ & 0.000016 \\
178 & 7\,200 & $(6,14,15,7)$ & $4^{6} 5$ & 0.000016 \\
179 & 1\,200 & $(7,19,24,12)$ & $4^{12}$ & 0.000015 \\
180 & 14\,400 & $(5,10,10,5)$ & $4^{5}$ & 0.000014 \\
181 & 14\,400 & $(5,10,10,5)$ & $4^{5}$ & 0.000014 \\
182 & 14\,400 & $(5,10,10,5)$ & $4^{5}$ & 0.000014 \\
183 & 7\,200 & $(6,13,13,6)$ & $4^{4} 5^{2}$ & 0.000014 \\
184 & 7\,200 & $(6,14,15,7)$ & $4^{6} 5$ & 0.000014 \\
185 & 7\,200 & $(6,13,13,6)$ & $4^{4} 5^{2}$ & 0.000014 \\
186 & 7\,200 & $(6,14,15,7)$ & $4^{6} 5$ & 0.000014 \\
187 & 3\,600 & $(6,14,16,8)$ & $4^{8}$ & 0.000014 \\
188 & 7\,200 & $(6,14,15,7)$ & $4^{6} 5$ & 0.000014 \\
189 & 14\,400 & $(7,16,16,7)$ & $4^{4} 5^{2} 6$ & 0.000013 \\
190 & 7\,200 & $(10,21,18,7)$ & $5^{2} 6^{4} 8$ & 0.000012 \\
191 & 1\,440 & $(13,32,30,11)$ & $5^{10} 12$ & 0.000012 \\
192 & 7\,200 & $(10,25,25,10)$ & $4^{4} 5^{3} 6^{2} 8$ & 0.000012 \\
193 & 3\,600 & $(6,14,16,8)$ & $4^{8}$ & 0.000011 \\
194 & 7\,200 & $(6,14,16,8)$ & $4^{8}$ & 0.000011 \\
195 & 14\,400 & $(5,10,10,5)$ & $4^{5}$ & 0.000011 \\
196 & 7\,200 & $(5,10,10,5)$ & $4^{5}$ & 0.000011 \\
197 & 1\,200 & $(5,10,10,5)$ & $4^{5}$ & 0.000011 \\
198 & 14\,400 & $(8,16,14,6)$ & $4^{2} 6^{4}$ & 0.000011 \\
199 & 14\,400 & $(8,18,17,7)$ & $4^{3} 5 6^{3}$ & 0.000010 \\
200 & 7\,200 & $(8,18,17,7)$ & $4^{3} 5^{2} 6 7$ & 0.000010 \\
201 & 2\,400 & $(6,14,15,7)$ & $4^{6} 5$ & 0.000010 \\
202 & 3\,600 & $(6,13,13,6)$ & $4^{4} 5^{2}$ & 0.000010 \\
203 & 7\,200 & $(8,21,22,9)$ & $4^{4} 5^{5}$ & 0.000010 \\
204 & 3\,600 & $(7,17,18,8)$ & $4^{4} 5^{4}$ & 0.000010 \\
205 & 7\,200 & $(7,18,20,9)$ & $4^{8} 6$ & 0.000009 \\
206 & 1\,440 & $(13,32,30,11)$ & $5^{10} 12$ & 0.000009 \\
207 & 14\,400 & $(5,10,10,5)$ & $4^{5}$ & 0.000009 \\
208 & 7\,200 & $(5,10,10,5)$ & $4^{5}$ & 0.000009 \\
209 & 14\,400 & $(5,10,10,5)$ & $4^{5}$ & 0.000008 \\
210 & 7\,200 & $(6,14,15,7)$ & $4^{6} 5$ & 0.000008 \\
211 & 3\,600 & $(6,14,16,8)$ & $4^{8}$ & 0.000008 \\
212 & 7\,200 & $(6,14,16,8)$ & $4^{8}$ & 0.000008 \\
213 & 14\,400 & $(8,16,14,6)$ & $4^{2} 6^{4}$ & 0.000008 \\
214 & 3\,600 & $(6,14,15,7)$ & $4^{6} 5$ & 0.000008 \\
215 & 2\,400 & $(9,20,18,7)$ & $5^{6} 8$ & 0.000007 \\
216 & 14\,400 & $(6,13,13,6)$ & $4^{4} 5^{2}$ & 0.000007 \\
217 & 7\,200 & $(6,13,13,6)$ & $4^{4} 5^{2}$ & 0.000007 \\
218 & 7\,200 & $(6,14,15,7)$ & $4^{6} 5$ & 0.000007 \\
219 & 3\,600 & $(6,14,16,8)$ & $4^{8}$ & 0.000007 \\
220 & 3\,600 & $(6,14,15,7)$ & $4^{6} 5$ & 0.000007 \\
221 & 14\,400 & $(5,10,10,5)$ & $4^{5}$ & 0.000007 \\
222 & 7\,200 & $(7,17,18,8)$ & $4^{5} 5^{3}$ & 0.000007 \\
223 & 7\,200 & $(8,20,20,8)$ & $4^{2} 5^{6}$ & 0.000007 \\
224 & 7\,200 & $(9,22,22,9)$ & $4^{2} 5^{5} 6^{2}$ & 0.000007 \\
225 & 2\,400 & $(7,18,21,10)$ & $4^{9} 5$ & 0.000006 \\
226 & 7\,200 & $(6,13,13,6)$ & $4^{4} 5^{2}$ & 0.000006 \\
227 & 3\,600 & $(6,14,16,8)$ & $4^{8}$ & 0.000006 \\
228 & 14\,400 & $(6,13,13,6)$ & $4^{4} 5^{2}$ & 0.000006 \\
229 & 3\,600 & $(9,24,24,9)$ & $5^{9}$ & 0.000006 \\
230 & 720 & $(7,16,16,7)$ & $4^{5} 6^{2}$ & 0.000006 \\
231 & 7\,200 & $(5,10,10,5)$ & $4^{5}$ & 0.000006 \\
232 & 14\,400 & $(5,10,10,5)$ & $4^{5}$ & 0.000006 \\
233 & 14\,400 & $(6,13,13,6)$ & $4^{4} 5^{2}$ & 0.000006 \\
234 & 14\,400 & $(5,10,10,5)$ & $4^{5}$ & 0.000006 \\
235 & 14\,400 & $(5,10,10,5)$ & $4^{5}$ & 0.000006 \\
236 & 3\,600 & $(7,18,20,9)$ & $4^{8} 6$ & 0.000005 \\
237 & 7\,200 & $(6,13,13,6)$ & $4^{4} 5^{2}$ & 0.000005 \\
238 & 3\,600 & $(5,10,10,5)$ & $4^{5}$ & 0.000005 \\
239 & 7\,200 & $(5,10,10,5)$ & $4^{5}$ & 0.000005 \\
240 & 3\,600 & $(7,18,21,10)$ & $4^{8} 5^{2}$ & 0.000005 \\
241 & 3\,600 & $(6,14,15,7)$ & $4^{6} 5$ & 0.000005 \\
242 & 2\,400 & $(9,24,27,12)$ & $4^{6} 5^{6}$ & 0.000005 \\
243 & 600 & $(10,30,30,10)$ & $5^{10}$ & 0.000004 \\
244 & 7\,200 & $(6,14,15,7)$ & $4^{6} 5$ & 0.000004 \\
245 & 14\,400 & $(5,10,10,5)$ & $4^{5}$ & 0.000004 \\
246 & 7\,200 & $(6,14,15,7)$ & $4^{6} 5$ & 0.000004 \\
247 & 14\,400 & $(7,16,16,7)$ & $4^{4} 5^{2} 6$ & 0.000004 \\
248 & 14\,400 & $(5,10,10,5)$ & $4^{5}$ & 0.000004 \\
249 & 1\,440 & $(8,22,25,11)$ & $4^{10} 7$ & 0.000004 \\
250 & 7\,200 & $(5,10,10,5)$ & $4^{5}$ & 0.000004 \\
251 & 3\,600 & $(11,22,18,7)$ & $4^{2} 6^{2} 8^{3}$ & 0.000004 \\
252 & 1\,200 & $(9,24,27,12)$ & $4^{6} 5^{6}$ & 0.000004 \\
253 & 3\,600 & $(7,18,21,10)$ & $4^{8} 5^{2}$ & 0.000004 \\
254 & 3\,600 & $(7,17,18,8)$ & $4^{4} 5^{4}$ & 0.000004 \\
255 & 2\,400 & $(6,14,15,7)$ & $4^{6} 5$ & 0.000003 \\
256 & 2\,400 & $(9,20,18,7)$ & $5^{6} 8$ & 0.000003 \\
257 & 14\,400 & $(5,10,10,5)$ & $4^{5}$ & 0.000003 \\
258 & 7\,200 & $(9,21,20,8)$ & $4^{2} 5^{5} 8$ & 0.000003 \\
259 & 14\,400 & $(7,15,14,6)$ & $4^{2} 5^{3} 6$ & 0.000003 \\
260 & 14\,400 & $(8,19,19,8)$ & $4^{4} 5^{3} 7$ & 0.000003 \\
261 & 3\,600 & $(6,13,13,6)$ & $4^{4} 5^{2}$ & 0.000003 \\
262 & 7\,200 & $(6,13,13,6)$ & $4^{4} 5^{2}$ & 0.000003 \\
263 & 14\,400 & $(6,13,13,6)$ & $4^{4} 5^{2}$ & 0.000003 \\
264 & 14\,400 & $(6,13,13,6)$ & $4^{4} 5^{2}$ & 0.000003 \\
265 & 3\,600 & $(5,10,10,5)$ & $4^{5}$ & 0.000003 \\
266 & 7\,200 & $(6,13,13,6)$ & $4^{4} 5^{2}$ & 0.000003 \\
267 & 7\,200 & $(5,10,10,5)$ & $4^{5}$ & 0.000003 \\
268 & 14\,400 & $(6,13,13,6)$ & $4^{4} 5^{2}$ & 0.000003 \\
269 & 7\,200 & $(9,22,21,8)$ & $5^{6} 6^{2}$ & 0.000003 \\
270 & 7\,200 & $(5,10,10,5)$ & $4^{5}$ & 0.000002 \\
271 & 7\,200 & $(6,13,13,6)$ & $4^{4} 5^{2}$ & 0.000002 \\
272 & 3\,600 & $(9,24,25,10)$ & $4^{2} 5^{8}$ & 0.000002 \\
273 & 7\,200 & $(8,16,14,6)$ & $4^{2} 6^{4}$ & 0.000002 \\
274 & 14\,400 & $(5,10,10,5)$ & $4^{5}$ & 0.000002 \\
275 & 14\,400 & $(5,10,10,5)$ & $4^{5}$ & 0.000002 \\
276 & 3\,600 & $(12,28,25,9)$ & $5^{6} 6^{2} 11$ & 0.000002 \\
277 & 7\,200 & $(6,14,15,7)$ & $4^{6} 5$ & 0.000002 \\
278 & 7\,200 & $(7,17,18,8)$ & $4^{4} 5^{4}$ & 0.000002 \\
279 & 7\,200 & $(6,13,13,6)$ & $4^{4} 5^{2}$ & 0.000002 \\
280 & 7\,200 & $(6,13,13,6)$ & $4^{4} 5^{2}$ & 0.000002 \\
281 & 3\,600 & $(6,14,15,7)$ & $4^{6} 5$ & 0.000002 \\
282 & 7\,200 & $(6,14,15,7)$ & $4^{6} 5$ & 0.000002 \\
283 & 3\,600 & $(6,14,15,7)$ & $4^{6} 5$ & 0.000002 \\
284 & 14\,400 & $(5,10,10,5)$ & $4^{5}$ & 0.000002 \\
285 & 7\,200 & $(7,17,18,8)$ & $4^{4} 5^{4}$ & 0.000002 \\
286 & 7\,200 & $(6,13,13,6)$ & $4^{4} 5^{2}$ & 0.000002 \\
287 & 14\,400 & $(6,13,13,6)$ & $4^{4} 5^{2}$ & 0.000001 \\
288 & 14\,400 & $(5,10,10,5)$ & $4^{5}$ & 0.000001 \\
289 & 7\,200 & $(5,10,10,5)$ & $4^{5}$ & 0.000001 \\
290 & 7\,200 & $(5,10,10,5)$ & $4^{5}$ & 0.000001 \\
291 & 14\,400 & $(5,10,10,5)$ & $4^{5}$ & 0.000001 \\
292 & 14\,400 & $(5,10,10,5)$ & $4^{5}$ & 0.000001 \\
293 & 3\,600 & $(7,17,18,8)$ & $4^{4} 5^{4}$ & 0.000001 \\
294 & 3\,600 & $(6,14,16,8)$ & $4^{8}$ & 0.000001 \\
295 & 7\,200 & $(9,20,18,7)$ & $5^{6} 8$ & 0.000001 \\
296 & 3\,600 & $(7,18,20,9)$ & $4^{8} 6$ & 0.000001 \\
297 & 7\,200 & $(6,13,13,6)$ & $4^{4} 5^{2}$ & 0.000001 \\
298 & 14\,400 & $(5,10,10,5)$ & $4^{5}$ & 0.000001 \\
299 & 7\,200 & $(6,13,13,6)$ & $4^{4} 5^{2}$ & 0.000001 \\
300 & 7\,200 & $(7,17,18,8)$ & $4^{4} 5^{4}$ & 0.000001 \\
301 & 7\,200 & $(6,13,13,6)$ & $4^{4} 5^{2}$ & 0.000001 \\
302 & 7\,200 & $(5,10,10,5)$ & $4^{5}$ & 0.000001 \\
303 & 7\,200 & $(6,14,15,7)$ & $4^{6} 5$ & 0.000001 \\
304 & 7\,200 & $(6,13,13,6)$ & $4^{4} 5^{2}$ & 0.000001 \\
305 & 7\,200 & $(5,10,10,5)$ & $4^{5}$ & 0.000001 \\
306 & 7\,200 & $(6,13,13,6)$ & $4^{4} 5^{2}$ & 0.000001 \\
307 & 14\,400 & $(5,10,10,5)$ & $4^{5}$ & 0.000001 \\
308 & 14\,400 & $(5,10,10,5)$ & $4^{5}$ & 0.000001 \\
309 & 7\,200 & $(6,14,15,7)$ & $4^{6} 5$ & 0.000001 \\
310 & 14\,400 & $(5,10,10,5)$ & $4^{5}$ & 0.000001 \\
311 & 7\,200 & $(9,20,18,7)$ & $5^{6} 8$ & 0.000001 \\
312 & 14\,400 & $(5,10,10,5)$ & $4^{5}$ & 0.000001 \\
313 & 14\,400 & $(5,10,10,5)$ & $4^{5}$ & 0.000001 \\
314 & 7\,200 & $(6,14,15,7)$ & $4^{6} 5$ & 0.000001 \\
315 & 7\,200 & $(6,13,13,6)$ & $4^{4} 5^{2}$ & 0.000001 \\
316 & 7\,200 & $(6,13,13,6)$ & $4^{4} 5^{2}$ & 0.000001 \\
317 & 14\,400 & $(5,10,10,5)$ & $4^{5}$ & 0.000001 \\
318 & 7\,200 & $(6,13,13,6)$ & $4^{4} 5^{2}$ & 0.000001 \\
319 & 7\,200 & $(6,13,13,6)$ & $4^{4} 5^{2}$ & 0.000001 \\
320 & 7\,200 & $(6,13,13,6)$ & $4^{4} 5^{2}$ & 0.000001 \\
321 & 14\,400 & $(5,10,10,5)$ & $4^{5}$ & 0.000001 \\
322 & 14\,400 & $(5,10,10,5)$ & $4^{5}$ & 0.000001 \\
323 & 14\,400 & $(5,10,10,5)$ & $4^{5}$ & 0.000001 \\
324 & 7\,200 & $(5,10,10,5)$ & $4^{5}$ & 0.000001 \\
325 & 3\,600 & $(6,13,13,6)$ & $4^{4} 5^{2}$ & 0.000001 \\
326 & 7\,200 & $(5,10,10,5)$ & $4^{5}$ & 0.000000 \\
327 & 7\,200 & $(6,14,15,7)$ & $4^{6} 5$ & 0.000000 \\
328 & 7\,200 & $(5,10,10,5)$ & $4^{5}$ & 0.000000 \\
329 & 7\,200 & $(5,10,10,5)$ & $4^{5}$ & 0.000000 \\
330 & 7\,200 & $(6,13,13,6)$ & $4^{4} 5^{2}$ & 0.000000 \\
331 & 14\,400 & $(6,13,13,6)$ & $4^{4} 5^{2}$ & 0.000000 \\
332 & 14\,400 & $(5,10,10,5)$ & $4^{5}$ & 0.000000 \\
333 & 14\,400 & $(6,13,13,6)$ & $4^{4} 5^{2}$ & 0.000000 \\
334 & 7\,200 & $(5,10,10,5)$ & $4^{5}$ & 0.000000 \\
335 & 14\,400 & $(5,10,10,5)$ & $4^{5}$ & 0.000000 \\
336 & 14\,400 & $(6,13,13,6)$ & $4^{4} 5^{2}$ & 0.000000 \\
337 & 7\,200 & $(5,10,10,5)$ & $4^{5}$ & 0.000000 \\
\end{longtable}

\normalsize

\clearpage
\section{Output of the exact checks}\label{app:checks}
The script \texttt{scripts/check\_ring.py} recomputes, in $\Qphi$ with rational pairs $m+n\varphi$,
every identity of Section~\ref{sec:ring}, the golden-scale test on the ratios of
Table~\ref{tab:scales}, the two sign rules against each other, the vertex layers and critical depths
of the dodecahedron (block~D), its $112$ piece vertices (E), the sign vectors and exact
$I_h$-orbits of its $165$ pieces (F, G), the modules $\M$, $\MV$, $\Ico$ and $E_8$ (H, I), and
the exact vertices of the representative cells of both four-dimensional cuts (J, K).  It reads only
the outputs of the exact slicers and runs in about a minute.
{\footnotesize
\begin{lstlisting}[basicstyle=\ttfamily\scriptsize,frame=none,xleftmargin=0pt,framexleftmargin=0pt]
== A. identities in Z[phi] used in the article ==
[PASS] phi^2 = phi + 1
[PASS] phi^-1 = phi - 1,  phi^-2 = 2 - phi,  phi^-3 = 2phi - 3,  phi^-4 = 5 - 3phi
[PASS] phi^3 = 1 + 2phi,  phi^4 = 2 + 3phi
[PASS] sqrt5 = 2phi - 1 and sqrt5^2 = 5
[PASS] phi^n = F_n phi + F_{n-1} for 1 <= n <= 14
[PASS] conj(phi) = 1 - phi = -phi^-1
[PASS] N(phi) = -1, Tr(phi) = 1;  N(sqrt5) = -5, Tr(sqrt5) = 0
[PASS] N(m + n phi) = m^2 + mn - n^2 is multiplicative (random check)
[PASS] sqrt5/2 is not an algebraic integer (Tr = 0, N = -5/4)
[PASS] 1 - phi^-4 = 3phi - 4 = sqrt5 / phi^2
[PASS] N(3phi - 4) = -5: the hereditary depth is not a unit
[PASS] N(phi^-4) = 1: the star depth is a unit
[PASS] (3phi-4)^4 = 25 phi^-8  (volume fraction of the hereditary core)   (= 0.5322)
[PASS] (phi^-4)^4 = phi^-16  (volume fraction of the star core)   (= 0.000453)
[PASS] 1/sqrt5 = (2phi - 1)/5 = sqrt5/5 has norm -1/5
[PASS] sin^2 72 = (5 + sqrt5)/8 = (2 + phi)/4 has norm 5/16, not a rational square
[PASS] (phi + 2) = sqrt5 phi, so |n|^2 = phi + 2 for the dodecahedron normals (+-1,+-phi,0)
[PASS] 2 phi^-2 = 4 - 2phi  (star-cut offset)
[PASS] (1 - phi^-4) 2 phi^2 = 2 sqrt5  (hereditary offset)
[PASS] 2 phi^-3 - phi^-2 = 5phi - 8 = phi^-5 = 0.0902 (spike height of the small stellated 120-cell)   (= 0.0902)
[PASS] 2 is inert: x^2 - x - 1 has no root mod 2
[PASS] 3 is inert: x^2 - x - 1 has no root mod 3
[PASS] 5 ramifies: 5 = sqrt5^2, sqrt5 = phi^-1 (phi + 2)
[PASS] 11 splits: 11 = N(3 + phi) = (3 + phi)(4 - phi)
[PASS] 19 splits: 19 = N(4 + phi)

== B. the exact sign test (two rules, same answer) ==
[PASS] trace/norm rule == p,q rule == floating point on 5169 pairs (a, b)
[PASS] no tie: p^2 = 5 q^2 has no integer solution (q != 0) up to 10^4
[PASS] F_n phi - F_{n+1} = (-1)^(n+1) phi^-n decided correctly for n = 5..14   (smallest |value| 1.19e-03)

== C. golden scales: lambda = +-phi^a sqrt5^b, decided by the norm ==
  phi^-2 (pentagon depth)                        value   0.381966  N = 1       -> (1, -2, 0)
  phi^-3 (dodecahedron depth)                    value   0.236068  N = -1      -> (1, -3, 0)
  phi^-4 (star depth)                            value   0.145898  N = 1       -> (1, -4, 0)
  3phi-4 = 1-phi^-4 (hereditary depth)           value   0.854102  N = -5      -> (1, -2, 1)
  1/sqrt5 (critical depth, face centres)         value   0.447214  N = -1/5    -> (1, 0, -1)
  phi^-1 (critical depth, edge midpoints)        value   0.618034  N = -1      -> (1, -1, 0)
  2phi^-2 (edge of the 120-cell)                 value   0.763932  N = 4       -> None
  (-21-4sqrt5)/19 (norm 1 but not integral)      value  -1.576014  N = 1       -> None
  5 - 2sqrt5 = 7 - 4phi (facet-parallel cut offset) value   0.527864  N = 5       -> (1, -3, 1)
[PASS] phi^-d are golden scales with b = 0 (units)
[PASS] 3phi - 4 = sqrt5 phi^-2 is a golden scale with b = 1 (not a unit)
[PASS] 1/sqrt5 is a golden scale with b = -1
[PASS] 2 phi^-2 is not a golden scale (rational factor 2)
[PASS] norm-1 non-integral element is rejected

== D. the dodecahedron: planes, vertex layers, critical depths (midradius 1) ==
[PASS] 20 vertices in Z[phi]^3
[PASS] 12 face normals in Z[phi]^3 with |n|^2 = 3 - phi, faces n.x = 1
  vertex heights along a face normal (fractions of the inradius): [-1, 3-2phi, -3+2phi, 1]
[PASS] vertex layers at +-1 and +-phi^-3, with 5 vertices on each cut plane
  face-centre heights: [-1, 1/5-2/5phi, -1/5+2/5phi, 1]
[PASS] face centres at +-1 and +-1/sqrt5 (adjacent unit normals have inner product 1/sqrt5)
  edge-midpoint heights: [-1, 1-phi, 0, -1+phi, 1]
[PASS] edge midpoints at +-1, +-phi^-1 and 0
[PASS] the three critical depths phi^-3 < 1/sqrt5 < phi^-1 are in this order

== E. the 112 vertices of the 165 pieces (24 planes, Cramer in Q(sqrt5)) ==
  distinct intersection points inside D: 112
[PASS] 112 distinct vertices
  rational denominators of the coordinates: {1: 112}
[PASS] all 112 vertices lie in Z[phi]^3 (no denominators)
[PASS] all 112 vertices lie in the vertex module M
  planes through each vertex (3 = simple): {3: 80, 5: 12, 6: 20}
[PASS] 80 simple vertices, 12 on 5 planes (the five-fold axes), 20 on 6 planes
  vertices of D: 20; on the boundary but not vertices: 60; interior: 32
[PASS] 20 are vertices of D, 60 on the boundary, 32 interior
  |N(W)| of the Cramer denominators W that occur: [Fraction(4, 1)]

== F. the 165 pieces as sign vectors; exact I_h-orbits ==
[PASS] 165 pieces in the exact slicer output
[PASS] 165 centroids give 165 distinct sign vectors against the 12 cut planes
  smallest |n.c - phi^-3| over all centroids and cut planes: 4.508e-02
[PASS] the rotation by 72 degrees about a 5-fold axis and the cyclic permutation generate a group of order 60
[PASS] with the central inversion: I_h of order 120
[PASS] rotation matrices have entries in (1/2) Z[phi], not all in Z[phi]
[PASS] I_h preserves the vertex module M
[PASS] I_h does not preserve Z[phi]^3
[PASS] I_h permutes the 12 cut planes
  exact I_h-orbits of the 165 sign vectors: 7, sizes [1, 12, 12, 20, 30, 30, 60]
[PASS] 7 orbits of sizes 1, 12, 12, 20, 30, 30, 60
[PASS] each exact orbit is exactly one class of the slicer (float invariants)
    weight 0: orbit size   1, slicer class 1, canonical sign vector ------------
    weight 1: orbit size  12, slicer class 5, canonical sign vector -----------+
    weight 2: orbit size  30, slicer class 6, canonical sign vector ----------++
    weight 3: orbit size  20, slicer class 2, canonical sign vector ---------+++
    weight 4: orbit size  30, slicer class 0, canonical sign vector --------++++
    weight 5: orbit size  60, slicer class 4, canonical sign vector -------+++++
    weight 6: orbit size  12, slicer class 3, canonical sign vector -----+++++-+
[PASS] Hamming weights 0,1,2,3 are the four stellation layers (1, 12, 30, 20 pieces) and 4,5,6 the gap (30, 60, 12)

== G. membership without cutting: sign vectors of random points of D ==
  400000 uniform points of D (936797 tried): 165 distinct sign vectors, all among the 165: True
[PASS] every sampled sign vector is a realised piece and all 165 are hit

== H. Z[phi]-modules in dimension 3 ==
  [Z[phi]^3 : M] = 4   (M = Z[phi]-span of the 20 vertices)
  [Z[phi]^3 : <n_F>] = 4   (span of the 12 face normals)
  [Z[phi]^3 : <2 m_e>] = 16   (span of the doubled edge midpoints)
[PASS] [Z[phi]^3 : M] = 4
[PASS] M is also the Z[phi]-span of the 12 face normals
[PASS] unit vectors are not in M
[PASS] edge midpoints: 2 m_e in M, m_e not all in M
[PASS] (2+phi)/5 has norm 1/5: face centres (2+phi)/5 n_F are not in Z[phi]^3
[PASS] phi^2 M = M (unit), sqrt5 M is a sublattice of index 5^3 = 125

== I. the 120-cell: Z[phi]^4, the vertex module M_V, the icosian ring ==
[PASS] 600 vertices and 120 facet normals, all in Z[phi]^4
[PASS] |n|^2 = 4; every vertex on exactly 4 facets; every facet carries 20 vertices; n.v <= 2phi^2 always
[PASS] with the EVEN permutations of (phi,1,phi^-1,0) the hyperplanes do not support facets   (96 of 96 fail)
  vertex heights along (1,0,0,0): [-1-phi, 1-2phi, -2, -1phi, -1, 1-phi, -2+phi, 0, 2-phi, -1+phi, 1, phi, 2, -1+2phi, 1+phi]
[PASS] heights 0, +-phi^-2, +-phi^-1, +-1, +-phi, +-2, +-sqrt5, +-phi^2
  layer sizes (height: count): {'-1-phi': 20, '1-2phi': 20, '-2': 30, '-1phi': 60, '-1': 60, '1-phi': 60, '-2+phi': 20, '0': 60, '2-phi': 20, '-1+phi': 60, '1': 60, 'phi': 60, '2': 30, '-1+2phi': 20, '1+phi': 20}
[PASS] layers 20,20,30,60,60,60,20,60,20,60,60,60,30,20,20
[PASS] height / inradius: phi^-2 / phi^2 = phi^-4 (star layer), sqrt5 / phi^2 = 1 - phi^-4 (hereditary layer)
[PASS] the 120 vectors n/2 are unit quaternions closed under multiplication (binary icosahedral group 2I)
  [Z[phi]^4 : M_V] = 16,  [I : Z[phi]^4] = 16,  [I : M_V] = 256
[PASS] [Z[phi]^4 : M_V] = 16, [I : Z[phi]^4] = 16, [I : M_V] = 256
[PASS] 60 mirrors (pairs +-n)
[PASS] every reflection of H4 preserves M_V
[PASS] every reflection of H4 preserves the icosian ring
[PASS] not every reflection preserves Z[phi]^4   (e.g. (-1, 0, 0, 0))
[PASS] phi^-4 M_V = M_V (unit), (sqrt5/phi^2) M_V has index 5^4 = 625 in M_V
[PASS] lambda = (2+phi)/5 = phi/sqrt5 is totally positive with norm 1/5
  Gram determinant of (I, Tr(x.y)) = 625/256;  of (I, 2 Tr(lambda x.y)) = 1, diagonal [4, 2, 4, 2, 2, 4, 2, 4]
[PASS] (I, 2 Tr(lambda x.y)) is even unimodular of rank 8: E8
[PASS] its 240 roots (norm 2) are the 120 unit icosians n/2 and their 120 multiples by phi^-1 (two concentric 600-cells)

== J. hereditary cut: cut hyperplanes n.x = -2+4phi = 4.472136 (facets n.x = 2phi^2) ==
  representative cells (one per orbit): 8; vertices in total: 712
[PASS] J. hereditary cut: every representative vertex re-derived exactly from four of its hyperplanes and checked on all of them   (0 failures)
  distinct exact vertices among the representatives: 665
  hyperplanes through a vertex (4 = simple): {4: 678, 6: 30, 8: 4}
  rational denominators of the coordinates: {1: 665}
  in Z[phi]^4: 665/665;  in the icosian ring: 665/665;  in the vertex module M_V: 665/665
  1 * (every vertex) lies in Z[phi]^4
  1 * (every vertex) lies in the icosian ring
[PASS] hereditary cut: all representative vertices in M_V (hence, by H4-invariance of M_V, every vertex of every cell)

== K. star cut: cut hyperplanes n.x = 4-2phi = 0.763932 (facets n.x = 2phi^2) ==
  representative cells (one per orbit): 338; vertices in total: 3461
[PASS] K. star cut: every representative vertex re-derived exactly from four of its hyperplanes and checked on all of them   (0 failures)
  distinct exact vertices among the representatives: 1748
  hyperplanes through a vertex (4 = simple): {4: 1544, 5: 569, 6: 751, 7: 82, 8: 55, 9: 50, 10: 188, 12: 189, 20: 33}
  rational denominators of the coordinates: {1: 1362, 2: 198, 3: 23, 5: 165}
  in Z[phi]^4: 1362/1748;  in the icosian ring: 1560/1748;  in the vertex module M_V: 1295/1748
  30 * (every vertex) lies in Z[phi]^4
  15 * (every vertex) lies in the icosian ring
[PASS] star cut: denominators 1, 2, 3, 5 occur; 30 x in Z[phi]^4 for every vertex x
[PASS] star cut: not all vertices in M_V   (1295 of 1748 in M_V, 1362 in Z[phi]^4, 1560 in the icosian ring)

== summary ==
0 failures
\end{lstlisting}
}

\end{document}